\documentclass[12pt, english, a4paper]{amsart}

\usepackage[english]{babel}

\usepackage[normalem]{ulem}

\usepackage[utf8]{inputenc}

\usepackage{geometry}
\usepackage{setspace}

\usepackage[style=alphabetic, bibstyle=alphabetic, citestyle=alphabetic, backend=biber, doi=false, isbn=false, maxnames=99, backref=true]{biblatex} 

\usepackage[autostyle,italian=guillemets]{csquotes} 

\usepackage{indentfirst}
\usepackage[T1]{fontenc}

\usepackage[title]{appendix} 
\usepackage{booktabs} 
\usepackage{caption} 
\usepackage{threeparttable} 
\usepackage{listings}
\usepackage{subcaption}
\usepackage{comment}

\usepackage{amsmath}
\usepackage{amssymb}
\usepackage{amsthm}
\usepackage{amsfonts}
\usepackage{mathrsfs}
\usepackage{mathtools} 
\usepackage{siunitx} 
\usepackage{tikz-cd}

\usepackage{graphicx}
\usepackage[colorlinks=true, allcolors=blue]{hyperref}
\usepackage{url}
\usepackage{color}
\usepackage[dvipsnames]{xcolor}

\usepackage{algorithm}
\usepackage{algorithmicx}
\usepackage{algpseudocode}
\usepackage{chngcntr}
\usepackage{lipsum}
\usepackage{diagbox}
\usepackage{paralist}
\usepackage{float}

\newtheorem{The}{Theorem}
\newtheorem{Pro}[The]{Proposition}
\newtheorem{Lem}[The]{Lemma}
\newtheorem{Cor}[The]{Corollary}

\newtheorem{Que}[The]{Question}
\newtheorem{Not}[The]{Notation}

\theoremstyle{definition}
\newtheorem{Def}[The]{Definition}
\newtheorem{Rem}[The]{Remark}

\numberwithin{The}{section}
\numberwithin{equation}{section}

\newcommand{\C}{\mathbb{C}}
\newcommand{\R}{\mathbb{R}}
\newcommand{\Q}{\mathbb{Q}}
\newcommand{\Z}{\mathbb{Z}}

\newcommand{\Qbar}{\overline{\Q}}

\DeclareMathOperator{\mfrak}{\mathfrak{m}}

\DeclareMathOperator{\Cfrak}{\mathfrak{C}}
\DeclareMathOperator{\Jfrak}{\mathfrak{J}}

\DeclareMathOperator{\Acal}{\mathcal{A}}
\DeclareMathOperator{\Bcal}{\mathcal{B}}
\DeclareMathOperator{\Ccal}{\mathcal{C}}

\DeclareMathOperator{\Hcal}{\mathcal{H}}

\DeclareMathOperator{\Jcal}{\mathcal{J}}

\DeclareMathOperator{\Ocal}{\mathcal{O}}

\DeclareMathOperator{\Abb}{\mathbb{A}}

\DeclareMathOperator{\Fbb}{\mathbb{F}}
\DeclareMathOperator{\Gbb}{\mathbb{G}}

\DeclareMathOperator{\abf}{\mathbf{a}}

\DeclareMathOperator{\Cbf}{\mathbf{C}}
\DeclareMathOperator{\Jbf}{\mathbf{J}}

\DeclareMathOperator{\Jac}{Jac}

\DeclareMathOperator{\Aut}{Aut}
\DeclareMathOperator{\End}{End}
\DeclareMathOperator{\SL}{SL}
\DeclareMathOperator{\GL}{GL}
\DeclareMathOperator{\GCD}{GCD}

\DeclareMathOperator{\Gal}{Gal}
\DeclareMathOperator{\M}{M}
\DeclareMathOperator{\Stab}{Stab}

\usepackage{float}   
\usepackage{caption} 



\title{Automorphism groups of curves with simple Jacobians}

\author{Luca Mauri}
\address{\small University of Pisa, Department of Mathematics, Largo Bruno Pontecorvo 5, Pisa}
\email{luca.mauri@phd.unipi.it}

\date{}

\begin{document}




\begin{abstract}
    We classify all automorphism groups of smooth, projective, connected curves over an algebraically closed field of characteristic $0$ with simple Jacobians. For every nontrivial group in this classification, we construct a family of curves with that automorphism group and prove that a density-one set of members has simple Jacobian. These families yield infinitely many new examples of unlikely intersections in the moduli spaces of principally polarized abelian varieties.
\end{abstract}


\maketitle

\setcounter{tocdepth}{1}

\tableofcontents


    \section{Introduction}
    \label{sec: Introduction}
            
        Let $k$ be an algebraically closed field of characteristic $0$ and let $\Acal$ be a principally polarized abelian variety over $k$. Denote by $\End( \Acal )$ the endomorphism ring of $\Acal$ and by $\End^{0}( \Acal ) \coloneqq \End (\Acal) \otimes_{\Z} \Q$ its endomorphism algebra. For simple abelian varieties, Albert's classification \cite{AlbertClassification, AlbertClassificationI, AlbertClassificationII} gives strong restrictions on the possible $\Q$-algebras $\End^{0}(\Acal)$, and work of Shimura \cite{ShimuraClassification} shows that, over $\C$, essentially all algebras allowed by these restrictions do occur, up to a small list of exceptional cases. Shimura's result can then be extended to $\Qbar$ by specialization. Since every abelian variety is isogenous to a product of simple abelian varieties, this gives a rather complete picture of endomorphism algebras of abelian varieties up to isogeny.
    
        The situation is very different for Jacobians of curves (we will always assume our curves to be smooth, projective, and connected, unless otherwise specified). If $\Ccal / k$ is a curve of genus $\ge 2$ and $\Jcal$ is its Jacobian, $\End^{0}(\Jcal)$ is constrained not only by the theory of abelian varieties, but also by the fact that $\Jcal$ comes from a curve. No classification is known of the endomorphism algebras of simple Jacobians. Moreover, even such a classification would not immediately yield a classification for all Jacobians, since a Jacobian need not be isogenous to a product of simple Jacobians.
    
        The natural action of $\Aut(\Ccal)$ on $\Jcal$ gives an injection of $\Aut(\Ccal)$ into the group of finite-order units of $\End(\Jcal)$; see \cite[Theorem~12.1, \S~III]{milneAV}. This leads to the following question.
    
        \begin{Que}
        \label{Main question; sec: Introduction}
            Let $k$ be an algebraically closed field of characteristic $0$. Classify the finite groups that occur as automorphism groups of curves over $k$ of genus $\ge 2$ whose Jacobians are simple.
        \end{Que}
    
        Our main result gives a complete answer to this problem.

        \begin{The}
        \label{Main theorem; sec: Introduction}
            Let $k$ be an algebraically closed field of characteristic $0$, and let $\Ccal / k$ be a curve of genus $\ge 2$ with simple Jacobian. Then $\Aut ( \Ccal )$ is isomorphic to one of the following groups:
            \begin{enumerate}[(i)]
                \item the cyclic group $C_{p^{n}}$ for some prime $p$ and some positive integer $n$;
                \item the cyclic group $C_{pq}$ for some distinct primes $p, q$ such that $pq \neq 6$;
                \item the generalized quaternion group $G_{2^{n}}$ (cf.~Theorem~\ref{the: classification Aut of hyperelliptic curves; sec: Preliminaries}) for some positive integer $n$;
                \item the trivial group.
            \end{enumerate}
            Moreover, every group in the list is the automorphism group of some curve over $k$ of genus $\ge 2$ with simple Jacobian. 
        \end{The}

        \begin{Rem}
        \label{rem: Main theorem; sec: Introduction}
            The action of $C_{p^{n}}$ (resp.~$G_{2^{n}}$) on a curve $\Ccal / k$ of genus $g \ge 2$ with simple Jacobian provides an embedding of $\Q ( \zeta_{p^{n}} )$ (resp.~the quaternion algebra $\left( \frac{ ( \zeta_{2^{n+1}} - \zeta_{2^{n+1}}^{-1} )^{2}, \; -1 }{ \Q( \zeta_{2^{n+1}} + \zeta_{2^{n+1}}^{-1} ) } \right)$) into $\End^{0} \bigl( \Jac ( \Ccal ) \bigr)$, which implies the constraint $\varphi ( p^{n} ) \mid 2g$ (resp.~$2^{n+1} \mid 2g$). We show that for every such genus $g$, the group $C_{p^{n}}$ (resp.~$G_{2^{n}}$) is the automorphism group of some curve over $k$ of genus $g$ with simple Jacobian, unless $p = 2$ and $g \in \bigl\{ \varphi ( 2^{n} ) / 2, \varphi ( 2^{n} ) \bigr\}$ (resp.~$g = 2^{n}$).
        \end{Rem}

        The proof of this theorem consists of two rather different parts. To prove the first statement, we analyze the action of $\Aut ( \Ccal )$ on the $k$-vector space of regular differential forms on $\Ccal$ and show that simplicity of the Jacobian forces this representation to be fixed-point-free. We then use the classification of finite groups admitting such representations, due to Wolf \cite[Theorems~6.1.11 and~6.3.1]{WolfBookFixRep}, together with strong constraints coming from the geometry of superelliptic curves, to exclude all groups except those appearing in Theorem~\ref{Main theorem; sec: Introduction}. 
    
        The second part of the proof is constructive (and significantly harder). Here, the main difficulty is to produce (families of) curves with prescribed automorphism groups and simple Jacobians, as there are few methods to prove the simplicity of a Jacobian. 
        We now describe the main results concerning the constructive part.
        They use the following notion of density, which we adopt throughout the paper:
        \begin{Def}
        \label{def: 100 percent meaning; sec: Introduction}
            Let $U \subseteq \Abb_{\Qbar}^{L}$ be a non-empty open subscheme and let $K$ be a number field. Denote by $H$ the absolute multiplicative Weil height on $K^{L}$, see Definition~\ref{def: height; subsec: specialization methods; sec: Preliminaries}. We say that a certain property $\mathcal{P}$ holds for $100 \%$ of the points in $U ( \Qbar ) \cap K^{L}$ if
            \begin{equation*}
                \lim_{ B \to \infty } \frac{ \#  \bigl\{ P \in U ( \Qbar ) \cap K^{L} \colon H ( P ) \le B \text{ and } \mathcal{P} \text{ holds for } P \bigr\} }{ \# \bigl\{ P \in U ( \Qbar ) \cap K^{L} \colon H ( P ) \le B \bigr\} } = 1 .
            \end{equation*}
        \end{Def}
        
        \begin{The}
        \label{Main theorem Cpn; sec: Introduction}
            Let $p$ be a prime, and let $n > 1$ and $L \ge 1$ be integers. If $p = 2$, assume in addition $L \ge 3$. For $\abf = ( a_{1}, \dots, a_{L} ) \in \Abb^{L}_{\Qbar} ( \Qbar )$, set $f_{\abf} ( x ) \coloneqq x \cdot \prod_{ l = 1 }^{L} \bigl( x^{ p^{ n - 1 } } - a_{l} \bigr)$, and let $U \subset \Abb^{L}_{\Qbar}$ be the non-empty open subscheme defined by $\operatorname{disc}_{x} ( f_{\abf} ) \neq 0$. For every $\abf \in U ( \Qbar )$, let $\Cfrak_{\abf} / \Qbar$ denote the smooth projective model of the affine curve 
            \begin{equation*}
                y^{p} = f_{\abf} ( x ) = x \cdot \prod_{ l = 1 }^{L} \bigl( x^{ p^{ n - 1 } } - a_{l} \bigr) .
            \end{equation*}
            For every number field $K$ and for $100 \%$ of the points $\abf \in U ( \Qbar ) \cap K^{L}$, one has $\Aut_{\Qbar} ( \Cfrak_{\abf} ) \cong C_{ p^{n} }$, and the Jacobian $\Jfrak_{\abf} \coloneqq \Jac ( \Cfrak_{\abf} )$ is (geometrically) simple, with $\End_{\Qbar} (\Jfrak_{\abf} ) \cong \Z [ \zeta_{ p^{n} } ]$. Moreover, if $L\ge 2$, the curves $\Cfrak_{\abf}$ obtained in this way represent infinitely many distinct isomorphism classes over $\Qbar$, and their Jacobians represent infinitely many distinct isogeny classes over $\Qbar$.
        \end{The}

        \begin{The}
        \label{main theorem G2n; sec: introduction}
            Let $n \ge 1$ and $L \ge 1$ be integers. For $\abf = ( a_{1}, \dots, a_{L} ) \in \Gbb_{m, \Qbar}^{L} ( \Qbar )$, set $f_{\abf} ( x ) \coloneqq x \cdot \bigl( x^{2^{n+1}} - 1 \bigr) \cdot \prod_{l = 1}^{L} \bigl[ \bigl( x^{2^{n}} - a_{l} \bigr) \cdot \bigl( x^{2^{n}} - a^{-1}_{l} \bigr) \bigr]$, and let $U \subset \Gbb_{m, \Qbar}^{L}$ be the non-empty open subscheme defined by $\operatorname{disc}_{x} ( f_{\abf} ) \neq 0$. For every $\abf \in U ( \Qbar )$, let $\Cfrak_{\abf} / \Qbar$ denote the smooth projective model of the affine curve 
            \begin{equation*}
                y^{2} = f_{\abf} ( x ) = x \cdot \bigl( x^{2^{n+1}} - 1 \bigr) \cdot \prod_{l = 1}^{L} \bigl[ \bigl( x^{2^{n}} - a_{l} \bigr) \cdot \bigl( x^{2^{n}} - a^{-1}_{l} \bigr) \bigr] .
            \end{equation*}
            For every number field $K$ and for $100 \%$ of the points $\abf \in U ( \Qbar ) \cap K^{L}$, one has $\Aut_{\Qbar} ( \Cfrak_{\abf} ) \cong G_{ 2^{n} }$, and the Jacobian $\Jfrak_{\abf} \coloneqq \Jac ( \Cfrak_{\abf} )$ is (geometrically) simple, with endomorphism algebra $\End_{\Qbar}^{0} (\Jfrak_{\abf} ) \cong \left( \frac{ ( \zeta_{2^{n+1}} - \zeta_{2^{n+1}}^{-1} )^{2}, \; -1 }{ \Q( \zeta_{2^{n+1}} + \zeta_{2^{n+1}}^{-1} ) } \right)$. Moreover, if $L\ge 2$, the curves $\Cfrak_{\abf}$ obtained in this way represent infinitely many distinct isomorphism classes over $\Qbar$, and their Jacobians represent infinitely many distinct isogeny classes over $\Qbar$.
        \end{The}

        It has been observed before that the Torelli locus intersects loci of abelian varieties with large endomorphism rings in subvarieties that are often larger than naive dimension counts predict. Many examples of this phenomenon are known: work of Tautz--Top--Verberkmoes \cite{TautzExamples}, Mestre \cite{MestreExamples}, Hashimoto--Murabayashi \cite{HashimotoExamples}, de Jong--Noot \cite{JongExamples}, Ekedahl--Serre \cite{EkedahlExamples}, Shimada \cite{ShimadaExamples}, Ellenberg \cite{EllenbergExamples}, Naranjo--Ortega--Pirola--Spelta \cite{SpeltaExamples}, and Cantoral-Farfán--Lombardo--Voight \cite{cantoral2023monodromy} exhibit families of curves whose Jacobians carry unusually large endomorphism algebras. Theorems~\ref{Main theorem Cpn; sec: Introduction} and~\ref{main theorem G2n; sec: introduction} provide infinitely many new families of this kind. For example, fix a positive integer $L$ and let $g = \varphi ( p^{n} ) \cdot L / 2$. The image of the map to the moduli space induced by the family in Theorem~\ref{Main theorem Cpn; sec: Introduction} has dimension $L - 1$ (Theorem~\ref{the: there are infinitely many classes; subsec: Preliminaries; sec: Cyclic groups of order p^n}). On the other hand, the dimension of the locus of abelian varieties of dimension $g$ whose endomorphism ring contains a subring isomorphic to $\Z [ \zeta_{p^{n}} ]$ is at most $g^{2} / \bigl( 2 \cdot \varphi ( p^{n} ) \bigr)$ (\cite[\S~$9.6$, p.~266]{BirkenhakeLangeAbVar}). Since the dimension of the moduli space of principally polarized abelian varieties of dimension $g$ is $g ( g + 1 ) / 2$ (\cite[\S~8.1 and \S~8.2]{BirkenhakeLangeAbVar}) and the dimension of the Torelli locus inside this moduli space is $3g - 3$ (\cite[\S~2]{HarrisModuliCurves}), it follows that, if $g \ge 9$, the intersection of the Torelli locus with the locus of abelian varieties of dimension $g$ whose endomorphism ring contains a subring isomorphic to $\Z [ \zeta_{p^{n}} ]$ has negative expected dimension.

        \subsection{Relation to the existing literature}
        \label{subsec: Relation to the existing literature; sec: introduction}

            We review previous results in the literature regarding Question~\ref{Main question; sec: Introduction}. To show that a group can be realized as the automorphism group of a curve over $k$ of genus $\ge 2$ with simple Jacobian, it is sufficient to consider the case $k = \Qbar$, since $\Qbar$ embeds into any algebraically closed field of characteristic $0$ (here we use the invariance of automorphism groups, endomorphism algebras, and simplicity of the Jacobian under extensions of algebraically closed fields of characteristic zero).

            We believe that the fact that the trivial group appears as the automorphism group of a curve over $\Qbar$ of genus $\ge 2$ with simple Jacobian is well-known. As we did not find a direct reference, we prove this fact in Lemma~\ref{lem: trivial group is realizable; sec: Preliminaries}.

            In a series of papers \cite{Zarhin2000, Zarhin2005, Zarhin2009, Zarhin2018, Zarhin2026}, Zarhin proves that $C_{p}$ appears as the automorphism group of a curve over $\Qbar$ of genus $\ge 2$ with simple Jacobian for every prime $p$. See Theorem~\ref{the: Zarhin; sec: Cyclic groups of order p^n} for a precise statement.

            The classical quaternion group $G_{2}$ appears as the automorphism group of infinitely many curves over $\Qbar$ with simple Jacobians of every even genus $g \ge 4$, as shown by Cantoral-Farfán--Lombardo--Voight in \cite{cantoral2023monodromy}.

            Finally, for all distinct primes $p, q$ such that $pq \neq 6$, the Catalan curve $y^{p} = x^{q} - 1$ over $\Qbar$ has automorphism group isomorphic to $C_{pq}$ and simple Jacobian, see \cite[\S~$2$]{HazamaCatalan} or \cite[\S~$3.2$]{GoodsonCatalan}. We complete the picture by showing that, up to isomorphism, this is the \textit{only} example with this automorphism group and simple Jacobian (Theorem~\ref{the: main; sec: Cyclic groups of order pq}).
        
        \subsection{Methods}
        \label{subsec: Methods; sec: Introduction}

            As anticipated, the proof of the second statement of Theorem~\ref{Main theorem; sec: Introduction} is considerably more difficult than that of the first, as there are few methods to prove the simplicity of a Jacobian. 
            
            One of the available techniques is Zarhin's method. This was introduced in \cite{Zarhin2000} to study the absolute simplicity of Jacobians of hyperelliptic curves over number fields, and the key idea is to study the simultaneous action of the absolute Galois group of the number field and of the endomorphism algebra of the Jacobian on the $2$-torsion group. In subsequent years, Zarhin refined this technique to analyze the Jacobians of superelliptic curves of $p^{n}$-degree by introducing the concepts of very simple, strongly simple, and central simple representations (see \cite[Definition~4.1]{Zarhin2026}), and by replacing the $2$-torsion group with the $( 1 - \zeta_{p^{n}} )$-torsion group. When attempting to apply Zarhin's method to the Jacobians of the curves appearing in Theorem~\ref{Main theorem Cpn; sec: Introduction}, one encounters several difficulties, in particular regarding the application of \cite[Lemma~2.3]{Zarhin2026} (cf. Remark~\ref{rem: impossibility of applying Zarhin method; sec: Cyclic groups of order p^n}). 

            Another technique was introduced by Cantoral-Farfán, Lombardo, and Voight in \cite{cantoral2023monodromy} to study the Jacobians of the curves appearing in Theorem~\ref{main theorem G2n; sec: introduction} when $n = 1$. They used an induction on the genus. The base case consisted of studying two explicit curves, of genera $4$ and $6$, which was carried out directly using techniques based on reduction modulo primes (see \cite{CostaComputationEndJacobians, DavideEndJacobiansFiniteField}). The inductive step was established via specialization techniques and degenerations to reducible curves with components of lower genus (see Section~\ref{subsec: specialization methods; sec: Preliminaries}).

            This second technique can be applied to both families of curves appearing in Theorems~\ref{Main theorem Cpn; sec: Introduction} and \ref{main theorem G2n; sec: introduction}. There are, however, two significant differences. In the setting of \cite{cantoral2023monodromy}, the base case consisted of two curves, while in our setting it consists of infinitely many curves. We overcome this difficulty using the theory of abelian varieties with complex multiplication (CM) and, for the case $p = 2$ of Theorem~\ref{Main theorem Cpn; sec: Introduction}, a careful study of particular curves. CM theory provides a criterion to prove simplicity of Jacobians in terms of combinatorial data called CM types and, in our setting, the CM type can be read off easily from the explicit equation of the curves. The second difference concerns the inductive step. In the setting of \cite{cantoral2023monodromy}, it was sufficient to keep track of how the endomorphisms specialize to reducible members of the family, whereas in our setting we must also keep track of how the torsion specializes.

        \subsection{Acknowledgments}
        \label{subsec: Acknowledgements; sec: Introduction}

            The author thanks Davide Lombardo for posing the question that started this work and for many useful discussions. The author is supported by the ``National Group for Algebraic and Geometric Structures, and their Application'' (GNSAGA--INdAM).

    \section{Preliminaries}
    \label{sec: Preliminaries}

        Let $k$ be an algebraically closed field of characteristic $0$. We assume all curves to be smooth, projective, and connected, unless otherwise specified.

        The following lemma, though easy, is the key to all subsequent arguments.

        \begin{Lem}
        \label{lem: necessary condition simplicity; sec: Preliminaries}
            Let $\Ccal / k$ be a curve with simple Jacobian, and let $f \colon \Ccal \to \Ccal_{1}$ be a nonconstant morphism of curves. If $g ( \Ccal_{1} ) < g ( \Ccal )$, then $\Ccal_{1} \cong \mathbb{P}_{1, k}$.
        \end{Lem}
        \begin{proof}
            The morphism $f$ is automatically finite and surjective, and it induces a homomorphism of abelian varieties $f^{*} \colon \Jac ( \Ccal_{1} ) \to \Jac ( \Ccal )$ with finite kernel. Indeed, the morphism $f$ induces another homomorphism of abelian varieties $f_{*} \colon \Jac ( \Ccal ) \to \Jac ( \Ccal_{1} )$, called the trace map, such that $f_{*} \circ f^{*}$ is the multiplication by $\deg ( f )$ on $\Jac ( \Ccal_{1} )$. Hence, the kernel of $f^{*}$ is contained in the kernel of the multiplication by $\deg ( f )$, which is finite.

            Since $g ( \Ccal_{1} ) < g ( \Ccal )$, the homomorphism $f^{*}$ cannot be surjective. Assume, for contradiction, that $\Ccal_{1} \not \cong \mathbb{P}_{1, k}$. Since $k$ is algebraically closed, it follows that $g ( \Ccal_{1} ) > 0$. Hence, the image of $f^{*}$ is a proper abelian subvariety of $\Jac ( \Ccal )$, since the kernel of $f^{*}$ is finite, which contradicts the simplicity of $\Jac ( \Ccal )$.
        \end{proof}

        This is the first consequence of Lemma~\ref{lem: necessary condition simplicity; sec: Preliminaries}.
        
        \begin{The}
        \label{the: necessary condition simplicity; sec: Preliminaries}
            Let $\Ccal / k$ be a curve of genus $\ge 2$ with simple Jacobian. For every $\sigma \in \Aut ( \Ccal ) \setminus \{ \operatorname{id} \}$, the quotient curve $\Ccal / \langle \sigma \rangle$ is isomorphic to $\mathbb{P}_{1, k}$.
        \end{The}
        \begin{proof}
            Consider a nontrivial element $\sigma$ of $\Aut ( \Ccal )$. Since $g (\Ccal) \ge 2$, $\sigma$ has finite order and the quotient scheme $\Ccal / \langle \sigma \rangle$ exists. Since $\Ccal$ is a curve, $\Ccal / \langle \sigma \rangle$ is a curve. Moreover, since $\sigma$ is not the identity, from the Riemann--Hurwitz formula we deduce $g \bigl( \Ccal / \langle \sigma \rangle \bigr) < g ( \Ccal )$. Apply Lemma~\ref{lem: necessary condition simplicity; sec: Preliminaries} to the nonconstant morphism $f_{\sigma} \colon \Ccal \to \Ccal / \langle \sigma \rangle$.
        \end{proof}

        \subsection{Curves}
        \label{subsec: Curves; sec: Preliminaries}

            In Proposition~\ref{pro: a curve with simple Jacobian is superelliptic; sec: Preliminaries}, we show that any curve over $k$ of genus $\ge 2$ with simple Jacobian and nontrivial automorphism group is superelliptic. To this end, we now review some definitions and properties, starting with hyperelliptic curves.

            \begin{Def}[hyperelliptic curve]
            \label{def: hyperelliptic curve; sec: Preliminaries}
                A hyperelliptic curve is a curve $\Ccal / k$ which admits a hyperelliptic involution, that is, an element $\sigma \in \Aut ( \Ccal )$ of order $2$ such that $\Ccal / \langle \sigma \rangle \cong \mathbb{P}_{1, k}$.
            \end{Def}

            \begin{Pro}
            \label{pro: unique element of order 2 in AutC if C is simple; sec: Preliminaries}
                Let $\Ccal / k$ be a curve of genus $\ge 2$ with simple Jacobian. There exists at most one element of order $2$ in $\Aut ( \Ccal )$.
            \end{Pro}
            \begin{proof}
                By Theorem~\ref{the: necessary condition simplicity; sec: Preliminaries}, two distinct elements of order $2$ in $\Aut ( \Ccal )$ would be two distinct hyperelliptic involutions of $\Ccal$, which contradicts \cite[\S~$7.4.3$, Proposition~$4.29$]{QingAlgGeom}.
            \end{proof}

            Hyperelliptic curves have been studied extensively since the nineteenth century, resulting in a well-developed theory. By \cite[\S~$7.4.3$, Corollary~$4.31$]{QingAlgGeom}, the hyperelliptic involution $\iota$ of a hyperelliptic curve $\Ccal$ is central in $\Aut ( \Ccal )$ and $\Aut ( \Ccal ) / \langle \iota \rangle$ is a finite subgroup of $\Aut \bigl( \mathbb{P}_{1, k} \bigr)$.
            
            The finite subgroups of $\Aut \bigl( \mathbb{P}_{1, k} \bigr)$ are, up to isomorphism, the cyclic groups $C_{n}$, the dihedral groups $D_{n}$, the alternating groups $A_{4}$ and $A_{5}$, and the symmetric group $S_{4}$ (see \cite[\S~$2.5$, Proposition~$16$]{SerreSubgroupsPGL2}). Hence, $\Aut ( \Ccal )$ is a central $C_{2}$-extension of one of these groups. Because of ramification constraints, some extensions cannot occur as the automorphism group of a hyperelliptic curve. The next theorem classifies precisely the extensions that do occur.
            
            \begin{The}
            \label{the: classification Aut of hyperelliptic curves; sec: Preliminaries}
                The automorphism group of a hyperelliptic curve $\Ccal$ of genus $g \ge 2$ is isomorphic to one of the following groups:
                \begin{center}
                \small
                    \begin{tabular}{|c|c|c|c|c|c|}
                        \hline
                        $\Aut ( \Ccal )$ & $\Aut ( \Ccal ) / \langle \iota \rangle$ & Constraints & $\Aut ( \Ccal )$ & $\Aut ( \Ccal ) / \langle \iota \rangle$ & Constraints \\
                        \hline
                        $C_{2} \times C_{n}$ & $C_{n}$ & $n \mid ( 2g + 2 )$, $g \neq n - 1$ & $D_{2n}$ & $D_{n}$ & $n \mid g$ \\
                        \hline
                        $C_{2n}$ & $C_{n}$ & $n \mid ( 2g + 1 )$ or $n \mid 2g$ & $H_{n}$ & $D_{n}$ & $n \mid ( g + 1 )$, $n < g + 1$ \\
                        \hline
                        $C_{2} \times D_{n}$ & $D_{n}$ & $n \mid ( g + 1 )$ & $U_{n}$ & $D_{n}$ & $2n \mid ( 2g - n )$ \\
                        \hline
                        $V_{n}$ & $D_{n}$ & $2n \mid ( 2g + 2 - n )$ & $G_{n}$ & $D_{n}$ & $n \mid g$, $n < g$ \\
                        \hline
                        $C_{2} \times A_{4}$ & $A_{4}$ & $g \ge 5$ & $W_{2}$ & $S_{4}$ & $g \ge 5$ \\
                        \hline
                        $\SL ( 2, \Fbb_{3} )$ & $A_{4}$ & $g \ge 4$ & $W_{3}$ & $S_{4}$ & $g \ge 8$ \\
                        \hline
                        $C_{2} \times S_{4}$ & $S_{4}$ & $g \ge 3$ & $C_{2} \times A_{5}$ & $A_{5}$ & $g \ge 5$ \\
                        \hline
                        $\GL ( 2, \Fbb_{3} )$ & $S_{4}$ & $g \ge 2$ & $\SL ( 2, \Fbb_{5} )$ & $A_{5}$ & $g \ge 14$ \\
                        \hline
                    \end{tabular}
                \end{center}
                where
                \begin{align*}
                    V_{n} &\coloneqq \bigl\langle x, y \,\lvert\, x^{4}, y^{n}, ( x \cdot y )^{2}, ( x^{-1} \cdot y )^{2} \bigr\rangle , & H_{n} &\coloneqq \bigl\langle x, y \,\lvert\, x^{4}, y^{2} \cdot x^{2}, ( x \cdot y )^{n} \bigr\rangle , \\
                    U_{n} &\coloneqq \bigl\langle x, y \,\lvert\, x^{2}, y^{2n}, x \cdot y \cdot x \cdot y^{n+1} \bigr\rangle , & G_{n} &\coloneqq \bigl\langle x, y \,\lvert\, x^{2} \cdot y^{n}, y^{2n}, x^{-1} \cdot  y \cdot x \cdot y \bigr\rangle , \\
                    W_{2} &\coloneqq \bigl\langle x, y \,\lvert\, x^{4}, y^{3}, y \cdot x^{2} \cdot y^{-1} \cdot x^{2}, ( x \cdot y )^{4} \bigr\rangle , & W_{3} &\coloneqq \bigl\langle x, y \,\lvert\, x^{4}, y^{3}, x^{2} \cdot ( x \cdot y )^{4}, ( x \cdot y )^{8} \bigr\rangle .
                \end{align*}
            \end{The}
            \begin{proof}
                The table lists only the constraints relevant to our purposes. For the complete set of necessary and sufficient conditions for a group to appear as the automorphism group of some hyperelliptic curve of genus $\ge 2$, we refer the reader to \cite[Theorem~$3.1$]{BujalanceAutHyp} or \cite[Theorems~$2.1$ and~$2.2$]{ShaskaAutHyp}.
            \end{proof}

            We continue by recalling the definition and properties of superelliptic curves.
            
            \begin{Def}[superelliptic curve]
            \label{def: superelliptic curve; sec: Preliminaries}
                A superelliptic curve is a curve whose function field is the field of fractions of $k [ x, y ] / \bigl( y^{n} - f(x) \bigr)$ for some integer $n \ge 2$ and some $f \in k [ x ] $ such that $y^{n} - f(x)$ is irreducible.
            \end{Def}

            Superelliptic curves generalize hyperelliptic curves, which correspond to the case $n = 2$ with $f$ being a square-free polynomial.
            
            \begin{Not}
            \label{not: superelliptic curve; sec: Preliminaries}
                We say that a superelliptic curve is given by $y^{n} = f(x)$ if its function field is the field of fractions of $k [ x, y ] / \bigl( y^{n} - f(x) \bigr)$.
            \end{Not}

            An interesting class of superelliptic curves is given by Catalan curves.
            
            \begin{Def}[Catalan curve]
            \label{def: Catalan curve; sec: Preliminaries}
                A Catalan curve is a superelliptic curve given by $y^{p} = x^{q} - 1$ for some distinct primes $p, q$.
            \end{Def}
            
            \begin{The}
            \label{the: properties of Catalan curve; sec: Preliminaries}
                The Catalan curve given by $y^{p} = x^{q} - 1$, where $p, q$ are distinct primes such that $pq \neq 6$, has automorphism group isomorphic to $C_{pq}$, and its Jacobian is simple with endomorphism algebra isomorphic to $\Q ( \zeta_{pq} )$.
            \end{The}
            \begin{proof}
                See \cite[\S~$2$]{HazamaCatalan} or \cite[\S~$3.2$]{GoodsonCatalan}.
            \end{proof}

            Next, we show that the trivial group appears as the automorphism group of a curve over $\Qbar$ of genus $\ge 2$ with simple Jacobian.

            \begin{Lem}
            \label{lem: trivial group is realizable; sec: Preliminaries}
                For every integer $g \ge 3$, there exists a curve over $\Qbar$ of genus $g$ with simple Jacobian and trivial automorphism group.
            \end{Lem}
            \begin{proof}
                First, we fix some notation. Let $K$ be a number field, $\Gamma_{K}$ be its absolute Galois group and $\Acal$ be an abelian variety over $K$. Let $\ell$ be a prime, $V_{\ell} ( \Acal )$ be the $\ell$-adic Tate module of $\Acal$ and $\rho_{\Acal, \ell^{\infty}} \colon \Gamma_{K} \to \Aut \bigl( V_{\ell}(\Acal) \bigr)$ be the $\ell$-adic Galois representation attached to $\Acal$. Denote by $G_{\Acal, \ell^{\infty}}$ the image of $\rho_{\Acal, \ell^{\infty}}$.
            
                By \cite[Theorems~1.1 and~5.4]{LandesmanTrivialExample}, for every integer $g \ge 3$ there exists a non hyperelliptic curve $\Ccal$ over $K$ of genus $g$ such that $G_{\Jcal, \ell^{\infty}}$ is $\ell$-adically open and Zariski dense in the symplectic group $\mathrm{GSp}_{2g} ( \Q_{\ell} )$, where $\Jcal$ is the Jacobian of $\Ccal$. Let $K'$ be the minimal finite extension of $K$ over which all the geometric endomorphisms of $\Jcal$ are well defined. Since $K'$ is a finite extension of $K$, we can replace $K$ with $K'$ without losing the property that $G_{\Jcal, \ell^{\infty}}$ is $\ell$-adically open and Zariski dense in $\mathrm{GSp}_{2g} ( \Q_{\ell} )$.
            
                Faltings' isogeny theorem asserts that $\End \bigl( \Jcal \bigr) \otimes_{\Z} \Q_{\ell}$ is isomorphic to the $\Q_{\ell}$-algebra of endomorphisms of $V_{\ell} \bigl( \Jcal \bigr)$ that commute with the action of the Zariski closure of $G_{\Jcal, \ell^{\infty}}$ inside $\mathrm{GSp}_{2g} ( \Q_{\ell} )$, which is $\mathrm{GSp}_{2g} ( \Q_{\ell} )$. It follows that $\End \bigl( \Jcal \bigr) \otimes_{\Z} \Q_{\ell}$ is isomorphic to $\Q_{\ell}$; hence, $\End^{0} \bigl( \Jcal \bigr)$ is isomorphic to $\Q$ and $\Jcal$ is simple. Since all the geometric endomorphisms of $\Jcal$ are defined over $K$, we deduce that $\Jcal$ is geometrically simple with geometric endomorphism algebra isomorphic to $\Q$.

                Since the geometric automorphism group of $\Ccal$ injects into the group of units of finite order of the geometric endomorphism algebra of $\Jcal$, see \cite[Theorem~$12.1$, \S~III]{milneAV}, we deduce that the geometric automorphism group of $\Ccal$ is either trivial or isomorphic to $C_{2}$. In the latter case, $\Ccal$ would be hyperelliptic by Theorem~\ref{the: necessary condition simplicity; sec: Preliminaries}, which is a contradiction; so, we have the result.
            \end{proof}

        \subsection{Abelian varieties}
        \label{subsec: Abelian varieties; sec: Preliminaries}

            We now recall two facts about abelian varieties.
            
            \begin{Def}[abelian variety with CM]
            \label{def: abelian variety of CM type; subsec: Abelian varieties; sec: Preliminaries}
                An abelian variety $A$ over a field $K$ of characteristic $0$ is said to have complex multiplication (CM) by a CM field $F$ if there exists an embedding $F \xhookrightarrow{} \End^{0} ( A )$ and $[ F \colon \Q ] = 2 \dim(A)$. The CM type of $A$ is the set of characters of the action of $F$ on $H^{0} ( A, \Omega^{1}_{A} ) \otimes \overline{K}$. If $F / \Q$ is Galois, then the CM type can be identified with a subset of $\Gal ( F / \Q )$.
            \end{Def}

            \begin{Lem}
            \label{lem: divisibility dimension and degree End simple abelian varieties; subsec: Abelian varieties; sec: Preliminaries}
                Let $X$ be a simple abelian variety over a field $k$ of characteristic $0$. The $\Q$-dimension of $\End^{0} ( X )$ divides $2 \dim ( X )$.
            \end{Lem}
            \begin{proof}
                \cite[Remark on p.~$182$]{mumford1970abelian}.
            \end{proof}

        \subsection{Specialization methods}
        \label{subsec: specialization methods; sec: Preliminaries}

            Let $L$ be a positive integer and $U$ be a non-empty open subscheme of $\Abb_{\Qbar}^{L}$. Consider a smooth family $\Cfrak \to U$ of curves and let $\Jfrak \to U$ be the Jacobian scheme of $\Cfrak \to U$. Moreover, let $\Cbf$ (resp.~$\Jbf$) be the generic fiber over $\Qbar ( U )$ of $\Cfrak \to U$ (resp.~$\Jfrak \to U$), and let $\Cbf^{\mathrm{al}}$ (resp.~$\Jbf^{\mathrm{al}}$) be the base change to an algebraic closure $\Qbar ( U )^{\mathrm{al}}$ of $\Qbar ( U )$. 

            Given a point $a = ( a_{1}, \dots, a_{L} ) \in U ( \Qbar )$, the specialization of $\Cfrak \to U$ at $a$ is the fiber $\Cfrak_{a}$, and the specialization of $\Jfrak \to U$ at $a$ is the fiber $\Jfrak_{a}$, which is the Jacobian of $\Cfrak_{a}$ since $\Jfrak \to U$ is the Jacobian scheme of $\Cfrak \to U$.

            It is possible to specialize $\Cfrak \to U$ at the $\Qbar$-points of the boundary $\partial U$ of $U$ in the following way (note that $\partial U =  \Abb^{L}_{\Qbar} \setminus U$ since $U$ is dense).
            Consider an integral, closed, normal subscheme $T \subset U$ isomorphic to a non-empty open subscheme of $\Abb^{1}_{\Qbar}$, a finite morphism $T' \to T$, and the pullbacks
            \begin{center}
                \begin{tikzcd}
                    \Cfrak_{T'} \arrow{r} \arrow{d} & \Cfrak_{T} \arrow{r} \arrow{d} & \Cfrak \arrow{d} \\ 
                    T' \arrow{r} & T \arrow{r} & U
                \end{tikzcd}
            \end{center}
                
            Suppose that $T'$ is itself isomorphic to a non-empty open subscheme of $\Abb^1_{\Qbar}$ and that the minimal regular model $\Cfrak_{\mathrm{reg}} \to \Abb^1_{\Qbar}$ of $\Cfrak_{T'} \to T'$ is semistable. Let $\Jfrak_{\mathrm{Ner}} \to \Abb^{1}_{\Qbar}$ be the Néron model of the abelian variety $\Jac \bigl( ( \Cfrak_{\mathrm{reg}} )_{\Qbar(T')} \bigr)$ over the field $\Qbar ( T' )$, so that $\Jfrak_{\mathrm{Ner}} \to \Abb^{1}_{\Qbar}$ is a semiabelian scheme; see \cite[Example~$8$, p.~$246$]{BoschNeronModel}. 

            \begin{Def}[admissible scheme]
            \label{def: admissible scheme; subsec: specialization methods; sec: Preliminaries}
                We call a scheme $T'$ as above \emph{admissible}. Given $t \in \Abb^{1}_{\Qbar}(\Qbar) \supset T'(\Qbar)$, we denote by $\Cfrak_{T', \, t}$ the fiber of $\Cfrak_{\mathrm{reg}}$ at $t$, and by $\Jfrak_{T', \, t}$ the fiber of $\Jfrak_{\mathrm{Ner}}$ at $t$.
            \end{Def}

            We now investigate the relation between the geometric endomorphism algebra of $\Jbf^{\mathrm{al}}$ and that of $\Jfrak_{T', t}$, and the relation between the torsion of $\Jbf^{\mathrm{al}}$ and that of $\Jfrak_{T', t}$, at least under the assumption that $\Jfrak_{T', t}$ is an abelian variety.

            \begin{Lem}
            \label{lem: End0 embeds and torsion bijects under specialization; subsec: specialization methods; sec: Preliminaries}
                Let $T, T'$ and $t \in \Abb^{1}_{\Qbar} ( \Qbar ) \supset T' ( \Qbar )$ be as in Definition~\ref{def: admissible scheme; subsec: specialization methods; sec: Preliminaries}, and suppose that $\Jfrak_{T', t}$ is an abelian variety. 
                
                There exists a canonical embedding $\End^{0} ( \Jbf^{\mathrm{al}} ) \xhookrightarrow{} \End^{0} ( \Jfrak_{T', t} )$ that preserves the dimensions of the identity components of the kernels of the idempotents. Moreover, given an endomorphism $\theta \in \End ( \Jbf^{\mathrm{al}} )$ with finite kernel, there is a bijection $\Jbf [ \theta ] \xrightarrow{\sim} \Jfrak_{T', t} [ \theta ]$.
            \end{Lem}
            \begin{proof}
                Let $\Jbf_{T}^{\mathrm{al}}$ (resp.~$\Jbf_{T'}^{\mathrm{al}}$) be the scalar extension to $\Qbar ( T )^{\mathrm{al}}$ (resp.~$\Qbar ( T' )^{\mathrm{al}}$) of the generic fiber of $\Jfrak_{T} \to T$ (resp.~$\Jfrak_{T'} \to T'$). Since $T' \to T$ is finite, $\Qbar ( T )^{\mathrm{al}} = \Qbar ( T' )^{\mathrm{al}}$, hence $\Jbf_{T}^{\mathrm{al}} = \Jbf_{T'}^{\mathrm{al}}$. Moreover, since $\Jfrak_{T', t}$ is defined over $\Qbar$, we have $\Jfrak_{T', t} = \Jfrak_{T', t}^{\mathrm{al}}$.
                    
                Let us begin by analyzing the relation between the geometric endomorphism algebras. By \cite[Lemma~$6.1.6$ $(d)$]{cantoral2023monodromy} we deduce that there exists a canonical embedding $\End^{0} ( \Jbf^{\mathrm{al}} ) \hookrightarrow \End^{0} ( \Jbf_T^{\mathrm{al}} )$. By \cite[Lemma~$6.1.7$ $(a)$]{cantoral2023monodromy} we deduce that there exists a canonical embedding $\End^{0} ( \Jbf_{T'}^{\mathrm{al}} ) \xhookrightarrow{} \End^{0} ( \Jfrak^{\mathrm{al}}_{T', t} )$. It follows that there exists a canonical embedding $\iota \colon \End^{0} ( \Jbf^{\mathrm{al}} ) \xhookrightarrow{} \End^{0} ( \Jfrak_{T', \, t}^{\mathrm{al}} )$.

                Let us continue by analyzing the relation between the torsion subgroups. Note first that an endomorphism of the generic fiber extends to an endomorphism of $\Jfrak$ by \cite[Lemma~6.1.6 (a)]{cantoral2023monodromy}.
                There exists a canonical map $\Jbf [ \theta ] \to \Jbf_{T} [ \theta ]$ constructed as follows. Since we are working in characteristic zero, the $\theta$-torsion subscheme $\Jfrak [ \theta ]$ is finite étale over $U$, hence there exists a finite étale cover $U' \to U$ such that all the $\theta$-torsion sections are defined over $U'$. After replacing $U, T, T' \subseteq \Abb^{1}_{\Qbar}$ by suitable finite étale covers, and $t$ by a $\Qbar$-point in the cover of $\Abb^{1}_{\Qbar}$ (by abuse of language, we again denote these covers and this point by $U, T, T'$ and $t$), we may then assume that $\theta$ is defined over $\Qbar(U)$, that it extends to $\mathfrak{J}$, and that all the $\theta$-torsion is defined over $U$. 
                The extension property of abelian schemes then shows that every section in $\Jbf [ \theta ]$ extends uniquely to a section in
                $\Jfrak ( U ) [ \theta ]$, so that the specialization map $\Jfrak ( U ) [ \theta ] \to \Jbf [ \theta ]$ is an isomorphism. We also have a restriction map $\Jfrak ( U ) [ \theta ] \to \Jfrak ( T' ) [ \theta ]$ and a specialization map $\Jfrak ( T' ) [ \theta ] \to \Jfrak_{ T', t } [ \theta ]$. These two maps are injective by \cite[Proposition~$3$, \S~$7.3$]{BoschNeronModel}. Composing with the inverse of $\Jfrak ( U ) [ \theta ] \to \Jbf [ \theta ]$, we obtain an embedding $\Jbf [ \theta ] \xhookrightarrow{} \Jfrak_{ T', t } [ \theta ]$. Moreover, since $\bigl| \ker ( \theta ) \bigr| = \deg ( \theta )$ and the degree is preserved under specialization maps, it follows that $\bigl| \Jbf [ \theta ] \bigr| = \bigl| \Jfrak_{ T', t } [ \theta ] \bigr|$. Hence, the embedding $\Jbf [ \theta ] \xhookrightarrow{} \Jfrak_{ T', t } [ \theta ]$ is a bijection.

                Finally, we prove the statement regarding the dimensions. For each $h \in \End^{0} ( \Jbf^{\mathrm{al}} )$, let $G_{h, 1}$ (resp.~$G_{h, 2}$) be the kernel of a chosen endomorphism in $\End ( \Jbf^{\mathrm{al}} )$ (resp.~$\End ( \Jfrak_{T', t}^{\mathrm{al}} )$) of the form $Nh$ (resp.~$\iota ( N h )$) for $N \in \Z \setminus \{ 0 \}$. Although these algebraic groups depend on the chosen multiple, their identity components, which are abelian varieties, do not. We have to prove that $\dim ( G_{h, 1}^{0} ) = \dim ( G_{h, 2}^{0} )$.

                Let $q$ be a prime not dividing the order of $G_{h, 2} / G_{h, 2}^{0}$. Then, $G_{h, 2} [ q ] = G_{h, 2}^{0} [ q ]$; otherwise, $G_{h, 2} / G_{h, 2}^{0}$ would contain an element of order $q$. Since the specialization of endomorphisms commutes with the specialization of torsion points by construction, and because the specialization of torsion points is injective, there exists an embedding $G_{h, 1}^{0} [ q ] \xhookrightarrow{} G_{h, 2} [ q ] = G_{h, 2}^{0} [ q ]$. As $G_{h, 1}^{0} [ q ]$ (resp.~$G_{h, 2}^{0} [ q ]$) is an $\Fbb_{q}$-vector space of dimension $2 \dim ( G_{h, 1}^{0} )$ (resp.~$2 \dim ( G_{h, 2}^{0} )$), it follows that $\dim ( G_{h, 1}^{0} ) \le \dim ( G_{h, 2}^{0} )$. Similarly, $\dim ( G_{1 - h, 1}^{0} ) \le \dim ( G_{1 - h, 2}^{0} )$.

                Let $h \in \End^{0} ( \Jbf^{\mathrm{al}} )$ be an idempotent. Then, $G_{h, 1}^{0}$ and $G_{1 - h, 1}^{0}$ intersect in only finitely many points. Since $1 = h + ( 1 - h )$, it follows that $G_{h, 1}^{0}$ and $G_{1 - h, 1}^{0}$ are complementary abelian subvarieties of $\Jbf^{\mathrm{al}}$. Similarly, $G_{h, 2}^{0}$ and $G_{1 - h, 2}^{0}$ are complementary abelian subvarieties of $\Jfrak_{T', t}$. Since $\dim ( \Jbf^{\mathrm{al}} ) = \dim ( \Jfrak_{T', t} )$, $\dim ( G_{h, 1}^{0} ) \le \dim ( G_{h, 2}^{0} )$ and $\dim ( G_{1 - h, 1}^{0} ) \le \dim ( G_{1 - h, 2}^{0} )$, it follows that $\dim ( G_{h, 1}^{0} ) = \dim ( G_{h, 2}^{0} )$, as desired.
            \end{proof}

            The following simple remark will be useful:
            
            \begin{Rem}
            \label{rem: simple observation of lem: End0 embeds and torsion bijects under specialization; subsec: specialization methods; sec: Preliminaries}
                Let $T, T'$ and $t \in \Abb^{1}_{\Qbar} ( \Qbar ) \supset T' ( \Qbar )$ be as in Definition~\ref{def: admissible scheme; subsec: specialization methods; sec: Preliminaries}, and suppose that $\Jfrak_{T', t}$ is an abelian variety. Fix an isotypic component $Z$ of $\Jbf^{\mathrm{al}}$. For each isotypic component $Y$ of $\Jfrak_{T', t}$, consider the sequence of morphisms
                \begin{equation*}
                    \End^{0} ( Z ) \xhookrightarrow{} \End^{0} ( \Jbf^{\mathrm{al}} ) \xhookrightarrow{\iota} \End^{0} ( \mathfrak{J}_{T', t} ) \twoheadrightarrow \End^{0} ( Y ) ,
                \end{equation*}
                where the embedding $\iota$ is provided by the first statement of Lemma~\ref{lem: End0 embeds and torsion bijects under specialization; subsec: specialization methods; sec: Preliminaries} and the other morphisms are the natural ones.

                As $Z$ is isotypic, $\End^{0} ( Z )$ is a simple algebra; thus, the induced morphism $\End^{0} ( Z ) \to \End^{0} ( Y )$ is either injective or trivial. Since $\End^{0} ( \Jbf^{\mathrm{al}} ) \xhookrightarrow{\iota} \End^{0} ( \mathfrak{J}_{T', t} )$ is nontrivial, the induced morphism cannot be trivial for every $Y$; thus, there exists an isotypic component $Y$ of $\Jfrak_{T', t}$ such that $\End^{0} ( Z )$ embeds into $\End^{0} ( Y )$. 

                Moreover, if $\Jbf^{\mathrm{al}}$ is isotypic, the induced morphism $\End^{0} ( \Jbf^{\mathrm{al}} ) \to \End^{0} ( Y )$ is injective for every $Y$, since the embedding $\iota$ maps the identity of $\End^{0} ( \Jbf^{\mathrm{al}} )$ to the identity of $\End^{0} ( \mathfrak{J}_{T', t}^{\mathrm{al}} )$.
            \end{Rem}

            Lemma~\ref{lem: End0 embeds and torsion bijects under specialization; subsec: specialization methods; sec: Preliminaries} has an interesting consequence:

            \begin{Lem}
            \label{lem: isogeny decomposition specializes; subsec: specialization methods; sec: Preliminaries}
                Let $T, T'$ and $t \in \Abb^{1}_{\Qbar} ( \Qbar ) \supset T' ( \Qbar )$ be as in Definition~\ref{def: admissible scheme; subsec: specialization methods; sec: Preliminaries}, and suppose that $\Jfrak_{T', t}$ is an abelian variety. 
                    
                Assume that $\Jbf^{\mathrm{al}}$ is isogenous to $X_{1} \times \cdots \times X_{S}$ for some positive integer $S$ and abelian varieties $X_{1}, \dots, X_{S}$ over $\Qbar ( U )^{\mathrm{al}}$. Then, there exist abelian varieties $Y_{1}, \dots, Y_{S}$ over $\Qbar$ such that $\Jfrak_{T', t}$ is isogenous to $Y_{1} \times \cdots \times Y_{S}$ and $\dim ( X_{s} ) = \dim ( Y_{s} )$ for every integer $1 \le s \le S$.

                In particular, there exists a simple factor of $\Jbf^{\mathrm{al}}$ of dimension at least the maximal dimension of a simple factor of $\Jfrak_{T', t}$.
            \end{Lem}
            \begin{proof}
                Without loss of generality, we can consider every $X_{s}$ as an abelian subvariety of $\Jbf^{\mathrm{al}}$ and the isogeny as the morphism $g \colon X_{1} \times \cdots \times X_{S} \to \Jbf^{\mathrm{al}}$ given by $( x_{1}, \dots, x_{S} ) \mapsto x_{1} + \cdots + x_{S}$. Let $N \ge 1$ be an integer and $h_{1}, \dots, h_{S}$ be endomorphisms of $\Jbf^{\mathrm{al}}$ such that for every $1 \le s \le S$ we have $h_{s}^{2} = N \cdot h_{s}$ and the identity component of the kernel of $h_{s}$ is $X_{s}$. 
                
                By the first statement of Lemma~\ref{lem: End0 embeds and torsion bijects under specialization; subsec: specialization methods; sec: Preliminaries}, there is an embedding $\iota \colon \End^{0} ( \Jbf^{\mathrm{al}} ) \xhookrightarrow{} \End^{0} ( \Jfrak_{T', t} )$ that preserves the dimensions of the identity components of the kernels of the idempotents. For each $h_{s}$, let $Y_{s}$ be the identity component of the kernel of $\iota ( h_{s} )$. By construction $\dim ( X_{s} ) = \dim ( Y_{s} )$ for every integer $1 \le s \le S$. Moreover, the morphism $\tilde{g} \colon Y_{1} \times \cdots \times Y_{S} \to \Jfrak_{T', t}$ given by $( y_{1}, \dots, y_{S} ) \mapsto y_{1} + \cdots + y_{S}$ has a finite kernel (equivalently, the kernel has dimension $0$), since the same holds for $g$. It follows that $\tilde{g}$ is an isogeny.
            \end{proof}

            \begin{Not}
            \label{not: specialization of abelian subvar; subsec: specialization methods; sec: Preliminaries}
                Following the notation of Lemma~\ref{lem: isogeny decomposition specializes; subsec: specialization methods; sec: Preliminaries}, for each $1 \le s \le S$, we say that the abelian subvariety $Y_{s}$ of $\Jfrak_{T', t}$ is the specialization of the abelian subvariety $X_{s}$ of $\Jbf^{\mathrm{al}}$.
            \end{Not}

            Lemma~\ref{lem: End0 embeds and torsion bijects under specialization; subsec: specialization methods; sec: Preliminaries} implies that the geometric endomorphism algebra of every fiber of $\Jfrak \to U$ which is an abelian variety contains a subalgebra isomorphic to $\End^{0} ( \Jbf^{\mathrm{al}} )$. We now investigate how many fibers of $\Jfrak \to U$ have geometric endomorphism algebra strictly larger than that of $\Jbf^{\mathrm{al}}$.

            \begin{Def}
            \label{def: height; subsec: specialization methods; sec: Preliminaries}
                Let $K$ be a number field of degree $d = [ K : \Q ]$. We denote by
                \begin{equation*}
                    H ( x_{0} : \dots : x_{n} ) = \prod_{ v \in M_{K} } \max_{i} | x_{i} |_{v}^{ n_{v} / d }
                \end{equation*}
                the absolute multiplicative Weil height on $\mathbb{P}_{ n, K }(K)$, where $M_{K}$ is the set of places of $K$ and, for every $v \in M_{K}$, we set $n_{v} \coloneqq [ K_{v} : \Q_{p} ]$ where $p$ is the place below $v$. Moreover, we denote by
                \begin{equation*}
                    H_{K} ( x_{0} : \dots : x_{n} ) \coloneqq H ( x_{0} : \dots : x_{n} )^{d}
                \end{equation*}
                the $K$-relative multiplicative Weil height. Finally, for $( x_{1}, \dots, x_{L} ) \in K^{L}$ we set as usual
                \begin{equation*}
                    H ( x_{1}, \dots, x_{L} ) \coloneqq H ( 1 : x_{1} : \dots : x_{L} )
                \end{equation*}
                and
                \begin{equation*}
                    H_{K} ( x_{1}, \dots, x_{L} ) \coloneqq H_{K} ( 1 : x_{1} : \dots : x_{L} ) .
                \end{equation*}
            \end{Def}  

            The following lemma is probably well-known to experts, but we have not been able to locate it in the literature in the simple form below, where the bound is linear in the degree of the polynomial.

            \begin{Lem}
            \label{lem: number of bounded zeros polynomial; subsec: Preliminaries; sec: Cyclic groups of order p^n}
                Let $K$ be a number field of degree $d = [ K : \Q ]$, and let $P \in K [ X_{1}, \dots, X_{L} ]$ be a non-zero polynomial of degree $D$. There exists a constant $C = C ( K, L ) > 0$, independent of $P$, such that, for every $B \ge 1$, we have
                \begin{equation*}
                    \# \bigl\{ ( x_{1}, \dots, x_{L} ) \in K^{L} \colon P ( x_{1}, \dots, x_{L} ) = 0,\ H ( x_{1}, \dots, x_{L} ) \le B \bigr\} \le C \cdot D \cdot B^{ d \cdot L } .
                \end{equation*}
                Equivalently, for every $T \ge 1$ we have
                \begin{equation*}
                    \# \bigl\{ ( x_{1}, \dots, x_{L} ) \in K^{L} \colon P ( x_{1}, \dots, x_{L} ) = 0,\ H_{K} ( x_{1}, \dots, x_{L} ) \le T \bigr\} \le C \cdot D \cdot T^{L} .
                \end{equation*}
            \end{Lem}

            For the proof, we will need the following statement.

            \begin{Pro}[{\cite[Proposition~2.2]{MR4516199}}]
            \label{prop: bound on affine coordinates; subsec: Preliminaries; sec: Cyclic groups of order p^n}
                Let $K$ be a number field. There exists a constant $c = c ( K ) > 0$ such that for every $n \ge 1$, each point $( x_{0} \colon \dots \colon x_{n} ) \in \mathbb{P}_{n, K} ( K )$ can be represented by a point $ ( y_{0}, \dots, y_{n} ) \in \Ocal_{K}^{ n + 1 } \setminus \{ 0 \}$ satisfying
                \begin{equation*}
                    \max_{ 0 \le i \le n } | \sigma ( y_{i} ) | \le c \cdot H ( x_{0} \colon \dots \colon x_{n} )
                \end{equation*}
                for every embedding $\sigma \colon K \hookrightarrow \C$.
            \end{Pro}

            \begin{proof}[Proof of Lemma~\ref{lem: number of bounded zeros polynomial; subsec: Preliminaries; sec: Cyclic groups of order p^n}]
                Let $(r_1, r_2)$ be the signature of the number field $K$. For $R \ge 1$, let
                \begin{equation*}
                    \Bcal_{K} ( R ) \coloneqq \bigl\{ a \in \Ocal_{K} \colon | \sigma ( a ) | \le R \text{ for every embedding } \sigma \colon K \hookrightarrow \C \bigr\} 
                \end{equation*}
                and
                \begin{equation*}
                    S( R ) \coloneqq \bigl\{ ( a_{1}, \dots, a_{ r_{1} }, b_{1}, \dots, b_{ r_{2} } ) \in \R^{ r_{1} } \times \C^{ r_{2} } \colon | a_{i} |, | b_{j} | \le R \text{ for $1 \le i \le r_{1}$ and $1 \le j \le r_{2}$} \bigr\} .
                \end{equation*}
                We prove first that there exists a constant $c_{K} > 0$ such that $ | \Bcal_{K} ( R ) | \le c_{K} \cdot R^{d}$ for all $R \ge 1$. Let $\iota \colon \Ocal_{K} \hookrightarrow \R^{ r_{1} } \times \C^{ r_{2} } \simeq \R^{d}$ be the Minkowski embedding and note that $\iota \bigl( \Bcal_{K} ( R ) \bigr) = \iota ( \Ocal_{K} ) \cap S ( R )$. Since $\iota ( \Ocal_{K} )$ is a lattice, $S ( R ) = R \cdot S ( 1 )$ and the volume of $S ( 1 )$ depends only on $K$, \cite[Theorem 2 on p.~128]{MR1282723} implies the existence of such a $c_{K}$.

                It follows from this and the elementary Schwartz--Zippel lemma that, for every non-zero polynomial $Q \in K [ X_{1}, \dots, X_{m} ]$ of degree $E$, we have
                \begin{equation}
                \label{eq: 1; lem: number of bounded zeros polynomial; subsec: Preliminaries; sec: Cyclic groups of order p^n}
                    \# \bigl\{ ( y_{1}, \dots, y_{m} ) \in \Bcal_{K} ( R )^{m} \colon Q ( y_{1}, \dots, y_{m} ) = 0 \bigr\} \le E \cdot \# \Bcal_{K} ( R )^{ m - 1 } \le c_{ K, m } \cdot E \cdot R^{ d ( m - 1 ) } ,
                \end{equation}
                where $c_{ K, m } = c_{K}^{ m - 1 }$ is independent of $Q$.

                Let $F ( X_{0}, \dots, X_{L} ) = X_{0}^{D} \cdot P \left( \frac{ X_{1} }{ X_{0} }, \dots, \frac{ X_{L} }{ X_{0} } \right)$ be the homogenization of $P$, which is a non-zero homogeneous polynomial of degree $D$. Every point $( x_{1}, \dots, x_{L} ) \in K^{L}$ with $P ( x_{1}, \ldots, x_{L} ) = 0$ determines a point $( 1 : x_{1} : \dots : x_{L} ) \in \mathbb{P}_{ L, K } ( K )$ satisfying $F ( 1, x_{1}, \dots, x_{L} ) = 0$. If $H ( x_{1}, \dots, x_{L} ) \le B$, Proposition~\ref{prop: bound on affine coordinates; subsec: Preliminaries; sec: Cyclic groups of order p^n} yields the existence of a representative $( y_{0}, \dots, y_{L} ) \in \Ocal_{K}^{ L + 1 } \setminus \{ 0 \}$ for this projective point such that $| \sigma ( y_{i} ) | \le c \cdot B$ for every $i$ and every embedding $\sigma \colon K \hookrightarrow \C$, that is, $( y_{0}, \dots, y_{L} )$ belongs to $\Bcal_{K} ( R )^{L+1}$ for $R = c \cdot B$. Since $F$ is homogeneous, we have $F ( y_{0}, \dots, y_{L} ) = 0$.

                It follows from \eqref{eq: 1; lem: number of bounded zeros polynomial; subsec: Preliminaries; sec: Cyclic groups of order p^n}, applied to the polynomial $F$ in $L + 1$ variables, that the number of points $( y_{0}, \dots, y_{L} )$ in $\Bcal_{K} ( c \cdot B )$ that satisfy $F ( y_{0}, \ldots, y_{L} ) = 0$ is at most
                \begin{equation*}
                    c_{ K, L + 1 } \cdot D \cdot ( c \cdot B )^{ d \cdot L } \ll_{ K, L } D \cdot B^{ d \cdot L }.
                \end{equation*}
                Counting the vectors $( y_{0}, \dots, y_{L} )$ can only overcount the projective points $( 1 : x_{1} : \dots : x_{L} )$, so the same upper bound holds for the original affine points $( x_{1}, \dots, x_{L} )$. This proves the first assertion. The second is an immediate consequence of the definition.
            \end{proof}

            \begin{Lem}
            \label{lem: how many fibers of Jfrak have bigger End0; subsec: specialization methods; sec: Preliminaries}
                Suppose that $\Cfrak \to U$ is the base change to $\Qbar$ of a family $\tilde{\Cfrak} \to \tilde{U}$ defined over the number field $F$. For every number field $K \supseteq F$ and for $100\%$ of the points $( a_{1}, \dots, a_{L} ) \in U ( \Qbar ) \cap K^{L}$, the canonical specialization map $\End^{0} ( \Jbf^{ \mathrm{al} } ) \to \End^{0} ( \Jfrak_{ ( a_{1},  \ldots, a_{L} ) } )$ is an isomorphism.
            \end{Lem}
            \begin{proof}
                Let $\tilde{ \Jfrak } \to \tilde{ U }$ be the Jacobian scheme of $\tilde{ \Cfrak } \to \tilde { U }$, and note that $\Jfrak_{ ( a_{1}, \dots, a_{L} ) } = \tilde{ \Jfrak }_{ ( a_{1}, \dots, a_{L} ) } \times_{F} \Qbar$, where we regard $( a_{1}, \dots, a_{L} ) \in \tilde{ U } ( K ) $ as a point of $U ( \Qbar ) \cap K^{L}$.
                For each positive real number $B$, let $V ( K, B )$ be the set of points $( a_{1}, \dots, a_{L} ) \in U ( \Qbar ) \cap K^{L}$ of absolute multiplicative Weil height at most $B$, and let $V_{\mathrm{ex}} ( K, B )$ be the subset of points in $V ( K, B )$ such that $\End^{0} ( \Jbf^{\mathrm{al}} ) \to \End^{0} ( \Jfrak_{ ( a_{1},  \ldots, a_{L} ) } )$ is not an isomorphism. Note that $V ( K, B )$ is a finite set (\cite[Lemma~$2$]{NorthcottProperty}). For a finite subset $S \subset \Abb_{\Qbar}^{L}(\Qbar) \cap K^{L}$, let $\omega ( S )$ be the minimal degree of a non-zero polynomial in $L$ variables with coefficients in $K$ vanishing on $S$.
                
                By applying \cite[Theorem on p.~459]{MasserSpecializationEnd} to the base change to $K$ of $\tilde{\Jfrak} \to \tilde{U}$, we deduce the existence of constants $C, \lambda > 0$ such that 
                \begin{equation*}
                    \omega \bigl( V_{\mathrm{ex}} ( K, B ) \bigr) \le C \bigl( \max \{ 1, \log(B) \} \bigr)^{\lambda}.
                \end{equation*}
                By Lemma~\ref{lem: number of bounded zeros polynomial; subsec: Preliminaries; sec: Cyclic groups of order p^n} we deduce that
                \begin{equation}
                \label{eq: 1; the: reduction first statement main the; subsec: Preliminaries; sec: Cyclic groups of order p^n}
                    \bigl| V_{\mathrm{ex}} ( K, B ) \bigr| = O \bigl( \log(B)^{\lambda} \cdot B^{ d \cdot L } \bigr) \qquad \text{with respect to $B$,}
                \end{equation}
                where $d \coloneqq [ K \colon \Q ]$. Since $U$ is open in $\Abb_{\Qbar}^{L}$, there exists a non-zero polynomial $P$ in $L$ variables with coefficients in $K$ such that its zero set contains $( \Abb_{\Qbar}^{L} \setminus U ) ( \Qbar ) \cap K^{L}$. Lemma~\ref{lem: number of bounded zeros polynomial; subsec: Preliminaries; sec: Cyclic groups of order p^n} implies that the number of points in $( \Abb_{\Qbar}^{L} \setminus U ) ( \Qbar ) \cap K^{L}$ of absolute multiplicative Weil height at most $B$ is $ O ( B^{ d \cdot L } )$ with respect to $B$. Since the number of points in $ \Abb_{\Qbar}^{L} ( \Qbar ) \cap K^{L}$ of absolute multiplicative Weil height at most $B$ is $\asymp B^{ d \cdot ( L + 1 ) }$ (\cite[Corollary on p.~447]{SchanuelEstimatesNorthcootProperty}), we deduce that 
                \begin{equation}
                \label{eq: 2; the: reduction first statement main the; subsec: Preliminaries; sec: Cyclic groups of order p^n}
                    | V ( K, B ) | \asymp B^{ d \cdot ( L + 1 ) } \qquad \text{with respect to $B$.}
                \end{equation}
                Combining \eqref{eq: 1; the: reduction first statement main the; subsec: Preliminaries; sec: Cyclic groups of order p^n} and \eqref{eq: 2; the: reduction first statement main the; subsec: Preliminaries; sec: Cyclic groups of order p^n} we deduce
                \begin{equation*}
                    \lim_{ B \to \infty } \frac{ | V_{\mathrm{ex}} ( K, B ) | }{ | V ( K, B ) | } = 0 .
                \end{equation*}
            \end{proof}

    \section{Possible automorphism groups}
    \label{sec: Possible automorphism groups}

        In this section, we find several constraints on the automorphism group of a curve over $k$ of genus $\ge 2$ with simple Jacobian. We then use these constraints to obtain a small list of groups to which the automorphism group of a curve over $k$ of genus $\ge 2$ with simple Jacobian must belong. In particular, we show the following:

        \begin{The}
        \label{the: possible Aut nice curve with simple Jacobian; sec: Possible automorphism groups}
            Let $\Ccal / k$ be a curve of genus $\ge 2$ with simple Jacobian. Then $\Aut ( \Ccal )$ is isomorphic to one of the following groups:
            \begin{enumerate}[(i)]
                \item the cyclic group $C_{p^{n}}$ for some prime $p$ and some positive integer $n$;
                \item the cyclic group $C_{pq}$ for some distinct primes $p, q$ such that $pq \neq 6$;
                \item the generalized quaternion group $G_{2^{n}}$ (cf. Theorem~\ref{the: classification Aut of hyperelliptic curves; sec: Preliminaries}) for some positive integer $n$;
                \item the trivial group.
            \end{enumerate}
        \end{The}
        \begin{proof}
            Combine Lemmas~\ref{lem: Aut even cardinality; subsec: Classification; sec: Possible automorphism groups} and~\ref{lem: Aut odd cardinality; subsec: Classification; sec: Possible automorphism groups} below.
        \end{proof}

        First, we prove that any curve over $k$ of genus $\ge 2$ with simple Jacobian and nontrivial automorphism group is superelliptic.

        \begin{Pro}
        \label{pro: a curve with simple Jacobian is superelliptic; sec: Preliminaries}
            Let $\Ccal / k$ be a curve of genus $\ge 2$ with simple Jacobian. If $\Aut ( \Ccal )$ contains an element of order $n \ge 2$, then $\Ccal$ is isomorphic to a superelliptic curve given by $y^{n} = f(x)$ for some $f \in k [ x ]$.
        \end{Pro}
        \begin{proof}
            Let $\sigma \in \Aut ( \Ccal )$ be of order $n \ge 2$. Since the Jacobian of $\Ccal$ is simple, $\Ccal / \langle \sigma \rangle$ is isomorphic to $\mathbb{P}_{1, k}$ (Theorem~\ref{the: necessary condition simplicity; sec: Preliminaries}); hence, the quotient morphism $\Ccal \to \Ccal / \langle \sigma \rangle$ is a ramified cover of $\mathbb{P}_{1, k}$ with abelian Galois group $\langle \sigma \rangle$ of order $n$. Since $k ( x )$ contains the $n$-th roots of unity, Kummer theory implies that
            \begin{equation*}
                k ( \Ccal ) = k ( x ) ( \sqrt[n]{f} ) = \operatorname{Frac} \Bigl( k [ x, y ] / \bigl( y^{n} - f(x) \bigr) \Bigr)
            \end{equation*}
            for some $f \in k ( x )$. Up to multiplying $f$ by the $n$-th power of its denominator, which does not change $k ( x ) ( \sqrt[n]{f} )$, we may assume that $f \in k [ x ]$.
        \end{proof}
            
        Let $\Ccal / k$ be a superelliptic curve given by $y^{n} = f(x)$, where $n \ge 2$ is an integer and $f \in k [ x ]$, and consider the degree $n$, ramified, Galois cover $\pi \colon \Ccal \to \mathbb{P}_{1, k}$ defined by $( x, y ) \mapsto x$. We provide a formula expressing the genus of $\Ccal$ in terms of the ramification data of $\pi$. 
            
        \begin{Def}
        \label{def: branch index; sec: Preliminaries}
            We define the branch index of a point $z \in \mathbb{P}_{1} ( k ) $ to be the ramification index of any point $a \in \pi^{-1} ( z ) $. This definition is independent of $a$ since $\pi$ is a Galois cover.
        \end{Def}

        \begin{Not}
        \label{not: mz; sec: Preliminaries}
            For each $z \in \mathbb{P}_{1}(k)$, define $m_{z}$ as follows. If $z$ is a zero of $f$, let $m_{z}$ be its multiplicity. If $z = \infty$, let $m_{\infty}$ be the degree of $f$. Otherwise, let $m_{z}$ be $n$.
        \end{Not}

        \begin{Lem}
        \label{lem: genus superelliptic curves; sec: Preliminaries}
            Let $\Ccal / k$ be a superelliptic curve given by $y^{n} = f(x)$, where $n \ge 2$ is an integer and $f \in k [ x ]$, and consider the cover $\pi \colon \Ccal \to \mathbb{P}_{1, k}$ defined by $( x, y ) \mapsto x$. The genus of $\Ccal$ satisfies
            \begin{equation*}
                2 g ( \Ccal ) = 2 - 2n + \sum_{ \substack{ s \mid n \\ s \neq 1 } } \Acal_{s} \cdot \biggl( n - \frac{n}{s} \biggr) ,
            \end{equation*}
            where $\Acal_{s}$ is the number of $z \in \mathbb{P}_{1}(k)$ such that $\GCD ( n, m_{z} ) = n / s$.
        \end{Lem}
        \begin{proof}
            The branch points of $\pi$ can occur only above the zeros of $f$ and the point at infinity $\infty$, see \cite[Proposition~$3.7.3$]{StichtenothBranchPointsSupCurves}. The branch index of $z \in \mathbb{P}_{1}(k)$ is $e_{z} = n / \GCD ( n, m_{z} )$, see \cite[Proposition~$3.7.3$]{StichtenothBranchPointsSupCurves} and \cite[paragraph after Assumption~$4.1$]{BouwBranchPoints}. The lemma follows from the Riemann--Hurwitz formula.
        \end{proof}

        As an interesting consequence of the genus formula provided by Lemma~\ref{lem: genus superelliptic curves; sec: Preliminaries}, we have the following proposition:
            
        \begin{Pro}
        \label{pro: genus superelliptic curve pq; sec: Preliminaries}
            Let $\Ccal_{pq} / k$ be a superelliptic curve of genus $\ge 2$ given by $y^{pq} = f(x)$, for some distinct primes $p, q$ and some $f \in k [ x ]$, with simple Jacobian. The genus of $\Ccal_{pq}$ is either $( p - 1 ) ( q - 1 )$ or $( p - 1 ) ( q - 1 ) / 2$.         
        \end{Pro}
        \begin{proof}
            Consider the curves $\Ccal_{p} \colon y^{p} = f(x)$ and $\Ccal_{q} \colon y^{q} = f(x)$.  There are nonconstant morphisms
            \begin{align*}
                \Ccal_{pq} &\to \Ccal_{p} \qquad\qquad\qquad ( x, y ) \mapsto ( x, y^{q} ) , \\
                \Ccal_{pq} &\to \Ccal_{q} \qquad\qquad\qquad ( x, y ) \mapsto ( x, y^{p} ) .
            \end{align*}
            Since $g ( \Ccal_{pq} ) \ge 2$, from the Riemann--Hurwitz formula we deduce that $g ( \Ccal_{p} ) < g ( \Ccal_{pq} )$ and $g ( \Ccal_{q} ) < g ( \Ccal_{pq} )$; hence, $g ( \Ccal_{p} ) = g ( \Ccal_{q} ) = 0$ (Lemma~\ref{lem: necessary condition simplicity; sec: Preliminaries}).

            Let us compute the genus of $\Ccal_{p}$ and $\Ccal_{q}$ using Lemma~\ref{lem: genus superelliptic curves; sec: Preliminaries}. Let $m_{z, pq}$, $m_{z, p}$ and $m_{z, q}$ denote the corresponding value of $m_{z}$ for $\Ccal_{pq}$, $\Ccal_{p}$ and $\Ccal_{q}$ (Notation~\ref{not: mz; sec: Preliminaries}). For each $z \in \mathbb{P}_{1}(k)$ that is either a zero of $f$ or $\infty$, we have $m_{z, pq} = m_{z, p} = m_{z, q}$; hence, since $p$ and $q$ are distinct primes, we deduce
            \begin{align*}
                \GCD ( pq, m_{z, pq} ) &= 1 \qquad\quad \Longrightarrow \qquad\quad \GCD ( p, m_{z, p} ) = 1 \text{ and } \GCD ( q, m_{z, q} ) = 1 , \\
                \GCD ( pq, m_{z, pq} ) &= p \qquad\quad \Longrightarrow \qquad\quad \GCD ( p, m_{z, p} ) = p \text{ and } \GCD ( q, m_{z, q} ) = 1 , \\
                \GCD ( pq, m_{z, pq} ) &= q \qquad\quad \Longrightarrow \qquad\quad \GCD ( p, m_{z, p} ) = 1 \text{ and } \GCD ( q, m_{z, q} ) = q . 
            \end{align*}
            For each divisor $s$ of $pq$, let $\Acal_{s}$ be the number of $z \in \mathbb{P}_{1}(k)$ such that $\GCD ( pq, m_{z, pq} ) = pq / s$. Lemma~\ref{lem: genus superelliptic curves; sec: Preliminaries} applied to each of the curves $\Ccal_{pq}$, $\Ccal_{p}$ and $\Ccal_{q}$ implies
            \begin{align}
            \label{equation 1 pro: genus superelliptic curve pq; sec: Preliminaries}
                2 g ( \Ccal_{pq} ) &= 2 - 2pq + \Acal_{pq} \cdot ( pq - 1 ) + \Acal_{p} \cdot ( pq - q ) + \Acal_{q} \cdot ( pq - p ) , \\
                2 g ( \Ccal_{p} ) &= 2 - 2p + ( \Acal_{pq} + \Acal_{p} ) \cdot ( p - 1 ) , \notag \\
                2 g ( \Ccal_{q} ) &= 2 - 2q + ( \Acal_{pq} + \Acal_{q} ) \cdot ( q - 1 ) \notag .
            \end{align}
            Since $g ( \Ccal_{p} ) = g ( \Ccal_{q} ) = 0$, we deduce
            \begin{equation}
            \label{equation 2 pro: genus superelliptic curve pq; sec: Preliminaries}
                \Acal_{pq} + \Acal_{p} = 2 , \qquad\qquad\qquad \Acal_{pq} + \Acal_{q} = 2 ,
            \end{equation}
            which, together with \eqref{equation 1 pro: genus superelliptic curve pq; sec: Preliminaries}, gives
            \begin{equation*}
                g ( \Ccal_{pq} ) = ( p - 1 ) ( q - 1 ) \biggl( 1 - \frac{\Acal_{pq}}{2} \biggr) .
            \end{equation*}
            The proposition follows since \eqref{equation 2 pro: genus superelliptic curve pq; sec: Preliminaries} implies that $0 \le \Acal_{pq} \le 2$.
        \end{proof}

        We now derive constraints on the automorphism group of a curve over $k$ of genus $\ge 2$ with simple Jacobian.
            
        \begin{The}
        \label{the: Aut does not contain elements order pqr; sec: Preliminaries}
            Let $\Ccal / k$ be a curve of genus $\ge 2$ with simple Jacobian. Then $\Aut ( \Ccal )$ does not contain elements of order $pqr$, where $p, q, r$ are distinct primes.
        \end{The}
        \begin{proof}
            Let $\sigma \in \Aut ( \Ccal )$ be of order $pqr$ for some distinct primes $p, q, r$. Since $\sigma^{r}$ has order $pq$, Propositions~\ref{pro: a curve with simple Jacobian is superelliptic; sec: Preliminaries} and~\ref{pro: genus superelliptic curve pq; sec: Preliminaries} imply that
            \begin{equation*}
                g ( \Ccal ) = ( p - 1 ) ( q - 1 ) \qquad\quad \text{or} \qquad\quad g ( \Ccal ) = \frac{( p - 1 ) ( q - 1 )}{2} .
            \end{equation*}
            Since $\sigma^{q}$ has order $pr$, the same propositions imply that
            \begin{equation*}
                g ( \Ccal ) = ( p - 1 ) ( r - 1 ) \qquad\quad \text{or} \qquad\quad g ( \Ccal ) = \frac{( p - 1 ) ( r -  1 )}{2} .
            \end{equation*}
            Since $\sigma^{p}$ has order $qr$, the same propositions imply that
            \begin{equation*}
                g ( \Ccal ) = ( q - 1 ) ( r - 1 ) \qquad\quad \text{or} \qquad\quad g ( \Ccal ) = \frac{( q - 1 ) ( r  - 1 )}{2} .
            \end{equation*}
            It is easy to see that $p, q, r$ cannot be distinct.
        \end{proof}

        \begin{The}
        \label{the: Aut does not have particular cyclic subgroups; sec: Preliminaries}
            Let $\Ccal / k$ be a curve of genus $\ge 2$ with simple Jacobian. Suppose that
            \begin{enumerate}
                \item there exists $\sigma_{1} \in \Aut ( \Ccal )$ of order $pq$, for some distinct primes $p, q$;
                \item there exists $\sigma_{2} \in \Aut ( \Ccal )$ of order $p$ such that $\langle \sigma_{2} \rangle$ is a normal subgroup of $\Aut ( \Ccal )$.
            \end{enumerate}
            If $q \neq 2$, there does not exist a subgroup of $\Aut ( \Ccal ) / \langle \sigma_{2} \rangle$ isomorphic to $C_{q^{2}}$.
        \end{The}
        \begin{proof}
            Assumption~$1$ and Propositions~\ref{pro: a curve with simple Jacobian is superelliptic; sec: Preliminaries} and~\ref{pro: genus superelliptic curve pq; sec: Preliminaries} imply that the genus of $\Ccal$ is either
            \begin{equation}
            \label{equation 1; the: Aut does not have particular cyclic subgroups; sec: Preliminaries}
                g ( \Ccal ) = ( p - 1 ) ( q -  1 ) \qquad\quad \text{or}\qquad\quad g ( \Ccal ) = \frac{( p - 1 ) ( q -  1 )}{2} .
            \end{equation}
            Since $\sigma_{2}$ has order $p$, we have that $\Ccal$ is isomorphic to a superelliptic curve $y^{p} = f(x)$, for some $f \in k [ x ]$, and the projection morphism $\Ccal \to \mathbb{P}_{1, k}$ given by $( x, y ) \mapsto x$ is the quotient morphism $\Ccal \to \Ccal / \langle \sigma_{2} \rangle$ (Proposition~\ref{pro: a curve with simple Jacobian is superelliptic; sec: Preliminaries}).
                
            Assume, for contradiction, that there exists a subgroup of $\Aut ( \Ccal ) / \langle \sigma_{2} \rangle$ isomorphic to $C_{q^{2}}$, and let $h$ be a generator of this subgroup. Every element of $\Aut ( \Ccal ) / \langle \sigma_{2} \rangle$ acts on $\Ccal / \langle \sigma_{2} \rangle \cong \mathbb{P}_{1, k}$ and stabilizes the branch locus of $\Ccal \to \mathbb{P}_{1, k}$. Since every cyclic subgroup of order $n$ of $\Aut ( \mathbb{P}_{1, k} )$ is conjugate to the subgroup generated by $z \mapsto \zeta_{n} \cdot z$, up to a change of coordinates on $\mathbb{P}_{1, k}$ we may assume that $h$ acts on $\mathbb{P}_{1, k}$ as multiplication by $\zeta_{q^{2}}$.
            It follows that
            \begin{equation*}
                f(x) = x^{n_{0}} \cdot \prod_{i = 1}^{d} \prod_{j = 1}^{q^{2}} ( x - \zeta_{q^{2}}^{j} \cdot a_{i} )^{n_{i, j}}
            \end{equation*}
            for some integer $d \ge 1$, some elements $a_{i} \in k^{*}$ such that the $a_{i}^{q^{2}}$ are pairwise distinct, some integer $0 \le n_{0} < p$ and some integers $1 \le n_{i, j} < p$ for $1 \le i \le d$ and $1 \le j \le q^{2}$. Let $\varepsilon = 1$ if $\infty$ is a branch point of $\Ccal \to \mathbb{P}_{1, k}$ and $\varepsilon = 0$ otherwise.

            The degree of the ramification divisor of $\Ccal \to \mathbb{P}_{1, k}$ is $( p - 1 ) ( dq^{2} + \varepsilon )$ if $n_{0} = 0$, and $( p - 1 ) ( dq^{2} + \varepsilon + 1 )$ if $n_{0} > 0$. The genus of $\Ccal$ is $( p - 1 ) ( dq^{2} + \varepsilon - 2 ) / 2$ if $n_{0} = 0$, and $( p - 1 ) ( dq^{2} + \varepsilon - 1 ) / 2$ if $n_{0} > 0$. From \eqref{equation 1; the: Aut does not have particular cyclic subgroups; sec: Preliminaries} we deduce that $n_{0} = \varepsilon = 0$, $d = 1$ and $q = 2$, which contradicts the assumption $q > 2$. 
        \end{proof}

        \subsection{Representation theory}
        \label{subsec: Representation theory; sec: Preliminaries}

            We show that the automorphism group of a curve over $k$ of genus $\ge 2$ with simple Jacobian admits finite-dimensional, irreducible, fixed-point-free representations over $k$. We assume all representations to be over $k$ and finite-dimensional.
            
            \begin{Def}[fixed-point-free representation]
            \label{def: fixed point free representation; subsec: Representation theory; sec: Preliminaries}
                Let $G$ be a finite group. A representation $\rho \colon G \to \GL ( V )$ is fixed-point-free if
                \begin{equation*}
                    V^{H} \coloneqq \bigl\{ v \in V \,\colon\, \rho(h)(v) = v \quad \forall \, h \in H \bigr\} = \{ 0 \}  
                \end{equation*}
                for every nontrivial subgroup $H$ of $G$.
            \end{Def}
           
            \begin{The}
            \label{the: AutC admits irr fixed point free rep; subsec: Representation theory; sec: Preliminaries}
                Let $\Ccal / k$ be a curve of genus $\ge 2$ with simple Jacobian. Then $\Aut ( \Ccal )$ admits irreducible, fixed-point-free representations. More precisely, every subrepresentation of the $\Aut ( \Ccal )$-module $H^{0} ( \Ccal, \Omega^{1}_{\Ccal} )$ is fixed-point-free.
            \end{The}
            \begin{proof}
                Let $H^{0} ( \Ccal, \Omega^{1}_{\Ccal} )$ be the $k$-vector space of regular differential forms on $\Ccal$. Every $\sigma \in \Aut ( \Ccal )$ induces an automorphism of $H^{0} ( \Ccal, \Omega^{1}_{\Ccal} )$ via pullback, so we have a representation
                \begin{equation}
                \label{eq: rep on reg diff forms; subsec: Representation theory; sec: Preliminaries}
                    \rho \colon \Aut ( \Ccal ) \to \GL \bigl( H^{0} ( \Ccal, \Omega^{1}_{\Ccal} ) \bigr) .
                \end{equation}

                Given a nontrivial subgroup $H$ of $\Aut ( \Ccal )$, consider the quotient morphism $\pi_{H} \colon \Ccal \to \Ccal / H$ which induces, via pullback, the injective morphism
                \begin{equation*}
                    \pi_{H}^{*} \colon H^{0} ( \Ccal / H, \Omega^{1}_{\Ccal / H} ) \to H^{0} ( \Ccal, \Omega^{1}_{\Ccal} ).
                \end{equation*}
                The image of this morphism is the $H$-invariant subspace $H^{0} ( \Ccal, \Omega^{1}_{\Ccal} )^{H}$. Since $\Ccal / H \cong \mathbb{P}_{1, k}$ (Theorem~\ref{the: necessary condition simplicity; sec: Preliminaries}), we deduce 
                \begin{equation*}
                    H^{0} ( \Ccal, \Omega^{1}_{\Ccal} )^{H} \cong H^{0} ( \Ccal / H, \Omega^{1}_{\Ccal / H} ) = \{ 0 \} .
                \end{equation*}
                Since this is true for every nontrivial subgroup $H$ of $\Aut ( \Ccal )$, the representation \eqref{eq: rep on reg diff forms; subsec: Representation theory; sec: Preliminaries} is fixed-point-free.
            \end{proof}
            
            \begin{Rem}
            \label{rem: classification groups admitting irr fixed point free rep; subsec: Representation theory; sec: Preliminaries}
                In \cite[Theorems~$6.1.11$ and~$6.3.1$]{WolfBookFixRep} one finds a classification of finite groups which admit irreducible, fixed-point-free representations. We use this classification in Lemma~\ref{lem: AutC is abelian if odd; subsec: Classification; sec: Possible automorphism groups}.
            \end{Rem}

        \subsection{Classification}
        \label{subsec: Classification; sec: Possible automorphism groups}

            We finally prove Theorem~\ref{the: possible Aut nice curve with simple Jacobian; sec: Possible automorphism groups}. We first consider the case in which $\Aut ( \Ccal )$ has even order.

            \begin{Lem}
            \label{lem: Aut even cardinality; subsec: Classification; sec: Possible automorphism groups}   
                Let $\Ccal / k$ be a curve of genus $\ge 2$ with simple Jacobian such that $\Aut ( \Ccal )$ has even order. Then $\Aut ( \Ccal )$ is isomorphic to one of the following groups:
                \begin{enumerate}[(i)]
                    \item the cyclic group $C_{2p}$ for some odd prime $p \ge 5$;
                    \item the cyclic group $C_{2^{n}}$ for some positive integer $n$;
                    \item the generalized quaternion group $G_{2^{n}}$ for some positive integer $n$.
                \end{enumerate}
            \end{Lem}
            \begin{proof}
                By Cauchy's theorem, there exists $\iota \in \Aut ( \Ccal )$ of order two. Since the Jacobian of $\Ccal$ is simple, we have $\Ccal / \langle \iota \rangle \cong \mathbb{P}_{1, k}$ (Theorem~\ref{the: necessary condition simplicity; sec: Preliminaries}); hence, $\Ccal$ is a hyperelliptic curve and $\iota$ is the hyperelliptic involution. 

                If $\Aut ( \Ccal ) / \langle \iota \rangle$ does not contain nontrivial elements of odd order, Theorem~\ref{the: classification Aut of hyperelliptic curves; sec: Preliminaries} implies that $\Aut ( \Ccal )$ is isomorphic to one of $C_{2} \times C_{2^{m}}$, $C_{2^{m+1}}$, $C_{2} \times D_{2^{m}}$, $V_{2^{m}}$, $D_{2^{m+1}}$, $H_{2^{m}}$, $U_{2^{m}}$, $G_{2^{m}}$ for some non-negative integer $m$. Since $\Aut ( \Ccal )$ contains at most one element of order $2$ (Proposition~\ref{pro: unique element of order 2 in AutC if C is simple; sec: Preliminaries}), if $m > 0$ we exclude cases $C_{2} \times C_{2^{m}}$, $C_{2} \times D_{2^{m}}$, $V_{2^{m}}$, $D_{2^{m+1}}$, $H_{2^{m}}$, $U_{2^{m}}$.

                If $\Aut ( \Ccal ) / \langle \iota \rangle$ contains an element of order $3$, there exists an element of $\Aut ( \Ccal )$ of order $6$, since $\iota$ is central. Propositions~\ref{pro: a curve with simple Jacobian is superelliptic; sec: Preliminaries} and~\ref{pro: genus superelliptic curve pq; sec: Preliminaries} imply that the genus of $\Ccal$ is $2$. From Theorem~\ref{the: classification Aut of hyperelliptic curves; sec: Preliminaries} it follows that $\Aut ( \Ccal )$ is isomorphic to one of $C_{2} \times C_{6}$, $C_{2} \times D_{3}$, $V_{2}$, $\GL ( 2, \Fbb_{3} )$. Since $\Aut ( \Ccal )$ contains at most one element of order $2$ (Proposition~\ref{pro: unique element of order 2 in AutC if C is simple; sec: Preliminaries}), we exclude all cases.

                If $\Aut ( \Ccal ) / \langle \iota \rangle$ contains an element of order $p \ge 5$ for some prime $p$, there exists an element of $\Aut ( \Ccal )$ of order $2p$, since $\iota$ is central. Propositions~\ref{pro: a curve with simple Jacobian is superelliptic; sec: Preliminaries} and~\ref{pro: genus superelliptic curve pq; sec: Preliminaries} imply that the genus of $\Ccal$ is either $( p - 1 ) / 2$ or $p - 1$. Theorem~\ref{the: classification Aut of hyperelliptic curves; sec: Preliminaries} implies that $\Aut ( \Ccal )$ is isomorphic to one of $C_{2p}$, $C_{2} \times C_{2p}$, $C_{2} \times D_{p}$, $C_{2} \times A_{5}$, $\SL ( 2, \Fbb_{5} )$. Since $\Aut ( \Ccal )$ contains at most one element of order $2$ (Proposition~\ref{pro: unique element of order 2 in AutC if C is simple; sec: Preliminaries}), we exclude cases $C_{2} \times C_{2p}$, $C_{2} \times D_{p}$, $C_{2} \times A_{5}$. Moreover, we exclude the case $\SL ( 2, \Fbb_{5} )$, which can be realized only by hyperelliptic curves of genus at least $14$ (Theorem~\ref{the: classification Aut of hyperelliptic curves; sec: Preliminaries}), because we have already proved that $g ( \Ccal ) \le p - 1$ (note that in this case we have $p = 5$).
            \end{proof}

            Next, we consider the case in which $\Aut ( \Ccal )$ has odd order.

            \begin{Lem}
            \label{lem: AutC is abelian if odd; subsec: Classification; sec: Possible automorphism groups}
                Let $\Ccal / k$ be a curve of genus $\ge 2$ with simple Jacobian. If $\Aut ( \Ccal )$ has odd order, it is abelian.
            \end{Lem}
            \begin{proof}
                Assume, for contradiction, that $\Aut ( \Ccal )$ is not abelian. Since $\Aut ( \Ccal )$ has odd order, the Feit--Thompson theorem implies that it is solvable. Recall that $\Aut ( \Ccal )$ admits irreducible, fixed-point-free representations (Theorem~\ref{the: AutC admits irr fixed point free rep; subsec: Representation theory; sec: Preliminaries}). The classification of finite solvable groups which admit irreducible, fixed-point-free representations, see \cite[Theorem~$6.1.11$]{WolfBookFixRep}, implies that $\Aut ( \Ccal )$ is isomorphic to a nontrivial semidirect product
                \begin{equation*}
                    \Aut ( \Ccal ) \cong C_{m} \rtimes_{\psi} C_{n} ,
                \end{equation*}
                for some positive integers $m, n$, and some homomorphism $\psi \colon C_{n} \to \Aut ( C_{m} )$ such that, letting $r \coloneqq \psi ( {1_{C_{n}}} ) \in ( \Z / m \Z )^{*} \cong \Aut ( C_{m} )$, we have:
                \begin{enumerate}[(i)]
                    \item $\GCD \bigl( n, \varphi(m) \bigr) > 1$;
                    \item $\GCD \bigl( n ( r - 1 ), m \bigr) = 1$;
                    \item $r^{n}\equiv 1 \pmod{m}$;
                    \item every prime divisor of the order $d$ of $r$ in $\Aut ( C_{m} )$ divides $n / d$. In particular, $d < n$.
                \end{enumerate}

                Since $\Aut ( \Ccal )$ does not contain an element whose order is the product of three distinct primes (Theorem~\ref{the: Aut does not contain elements order pqr; sec: Preliminaries}), it follows that $m = p_{1}^{a_{1}} p_{2}^{a_{2}}$ and $n = q_{1}^{b_{1}} q_{2}^{b_{2}}$ for some distinct primes $p_{1}, p_{2}$, some distinct primes $q_{1}, q_{2}$, and some non-negative integers $a_{1}, a_{2}, b_{1}, b_{2}$. Without loss of generality, we can assume $a_{1}, b_{1} > 0$.

                Since every prime that divides $d$ also divides $n / d$, by (iv), we deduce that $q_{1}$ divides $n / d$, and also $q_{2}$ divides $n / d$ if $b_{2} > 0$. Moreover, since $d \cdot C_{n}$ is in the kernel of $\psi$ by definition of $d$, the subgroup $C_{m} \rtimes_{\psi} ( d \cdot C_{n} )$ of $\Aut ( \Ccal )$ is isomorphic to the direct product $C_{m} \times C_{n/d} \cong C_{mn/d}$ (note that $( m, n / d ) = 1$ by (ii)). Since $\Aut ( \Ccal )$ does not contain an element of order a product of three distinct primes (Theorem~\ref{the: Aut does not contain elements order pqr; sec: Preliminaries}), we deduce that $mn / d$ cannot be divisible by three distinct primes; hence, at least two of the exponents $a_{1}, a_{2}, b_{1}, b_{2}$ must vanish, so $a_{2} = b_{2} = 0$.
  
                It follows that $\Aut ( \Ccal ) \cong C_{p_{1}^{a_{1}}} \rtimes_{\psi} C_{q_{1}^{b_{1}}}$. Consider the elements $x = ( p_{1}^{a_{1} - 1}, n / q_{1} )$ and $y = ( p_{1}^{a_{1} - 1}, 0_{C_{q_{1}^{b_{1}}}} )$ of $\Aut ( \Ccal )$. Since $d$ divides $n / q_{1}$, by (iv), it follows that $x^{k} = ( k p_{1}^{a_{1} - 1}, k n / q_{1} )$; hence, $x$ has order $p_{1} q_{1}$. Moreover, $\langle y \rangle$ is normal in $\Aut ( \Ccal )$, $y$ has order $p_{1}$ and $\Aut ( \Ccal ) / \langle y \rangle$ has a subgroup isomorphic to $C_{q_{1}^{b_{1}}}$. Since $d$ divides $n / q_{1}$ and $d \neq 1$, it follows that $b_{1} \ge 2$, which contradicts Theorem~\ref{the: Aut does not have particular cyclic subgroups; sec: Preliminaries}.
            \end{proof}
            
            \begin{Lem}
            \label{lem: Aut odd cardinality; subsec: Classification; sec: Possible automorphism groups}   
                Let $\Ccal / k$ be a curve of genus $\ge 2$ with simple Jacobian such that $\Aut ( \Ccal )$ has odd order. Then $\Aut ( \Ccal )$ is isomorphic to one of the following groups:
                \begin{enumerate}[(i)]
                    \item the cyclic group $C_{p^{n}}$ for some odd prime $p$ and some positive integer $n$;
                    \item the cyclic group $C_{pq}$ for some distinct odd primes $p, q$;
                    \item the trivial group.
                \end{enumerate}
            \end{Lem}
            \begin{proof}
                We proved in Lemma~\ref{lem: AutC is abelian if odd; subsec: Classification; sec: Possible automorphism groups} that $\Aut ( \Ccal )$ is abelian. Since all irreducible representations of abelian groups are one-dimensional (recall that we only consider representations over the algebraically closed field $k$), an abelian group admits irreducible, fixed-point-free representations if and only if it is cyclic. Since $\Aut ( \Ccal )$ admits irreducible, fixed-point-free representations (Theorem~\ref{the: AutC admits irr fixed point free rep; subsec: Representation theory; sec: Preliminaries}), we deduce that $\Aut ( \Ccal )$ is cyclic. Since $\Aut ( \Ccal )$ does not contain an element of order a product of three distinct primes (Theorem~\ref{the: Aut does not contain elements order pqr; sec: Preliminaries}), it follows that $\Aut ( \Ccal ) \cong C_{p_{1}^{a_{1}}p_{2}^{a_{2}}}$ for some distinct odd primes $p_{1}, p_{2}$ and some non-negative integers $a_{1}, a_{2}$. 
                
                Cases $(i)$ and $(iii)$ correspond to $a_{1} = 0$ or $a_{2} = 0$. Let us assume $a_{1} a_{2} \neq 0$. Applying Theorem~\ref{the: Aut does not have particular cyclic subgroups; sec: Preliminaries} to an element of order $p_{1} p_{2}$ and to an element of order $p_{1}$ of $\Aut ( \Ccal )$, we deduce that $a_{2} = 1$. Similarly, $a_{1} = 1$; hence, $\Aut ( \Ccal )$ falls into case (ii).
            \end{proof}

    \section{Cyclic groups of order \texorpdfstring{$p^{n}$}{a power of p}}
    \label{sec: Cyclic groups of order p^n}

        In this section, we analyze the cyclic groups $C_{p^{n}}$, where $p$ is a prime and $n$ is a positive integer. We show that these groups can be realized as the automorphism groups of curves over $\Qbar$ of genus $\ge 2$ with simple Jacobians.
        
        In a series of papers \cite{Zarhin2000, Zarhin2005, Zarhin2009, Zarhin2018, Zarhin2026}, Zarhin developed several ideas to treat the case $n = 1$. He obtained the following result:
        
        \begin{The}
        \label{the: Zarhin; sec: Cyclic groups of order p^n}
            Let $p$ be a prime, $K \supseteq \Q ( \zeta_{p} )$ be a number field and $f \in K [ x ]$ be an irreducible polynomial of degree $L \ge 5$. 
            If the Galois group of $f$ is isomorphic to either $A_{L}$ or $S_{L}$, the curve over $K$ given by
            \begin{equation*}
                y^{p} = f(x)
            \end{equation*}
            has geometric automorphism group isomorphic to $C_{p}$, and its Jacobian is geometrically simple with geometric endomorphism ring isomorphic to $\Z [ \zeta_{p} ]$ and geometric endomorphism algebra isomorphic to $\Q ( \zeta_{p} )$. 
        \end{The}
        \begin{proof}
            The statement regarding the geometric endomorphism ring and the geometric endomorphism algebra was established in \cite[Theorem~$2.1$]{Zarhin2000} (for the case $p = 2$) and in \cite[Theorem~$1.3$, Remark~$1.5$]{Zarhin2026} (for the case $p > 2$). The statement regarding the geometric automorphism group follows easily; see the proof of Theorem~\ref{the: reduction first statement main the; subsec: Preliminaries; sec: Cyclic groups of order p^n}.
        \end{proof}

        \begin{Rem}
        \label{rem: impossibility of applying Zarhin method; sec: Cyclic groups of order p^n}
            As anticipated in Section~\ref{subsec: Methods; sec: Introduction}, when attempting to apply Zarhin's method to the Jacobians of the curves appearing in Theorem~\ref{the: main; sec: Cyclic groups of order p^n}, one encounters several difficulties, in particular regarding the application of \cite[Lemma~2.3]{Zarhin2026}. Indeed, Zarhin applies the lemma with $\Hcal$ being the endomorphism algebra of the Jacobian, $\Lambda$ the endomorphism ring of the Jacobian, $E = \Q ( \zeta_{p^{n}} )$ and $\mfrak = ( 1 - \zeta_{p^{n}} )$. Zarhin's assumptions on the superelliptic curves ensure that the $( 1 - \zeta_{p^{n}} )$-torsion groups of the Jacobians are central simple modules (see \cite[Definition~4.1]{Zarhin2026}) over the absolute Galois group of the number fields of definition of the superelliptic curves. This property ensures that $\Lambda / (\mfrak \cdot \Lambda )$ is a matrix algebra over $\Fbb_{p}$; hence, \cite[Lemma~2.3]{Zarhin2026} can be applied. In our setting, however, $\Lambda / (\mfrak \cdot \Lambda )$ is a product of matrix algebras over $\Fbb_{p}$, and a direct generalization of the lemma to this case is false.
        \end{Rem}
            
        Using different methods, we obtain the following result for the case $n > 1$:
            
        \begin{The}
        \label{the: main; sec: Cyclic groups of order p^n}
            Let $p$ be a prime, and let $n > 1$ and $L \ge 1$ be integers. If $p = 2$, assume in addition $L \ge 3$. For $\abf = ( a_{1}, \dots, a_{L} ) \in \Abb^{L}_{\Qbar} ( \Qbar )$, set $f_{\abf} ( x ) \coloneqq x \cdot \prod_{ l = 1 }^{L} \bigl( x^{ p^{ n - 1 } } -a_{l} \bigr)$, and let $U \subset \Abb^{L}_{\Qbar}$ be the non-empty open subscheme defined by $\operatorname{disc}_{x} ( f_{\abf} ) \neq 0$. For every $\abf \in U ( \Qbar )$, let $\Cfrak_{\abf} / \Qbar$ denote the smooth projective model of the affine curve 
            \begin{equation}
            \label{eq: Cfraka1...aL; sec: Cyclic groups of order p^n}
                y^{p} = f_{\abf} ( x ) = x \cdot \prod_{ l = 1 }^{L} \bigl( x^{ p^{ n - 1 } } - a_{l} \bigr) .
            \end{equation}
            For every number field $K$ and for $100 \%$ of the points $\abf \in U ( \Qbar ) \cap K^{L}$, one has $\Aut_{\Qbar} ( \Cfrak_{\abf} ) \cong C_{ p^{n} }$, and the Jacobian $\Jfrak_{\abf} \coloneqq \Jac ( \Cfrak_{\abf} )$ is (geometrically) simple, with $\End_{\Qbar} (\Jfrak_{\abf} ) \cong \Z [ \zeta_{ p^{n} } ]$. Moreover, if $L\ge 2$, the curves $\Cfrak_{\abf}$ obtained in this way represent infinitely many distinct isomorphism classes over $\Qbar$, and their Jacobians represent infinitely many distinct isogeny classes over $\Qbar$.
        \end{The}

        This section is organized as follows. In Section~\ref{subsec: Setup; sec: Cyclic groups of order p^n} we fix the notation, while in Section~\ref{subsec: Origin of the family; sec: Cyclic groups of order p^n} we provide some background for the curves \eqref{eq: Cfraka1...aL; sec: Cyclic groups of order p^n}. Section~\ref{subsec: Preliminaries; sec: Cyclic groups of order p^n} is devoted to analyzing several properties of these curves, while in Section~\ref{subsec: specialization methods; sec: Cyclic groups of order p^n} we extend the analysis using specialization techniques and degenerations of curves. Finally, in Sections~\ref{subsec: The case p>2; sec: Cyclic groups of order p^n} and~\ref{subsec: The case p=2; sec: Cyclic groups of order p^n}, we prove Theorem~\ref{the: main; sec: Cyclic groups of order p^n} by distinguishing the cases $p > 2$ and $p = 2$.

        \subsection{Setup}
        \label{subsec: Setup; sec: Cyclic groups of order p^n}

            We begin with a brief setup. Throughout Section~\ref{sec: Cyclic groups of order p^n}, we fix a prime $p$ and an integer $n \ge 2$. Moreover, we denote by $L$ an integer $\ge 1$ and by $\alpha_{1}, \dots, \alpha_{L}$ the variables of $\Abb_{\Qbar}^{L}$.

            \begin{Def}
            \label{def: Cfrak, Jfrak, Cbf, Jbf; subsec: Setup; sec: Cyclic groups of order p^n}
                Define the family $\Cfrak \to U$ of curves over $\Qbar$ given by the relative smooth projective model of
                \begin{equation}
                \label{eq: Cfrak; subsec: Setup; sec: Cyclic groups of order p^n}
                    y^{p} = x \cdot \prod_{l = 1}^{L} \bigl( x^{p^{n-1}} - \alpha_{l} \bigr) ,
                \end{equation}
                where $U$ is the non-empty open subscheme of $\Abb_{\Qbar}^{L}$ on which the discriminant (with respect to $x$) of the polynomial on the right-hand side does not vanish. Let $\Jfrak \to U$ be the Jacobian scheme of $\Cfrak \to U$.

                Moreover, let $\Cbf$ (resp.~$\Jbf$) be the generic fiber over $\Qbar ( U )$ of $\Cfrak \to U$ (resp.~$\Jfrak \to U$), and let $\Cbf^{\mathrm{al}}$ (resp.~$\Jbf^{\mathrm{al}}$) be the scalar extension to an algebraic closure $\Qbar ( U )^{\mathrm{al}}$ of $\Qbar ( U )$. 
            \end{Def}

            \begin{Def}
            \label{def: varphi and Phi; subsec: Setup; sec: Cyclic groups of order p^n}
                Define the automorphism of $\Cbf$ given by
                \begin{equation}
                \label{eq: phi; subsec: Setup; sec: Cyclic groups of order p^n} 
                    \varphi \colon \Cbf \to \Cbf \qquad\qquad\qquad ( x, y ) \mapsto ( \zeta_{p^{n-1}} \cdot x, \zeta_{p^{n}} \cdot y ) ,
                \end{equation}
                which generates a subgroup of $\Aut ( \Cbf )$ isomorphic to $C_{p^{n}}$, and let $\Phi$ be the endomorphism of $\Jbf$ induced by $\varphi$.
            \end{Def}

        \subsection{Origin of the family}
        \label{subsec: Origin of the family; sec: Cyclic groups of order p^n}
        
            We explain the origin of the family \eqref{eq: Cfrak; subsec: Setup; sec: Cyclic groups of order p^n}. Let $\Ccal_{n}$ be a curve over $\Qbar$ of genus $\ge 2$ with simple Jacobian and automorphism group isomorphic to $C_{p^{n}}$. We know that $\Ccal_{n}$ is isomorphic to a superelliptic curve given by $y^{p^{n}} = f(x)$ for some $f \in \Qbar [ x ]$ (Proposition~\ref{pro: a curve with simple Jacobian is superelliptic; sec: Preliminaries}). For each integer $1 \le k < n$, consider the superelliptic curve $\Ccal_{k}$ given by $y^{p^{k}} = f(x)$. The existence of the nonconstant morphisms $\Ccal_{n} \to \Ccal_{k}$ given by $( x, y ) \mapsto ( x, y^{p^{n-k}} )$, together with the assumption that the Jacobian of $\Ccal_{n}$ is simple, implies that $\Ccal_{k}$ has genus $0$ for every integer $1 \le k < n$ (Lemma~\ref{lem: necessary condition simplicity; sec: Preliminaries}).
                
            For each integer $1 \le k \le n$, denote by $m_{ z, p^{k} }$ the corresponding value of $m_{z}$ for $\Ccal_{k}$ (Notation~\ref{not: mz; sec: Preliminaries}), and let $\Acal_{ s, p^{k} }$ be the number of $z \in \mathbb{P}_{1}(\Qbar)$ such that $\GCD ( p^{k}, m_{ z, p^{k} } ) = p^{k} / s$. For each $z \in \mathbb{P}_{1}(\Qbar)$ that is either a zero of $f$ or $\infty$, the value $m_{ z, p^{k} }$ is independent of $k$. It is easy to check that, for each integer $1 \le k \le n$ and for each divisor $s$ of $p^{k}$, we have $\Acal_{ s, p^{k} } = \Acal_{ s \cdot p^{n-k}, p^{n} }$.

            Applying Lemma~\ref{lem: genus superelliptic curves; sec: Preliminaries} to $\Ccal_{1}$ and using the fact that the genus of $\Ccal_{1}$ is zero, we deduce that $\Acal_{ p^{n}, p^{n} } = 2$. Proceeding similarly with $\Ccal_{2}$, we deduce that $\Acal_{ p^{n-1}, p^{n} } = 0$. By continuing in this manner up to $\Ccal_{n-1}$, we obtain $\Acal_{ p^{2}, p^{n} } = \cdots = \Acal_{ p^{n-1}, p^{n} } = 0$.

            It follows that the cover $\Ccal_{n} \to \mathbb{P}_{1, \Qbar}$ given by $( x, y ) \mapsto x$ has two branch points with branch index $p^{n}$, and $L = \Acal_{ p, p^{n} }$ branch points with branch index $p$. Since $\Aut ( \mathbb{P}_{1, \Qbar} )$ is triply transitive on $\mathbb{P}_{1, \Qbar}$, we may assume that the two branch points with branch index $p^{n}$ are $0$ and $\infty$. Hence, there exist $L$ distinct elements $a_{l} \in \Qbar^{*}$, an integer $0 < b < p^{n}$ coprime to $p$, and $L$ integers $0 < b_{l} < p$, such that $\Ccal_{n}$ is isomorphic to the superelliptic curve given by
            \begin{equation}
            \label{eq: preliminary model; subsec: Setup; sec: Cyclic groups of order p^n}
                y^{p^{n}} = x^{b} \cdot \biggl( \prod_{l = 1}^{L} ( x -a_{l} )^{b_{l}} \biggr)^{p^{n-1}} .
            \end{equation}

            Let us consider the simplest case, namely assume $b = b_{1} = \cdots = b_{L} = 1$, so that $\Ccal_{n}$ is given by $y^{p^{n}} = x \cdot g(x)^{p^{n-1}}$, where $g(x) = ( x - a_{1} ) \cdots ( x - a_{L} )$. Define the rational function $z = y^{p} / g(x)$ on $\Ccal_{n}$. Since $z^{p^{n-1}} = x$, it follows that $\Ccal_{n}$ is isomorphic to the superelliptic curve given by $y^{p} = z \cdot g(x) = z \cdot g( z^{p^{n-1}} )$; hence, $\Ccal_{n}$ is a fiber of $\Cfrak \to U$.

            \begin{Rem}
            \label{rem: genus of Cbf; subsec: Setup; sec: Cyclic groups of order p^n}
                By Lemma~\ref{lem: genus superelliptic curves; sec: Preliminaries}, $\Cfrak \to U$ is a family of curves of genus $\varphi ( p^{n} ) \cdot L / 2$ and $\Jfrak \to U$ is an abelian scheme of relative dimension $\varphi ( p^{n} ) \cdot L / 2$, where $\varphi$ denotes Euler's totient function.
            \end{Rem}

        \subsection{Preliminaries}
        \label{subsec: Preliminaries; sec: Cyclic groups of order p^n}
            
            We continue by analyzing the families $\Cfrak \to U$ and $\Jfrak \to U$. First, we recall a description of a basis of $H^{0} ( \Cbf^{\mathrm{al}}, \Omega^{1}_{ \Cbf^{\mathrm{al}} } )$. Note that the polynomial (with respect to $x$) on the right-hand side of \eqref{eq: Cfrak; subsec: Setup; sec: Cyclic groups of order p^n} has simple roots, so we can apply \cite[Proposition~$2$]{TowseCyclicCoverP1}. 

            \begin{Lem}
            \label{lem: basis of regular differential forms on Cbf; subsec: Preliminaries; sec: Cyclic groups of order p^n}
                The differential forms
                \begin{equation*}
                    \omega_{i, j} \coloneqq \frac{x^{i} \cdot dx}{y^{j}} , \qquad\qquad\qquad \text{for } 1 \le j \le p - 1 \text{ and } 0 \le i \le j \cdot L \cdot p^{n-2} - 1 ,
                \end{equation*}
                form a basis of $H^{0} ( \Cbf^{\mathrm{al}}, \Omega^{1}_{ \Cbf^{\mathrm{al}} } )$.
            \end{Lem}
            \begin{proof}
                Apply \cite[Proposition~$2$]{TowseCyclicCoverP1} with $d = 1 + L \cdot p^{n-1}$, $n = p$, $e = p - 1$ and $k = L \cdot p^{n-2} + 1$, after replacing $\C$ with $\Qbar(U)^{\mathrm{al}}$.
            \end{proof}

            The action of $\varphi$ (Definition~\ref{def: varphi and Phi; subsec: Setup; sec: Cyclic groups of order p^n}) on $H^{0} ( \Cbf^{\mathrm{al}}, \Omega^{1}_{ \Cbf^{\mathrm{al}}} )$ via pullback provides information about $\End^{0} ( \Jbf^{\mathrm{al}} )$.

            \begin{Lem}
            \label{lem: End0Jbf contains Q(zetapn); subsec: Preliminaries; sec: Cyclic groups of order p^n}
                The $\Q$-subalgebra of $\End^{0}  ( \Jbf^{\mathrm{al}} )$ generated by $\Phi$ is isomorphic to $\Q ( \zeta_{p^{n}} )$. Moreover, $\Q ( \Phi )$ injects into the endomorphism algebra of every isotypic component of $\Jbf^{\mathrm{al}}$. 
            \end{Lem}
            \begin{proof}
                Let $P$ (resp.~$Q_{n}$) be the minimal polynomial of $\Phi$ (resp.~$\zeta_{p^{n}}$) over $\Q$. The action of $\varphi$ on $H^{0} ( \Cbf^{\mathrm{al}}, \Omega^{1}_{ \Cbf^{\mathrm{al}}} )$ via pullback is given by $\omega_{i, j} \mapsto \zeta_{ p^{n} }^{ ( i + 1 ) p - j } \cdot \omega_{i, j}$. Since $( i + 1 ) p - j$ is always coprime to $p$, as $1 \le j \le p - 1$ (Lemma~\ref{lem: basis of regular differential forms on Cbf; subsec: Preliminaries; sec: Cyclic groups of order p^n}), it follows that $Q_{n} ( \zeta_{ p^{n} }^{ ( i + 1 ) p - j } ) = 0$ for every $1 \le j \le p - 1$ and $0 \le i \le j \cdot L \cdot p^{n-2} - 1$; hence $Q_{n} ( \Phi ) = 0$ and $P \mid Q_{n}$. On the other hand, since $P ( \Phi ) = 0$ by definition, by choosing $i = 0$ and $j = p - 1$ we have $P ( \zeta_{p^{n}} ) = 0$, which implies $Q_{n} \mid P$. So, $\Q ( \Phi ) \cong \Q ( \zeta_{p^{n}} )$, as desired.
                    
                Finally, since $\End^{0} ( \Jbf^{\mathrm{al}} )$ projects onto the endomorphism algebra of every isotypic component of $\Jbf^{\mathrm{al}}$, there exists a morphism from $\Q ( \Phi )$ to the endomorphism algebra of every isotypic component of $\Jbf^{\mathrm{al}}$. Since $\Q ( \Phi )$ is a field, this morphism is either trivial or injective. As $\Phi$ acts nontrivially on every isotypic component of $\Jbf^{\mathrm{al}}$, it is injective.
            \end{proof}

            We are interested in $\End^{0} ( \Jbf^{\mathrm{al}} )$ because of the following theorem:
                
            \begin{The}
            \label{the: reduction first statement main the; subsec: Preliminaries; sec: Cyclic groups of order p^n}
                If $\End^{0} ( \Jbf^{\mathrm{al}} )$ is isomorphic to $\Q ( \zeta_{p^{n}} )$, then for every number field $K$ and for $100 \%$ of the points $( a_{1},  \ldots, a_{L} ) \in U ( \Qbar ) \cap K^{L}$, the smooth projective model over $\Qbar$ of the affine curve
                \begin{equation*}
                    y^{p} = x \cdot \prod_{l = 1}^{L} \bigl( x^{p^{n-1}} - a_{l} \bigr)
                \end{equation*}
                has automorphism group isomorphic to $C_{ p^{n} }$, and its Jacobian is (geometrically) simple, with endomorphism ring isomorphic to  $\Z [ \zeta_{ p^{n} } ]$.
            \end{The}
            \begin{proof}
                Note that $\Cfrak \to U$ is the base change to $\Qbar$ of a family defined over $\Q$. Lemma~\ref{lem: how many fibers of Jfrak have bigger End0; subsec: specialization methods; sec: Preliminaries} implies that for every number field $K$ and for $100\%$ of the points $( a_{1},  \ldots, a_{L} )$ in $U ( \Qbar ) \cap K^{L}$, we have that $\End^{0} ( \Jfrak_{ ( a_{1}, \dots, a_{L} ) } )$ is isomorphic to $\End^{0} ( \Jbf^{\mathrm{al}} ) \cong \Q ( \zeta_{p^{n}} )$. 

                Let $\tilde{\varphi}$ be the specialization of $\varphi$ to $\Cfrak_{ ( a_{1}, \dots, a_{L} ) }$, and note that $\tilde{\varphi}$ generates a subgroup of $\Aut ( \Cfrak_{ ( a_{1}, \dots, a_{L} ) } )$ isomorphic to $C_{p^{n}}$. From \cite[Theorem~$12.1$, \S~III]{milneAV} it follows that $\Aut ( \Cfrak_{ ( a_{1}, \dots, a_{L} ) } )$ injects into the group of units of finite order of $\End^{0} ( \Jfrak_{ ( a_{1}, \dots, a_{L} ) } ) \cong \Q ( \zeta_{p^{n}} )$. 
                
                If $p = 2$, the group of units of finite order of $\End^{0} ( \Jfrak_{ ( a_{1}, \dots, a_{L} ) } )$ is isomorphic to $C_{2^{n}}$; hence, $\Aut ( \Cfrak_{ ( a_{1}, \dots, a_{L} ) } )$ is isomorphic to $C_{2^{n}}$. If $p > 2$, the group of units of finite order of $\End^{0} ( \Jfrak_{ ( a_{1}, \dots, a_{L} ) } )$ is isomorphic to $C_{2p^{n}}$; hence $\Aut ( \Cfrak_{ ( a_{1}, \dots, a_{L} ) } )$ is isomorphic either to $C_{p^{n}}$ or $C_{2p^{n}}$. Since the Jacobian of $\Cfrak_{ ( a_{1}, \dots, a_{L} ) }$ is simple and $n > 1$, $\Aut ( \Cfrak_{ ( a_{1}, \dots, a_{L} ) } )$ cannot be isomorphic to $C_{2p^{n}}$ (Theorem~\ref{the: possible Aut nice curve with simple Jacobian; sec: Possible automorphism groups}), so it is isomorphic to $C_{p^{n}}$, as desired.
                
                Finally, by definition $\End ( \Jfrak_{ ( a_{1}, \dots, a_{L} ) } )$ is an order of $\End^{0} ( \Jfrak_{ ( a_{1}, \dots, a_{L} ) } )$. Since $\tilde{\Phi}$ generates a subring of $\End ( \Jfrak_{ ( a_{1}, \dots, a_{L} ) } )$ isomorphic to $\Z [ \zeta_{p^{n}} ]$, which is the unique maximal order of $\Q ( \zeta_{p^{n}} )$, we deduce that $\End ( \Jfrak_{ ( a_{1}, \dots, a_{L} ) } )$ is isomorphic to $\Z [ \zeta_{p^{n}} ]$.
            \end{proof}

            The following theorem shows that the map to moduli space induced by the family $\Cfrak \to U$ is nonconstant.
                
            \begin{The}
            \label{the: there are infinitely many classes; subsec: Preliminaries; sec: Cyclic groups of order p^n}
                The image of the map to the moduli space induced by the family $\Cfrak \to U$ has dimension $L - 1$. Moreover, the second statement of Theorem~\ref{the: main; sec: Cyclic groups of order p^n} holds.
            \end{The}
            \begin{proof}
                Consider the morphism $\Abb_{\Qbar}^{L - 1} \to \Abb_{\Qbar}^{L}$ given by $( \alpha_{2}, \dots, \alpha_{L} ) \mapsto ( 1, \alpha_{2}, \dots, \alpha_{L} )$, and let $\Cfrak' \to U'$ be the pullback of $\Cfrak \to U$ via this morphism:
                \begin{equation*}
                    \begin{tikzcd}
                        \Cfrak' \arrow{r} \arrow{d} & \Cfrak \arrow{d} \\ 
                        U' \arrow{r} & U
                    \end{tikzcd}
                \end{equation*} 
                Explicitly, $\Cfrak' \to U'$ is given by
                \begin{equation*}
                    y^{p} = x \cdot ( x^{p^{n-1}} - 1 ) \cdot \prod_{l = 2}^{L} ( x^{p^{n-1}} - \alpha_{l} ) .
                \end{equation*}
                Every fiber of $\Cfrak \to U$ over $\Qbar$-points of $U$ is isomorphic over $\Qbar$ to a fiber of $\Cfrak' \to U'$ over some $\Qbar$-point of $U'$; hence, the image of the map to the moduli space induced by the family $\Cfrak \to U$ has dimension at most $L - 1$. 
                
                If every isomorphism class of curves over $\Qbar$ is represented at most finitely many times by the fibers of $\Cfrak' \to U'$ over $\Qbar$-points of $U'$, the image of the map to the moduli space induced by the family $\Cfrak \to U$ has dimension at least $L - 1$, and therefore exactly $L - 1$. Moreover, \cite[Proposition~6.6.1]{cantoral2023monodromy} applied to $\Cfrak' \to U'$ implies that the fibers of $\Jfrak' \to U'$ over $\Qbar$-points of $U'$ fall into infinitely many distinct isogeny classes of abelian varieties over $\Qbar$; hence, the second statement of Theorem~\ref{the: main; sec: Cyclic groups of order p^n} holds.
                
                Assume, by way of contradiction, that there exists $a \in U' ( \Qbar )$ such that the set
                \begin{equation*}
                    X = \bigl\{ a' \in U' ( \Qbar ) \,\colon\, \Cfrak'_{a} \cong \Cfrak'_{a'} \bigr\}
                \end{equation*}
                is infinite. For every $a' \in X$, choose an isomorphism $\psi_{a'} \colon \Cfrak'_{a} \to \Cfrak'_{a'}$. Since the field extension $\Qbar ( x ) \subset \Qbar ( \Cfrak'_{a'} )$ is Galois and $\psi_{a'}$ is an isomorphism, the field extension $\psi_{a'}^{*} \bigl( \Qbar ( x ) \bigr) \subset \psi_{a'}^{*} \bigl( \Qbar ( \Cfrak'_{a'} ) \bigr)$ is also Galois. Since $\psi_{a'}^{*} \bigl( \Qbar ( \Cfrak'_{a'} ) \bigr) = \Qbar ( \Cfrak'_{a} )$, there exists a subgroup $H$ of $\Aut ( \Cfrak'_{a} )$ such that $\psi_{a'}^{*} \bigl( \Qbar ( x ) \bigr) = \Qbar ( \Cfrak'_{a} )^{H}$.

                Let $H_{1}, H_{2}, \ldots, H_{M}$ be the subgroups of $\Aut ( \Cfrak'_{a} )$. For every integer $1 \le m \le M$ define the set
                \begin{equation*}
                    X_{m} = \bigl\{ a' \in X \,\colon\, \psi_{a'}^{*} \bigl( \Qbar(x) \bigr) = \Qbar ( \Cfrak'_{a} )^{H_{m}} \bigr\} ,
                \end{equation*}
                and note that $X = \bigcup_{m} X_{m}$. Since this union is indexed by a finite set and $X$ is infinite, there exists an integer $1 \le k \le M$ such that $X_{k}$ is infinite.

                Choose a point $a' = ( a_{2}', \ldots, a_{L}' ) \in X_{k}$ and consider another point $a'' = ( a''_{2}, \ldots, a''_{L} ) \in X_{k}$. Define the polynomials
                \begin{align*}
                    f_{a'}(x) &= x \cdot ( x^{p^{n-1}} - 1 ) \cdot \prod_{l = 2}^{L} ( x^{p^{n-1}} - a_{l}' ) , \\
                    f_{a''}(x) &= x \cdot ( x^{p^{n-1}} - 1 ) \cdot \prod_{l = 2}^{L} ( x^{p^{n-1}} - a''_{l} ) ,
                \end{align*}
                in $\Qbar [ x ]$ and note that $\Qbar ( \Cfrak'_{a'} ) = \Qbar ( x ) \bigl( \sqrt[p]{ f_{a'}(x) } \bigr)$ and $\Qbar ( \Cfrak'_{a''} ) = \Qbar ( x ) \bigl( \sqrt[p]{ f_{a''}(x) } \bigr)$. Since the automorphism $( \psi_{a'} \circ \psi_{a''}^{-1} )^{*} \colon \Qbar ( \Cfrak'_{a'} ) \to \Qbar ( \Cfrak'_{a''} )$ sends the field $\Qbar ( x )$ of $\Qbar ( \Cfrak'_{a'} )$ into the field $\Qbar ( x )$ of $\Qbar ( \Cfrak'_{a''} )$, there exists an automorphism $\eta \in \Aut \bigl( \mathbb{P}_{1, \, \Qbar} \bigr)$ such that $\Qbar ( \Cfrak'_{a''} ) = \Qbar ( x ) \Bigl( \sqrt[p]{ f_{a'} \bigl( \eta(x) \bigr) } \Bigr)$. Kummer theory implies that there exist $u \in \Qbar ( x )$ and a positive integer $r$ coprime to $p$ such that
                \begin{equation*}
                    f_{a'} \bigl( \eta(x) \bigr) = u(x)^{p} \cdot f_{a''}(x)^{r} .
                \end{equation*}

                This means that $\eta$ sends the divisor of $f_{a'}$ modulo $p$ to the divisor of $f_{a''}^{r}$ modulo $p$; hence, $\eta$ establishes a bijection between the sets
                \begin{equation*}
                    S' = \{ 0, \infty \} \cup \{ \zeta_{p^{n-1}}^{i} \,\colon\, 1 \le i \le p^{n-1} \} \cup \bigl\{ \zeta_{p^{n-1}}^{i} \cdot ( a_{l}' )^{1 / p^{n-1}} \,\colon\, 1 \le i \le p^{n-1} \text{ and } 2 \le l \le L \bigr\}
                \end{equation*}
                and
                \begin{equation*}
                    S'' = \{ 0, \infty \} \cup \{ \zeta_{p^{n-1}}^{i} \,\colon\, 1 \le i \le p^{n-1} \} \cup \bigl\{ \zeta_{p^{n-1}}^{i} \cdot ( a''_{l} )^{1 / p^{n-1}} \,\colon\, 1 \le i \le p^{n-1} \text{ and } 2 \le l \le L \bigr\} .
                \end{equation*}
                     
                To sum up, to every $a'' \in X_{k}$ we can associate an automorphism $\eta = \eta ( a'' )$ of $\mathbb{P}_{1, \Qbar}$ such that $\eta ( S' ) = S''$. The inverse image of an automorphism of $\mathbb{P}_{1, \Qbar}$ under this association is finite. Indeed, assume by way of contradiction that there exists $a'' = ( a_{2}'', \dots, a_{L}'' ) \in X_{k}$ such that the set $\bigl\{ a''' = ( a_{2}''', \dots, a_{L}''' ) \in X_{k} \,\colon\, \eta ( a''' ) = \eta ( a'' ) \bigr\}$ is infinite. Then, the set $\{ a''' \in X_{k} \,\colon\, S''' = S'' \}$ is also infinite. Note that $S''' = S''$ implies $(S''')^{p^{n-1}} = (S'')^{p^{n-1}}$, where $(S''')^{p^{n-1}}$ (resp.~$(S'')^{p^{n-1}}$) denotes the set obtained by raising every element of $S'''$ (resp.~$S''$) to the power of $p^{n-1}$. Ignoring repetitions of elements, we have $(S''')^{p^{n-1}} = \{ 0, \infty, 1,  a_{2}''', \dots, a_{L}''' \}$ and $(S'')^{p^{n-1}} = \{ 0, \infty, 1, a_{2}'', \dots, a_{L}'' \}$. Hence, $( a_{2}''', \dots, a_{L}''' )$ is a permutation of $( a_{2}'', \dots, a_{L}'' )$, which contradicts the assumption that the set $\{ a''' \in X_{k} \,\colon\, S''' = S'' \}$ is infinite.
                    
                Since $X_{k}$ is infinite, it follows that the set 
                \begin{equation*}
                    Y \coloneqq \bigl\{ g \in \Aut ( \mathbb{P}_{1, \Qbar} ) \,\colon\, 0, 1, \infty \in g(S') \bigr\} 
                \end{equation*}
                is infinite, since it contains $\{ \eta ( a'' ) : a'' \in X_{k} \}$, and we have already established that the map $a'' \mapsto \eta ( a'' )$ has finite fibers. For every choice of distinct elements $s_{1}, s_{2}, s_{3} \in S'$, consider the set
                \begin{equation*}
                    Y_{ s_{1}, s_{2}, s_{3} } \coloneqq \bigl\{ g \in \Aut ( \mathbb{P}_{1, \Qbar} ) \,\colon\, g ( s_{1} ) = 0, g ( s_{2} ) = 1 \text{ and } g ( s_{3} ) = \infty \bigr\} .
                \end{equation*}
                Since $Y = \bigcup_{ s_{1}, s_{2}, s_{3} \in S' } Y_{ s_{1}, s_{2}, s_{3} }$ and this union is indexed by a finite set, there exist distinct $\tilde{s}_{1}, \tilde{s}_{2}, \tilde{s}_{3} \in S'$ such that $Y_{ \tilde{s}_{1}, \tilde{s}_{2}, \tilde{s}_{3} }$ is infinite. However, there exists at most one automorphism of $\mathbb{P}_{1, \Qbar}$ sending a fixed ordered triple of distinct points of $\mathbb{P}_{1, \Qbar}$ to another. This yields a contradiction, as desired.
            \end{proof}

            \begin{Rem}
            \label{rem: if L = 1, there are finitely many classes; subsec: Preliminaries; sec: Cyclic groups of order p^n}
                The curve $y^{p} = x \cdot ( x^{p^{n-1}} - a^{-1} )$, where $a \in \Qbar^{*}$, is isomorphic over $\Qbar$ to the curve $y^{p} = x \cdot ( x^{p^{n-1}} - 1 )$ via the isomorphism $( x, y ) \mapsto \bigl( \sqrt[p^{n-1}]{a} \cdot x, \sqrt[p^{n}]{a^{p^{n-1}+1}} \cdot y \bigr)$.
            \end{Rem}

            An interesting consequence of Theorem~\ref{the: there are infinitely many classes; subsec: Preliminaries; sec: Cyclic groups of order p^n} is the following.

            \begin{Lem}
            \label{lem: Cbf is not of CM type if L > 1; subsec: Preliminaries; sec: Cyclic groups of order p^n}
                If $L \ge 2$, then $\Jbf^{\mathrm{al}}$ is not isogenous to a product of abelian varieties with CM. 
            \end{Lem}
            \begin{proof}
                Assume, for contradiction, that $\Jbf^{\mathrm{al}}$ is isogenous to a product of abelian varieties with CM by certain CM fields. Then, the first statement of Lemma~\ref{lem: End0 embeds and torsion bijects under specialization; subsec: specialization methods; sec: Preliminaries} implies that every fiber of $\Jfrak \to U$ over $\Qbar$-points of $U$ is isogenous to a product of abelian varieties with CM by the same fields, since the genus of $\Cbf^{\mathrm{al}}$ and of the fiber is the same.

                Since abelian varieties of fixed dimension and of fixed CM type are all isogenous \cite[Theorem on p.~$213$]{mumford1970abelian}, it follows that there are finitely many isogeny classes of abelian varieties of fixed dimension with CM by the same field. Hence, the fibers of $\Jfrak \to U$ over $\Qbar$-points of $U$ would fall into finitely many isogeny classes of abelian varieties over $\Qbar$, contradicting Theorem~\ref{the: there are infinitely many classes; subsec: Preliminaries; sec: Cyclic groups of order p^n}.
            \end{proof}

            One of the ingredients in the proof of Theorem~\ref{the: main; sec: Cyclic groups of order p^n} is the theory of abelian varieties with complex multiplication (CM). Recall that the set of field embeddings of $\Q ( \zeta_{p^{n}} )$ into $\Qbar$ is in bijection with the set of integers between $0$ and $p^{n}$ that are coprime to $p$. We use the symbol $\sigma$ to denote both a field embedding and its associated integer, and we write $\sigma = p \cdot \sigma_{1} - \sigma_{2}$ for $1 \le \sigma_{1} \le p^{n-1}$ and $1 \le \sigma_{2} \le p-1$.

            \begin{Lem}
            \label{lem: if L = 1 then Jbf is of CM type; subsec: Preliminaries; sec: Cyclic groups of order p^n}
                If $L = 1$, then $\Jbf^{\mathrm{al}}$ has CM by $\Q ( \zeta_{p^{n}} )$. Its CM type is 
                \begin{equation*}
                    T \coloneqq \bigl\{ p \cdot \sigma_{1} - \sigma_{2} \,\colon\, 1 \le \sigma_{2} \le p - 1 \text{ and } 1 \le \sigma_{1} \le \sigma_{2} \cdot p^{n-2} \bigr\} \subseteq \Gal \bigl( \Q ( \zeta_{p^{n}} ) / \Q \bigr) .
                \end{equation*}
            \end{Lem}
            \begin{proof}
                We know that $\dim ( \Jbf^{\mathrm{al}} ) = \varphi ( p^{n} ) / 2$ (Remark~\ref{rem: genus of Cbf; subsec: Setup; sec: Cyclic groups of order p^n}) and $\Q ( \zeta_{p^{n}} )$ embeds into $\End^{0} ( \Jbf^{\mathrm{al}} )$, by the first statement of Lemma~\ref{lem: End0Jbf contains Q(zetapn); subsec: Preliminaries; sec: Cyclic groups of order p^n}; hence,  $\Jbf^{\mathrm{al}}$ has CM by $\Q ( \zeta_{p^{n}} )$. The proof of Lemma~\ref{lem: End0Jbf contains Q(zetapn); subsec: Preliminaries; sec: Cyclic groups of order p^n} shows that the CM type of $\Jbf^{\mathrm{al}}$ is $T$.
            \end{proof}
                
            \begin{Lem}
            \label{lem: stabilizer of CM type; subsec: Preliminaries; sec: Cyclic groups of order p^n}
                The stabilizer $\Stab ( T )$ of the CM type $T \subseteq \Gal \bigl( \Q ( \zeta_{p^{n}} ) / \Q \bigr)$ is
                \begin{equation*}
                    \Stab ( T ) = 
                    \begin{cases}
                        \{ 1 \} & \text{if $p \neq 2$,} \\
                        \{ 1, 2^{n-1} - 1 \} & \text{if $p = 2$.}
                    \end{cases}
                \end{equation*} 
            \end{Lem}
            \begin{proof}
                For each $i \in ( \Z / p \Z )^{*}$, let $T_{i} \coloneqq \{ \sigma \in T \,\colon\, \sigma \equiv i \pmod{p} \}$. By definition, the elements $\sigma = p \cdot \sigma_{1} - \sigma_{2}$ of $T_{i}$ have $\sigma_{2} = p - i_{0}$, where $i_{0} \in \{ 1, \ldots, p - 1 \}$ is the unique integer in this interval that is congruent to $i$ modulo $p$. It follows that $| T_{i} | = ( p - i_{0} ) \cdot p^{n-2}$. Note in particular that $| T_{1} | \ge | T_{i} |$ for all $i \in ( \Z / p \Z )^{*}$, with equality if and only if $i \equiv 1 \pmod{p}$.
    
                If $k \in ( \Z / p^{n} \Z )^{*}$ stabilizes $T$, then it sends $T_{i}$ to $T_{ik}$ for every $i \in ( \Z / p \Z )^{*}$. In particular, it sends $T_{1}$ into $T_{k}$, hence $| T_{1} | \le | T_{k} |$. Since $| T_{1} |$ is strictly greater than $| T_{i} |$ for any $i \not \equiv 1 \pmod{p}$, we must have $k \equiv 1 \pmod{p}$.

                Thus, $\Stab ( T )$ is contained in the subgroup $H \coloneqq \bigl\{ \sigma \in ( \Z / p^{n} \Z )^{*} \,\colon\, \sigma \equiv 1 \pmod{p} \bigr\}$. When $p$ is odd, $H$ is cyclic of order $p^{n-1}$, generated by the class of $1 + p$. When $p = 2$, the group $H$ is isomorphic to $\frac{\Z}{2 \Z} \times \frac{\Z}{2^{n-2} \Z}$, with the two factors generated by the classes of $-1$ and $5$ respectively.

                We will need the following simple observation: by definition, $T$ contains each of the elements $1, \ldots, p^{n-1} - 1$ that are coprime to $p$ (corresponding to $\sigma_{1} \le p^{n-2}$, which is admissible for any value of $\sigma_{2}$), and it does not contain any of the elements $p^{n} - p^{n-1} + 1, \ldots, p^{n} - 1$ (corresponding to $\sigma_{1} > p^{n-1} - p^{n-2}$, not admissible for any value of $\sigma_{2}$).
    
                We first consider the case of odd $p$. If $\Stab ( T )$ is nontrivial, then it contains a nontrivial subgroup of $H$. Since $H$ is cyclic of prime power order, it contains a unique minimal subgroup (cyclic of order $p$), consisting of the elements $\{ 1 + h p^{n-1} \,\colon\, 0 \le h \le p - 1 \}$. In particular, if $\Stab ( T )$ is nontrivial, then it contains $1 + p^{n-1}$. However, this is not the case, because $p - 1$ belongs to $T$, while $( 1 + p^{n-1} )  ( p - 1 ) > p^{n} - p^{n-1}$ does not. Note that $( 1 + p^{n-1} ) ( p - 1 ) = p^{n} - p^{n-1} + p - 1 < p^{n}$ since $n \ge 2$, so there is no problem in identifying the residue class of $( 1 + p^{n-1} ) ( p - 1 )$ modulo $p^{n}$ with the corresponding integer in $\{ 0, \ldots, p^{n} \}$.

                The argument for $p = 2$ is similar. If $n = 2$, the lemma is trivial; so, let us assume $n > 2$. If $\Stab ( T )$ intersects nontrivially the subgroup of $( \Z / 2^{n} \Z)^{*}$ generated by $5$, then it must contain its unique element of order $2$, which is $1 + 2^{n-1}$. As above, $1 + 2^{n-1}$ does not stabilize $T$, because $1$ is an element of $T$, but $( 1 + 2^{n-1} ) \cdot 1 > 2^{n} - 2^{n-1}$ is not. Thus, $\Stab ( T ) \cap \langle 5 \rangle$ is trivial. This implies that the subgroup generated by $\Stab ( T )$ and $\langle 5 \rangle$ inside $( \Z / 2^{n} \Z )^{*}$ has order equal to $| \Stab(T) | \cdot | \langle 5 \rangle | = | \Stab(T) | \cdot 2^{n-2}$. Since the whole group $( \Z / 2^{n} \Z )^{*}$ has order $2^{n-1}$, this implies $| \Stab(T) | \le 2$. Finally, it is easy to check that $2^{n-1} - 1$ is a nontrivial element of $\Stab ( T )$, hence $\Stab(T)$ has order precisely $2$ and is generated by $2^{n-1} - 1$.
            \end{proof}

            We conclude this preliminary section by analyzing the $( 1 - \Phi )$-torsion of $\Jbf$, given by $\Jbf [ 1 - \Phi ] \coloneqq \bigl\{ D \in \Jbf^{\mathrm{al}} \bigl( \Qbar(U)^{\mathrm{al}} \bigr) \,\colon\, ( 1 - \Phi ) \cdot D = 0 \bigr\}$.

            \begin{Lem}
            \label{lem: properties of 1-Phi torsion part; subsec: Preliminaries; sec: Cyclic groups of order p^n}
                The points
                \begin{equation}
                \label{eq: Dalphal; subsec: Preliminaries; sec: Cyclic groups of order p^n}
                    D_{\alpha_{l}} \coloneqq \Bigl[ \bigl( \zeta_{p^{n-1}} \cdot \alpha_{l}^{1 / p^{n-1}}, 0 \bigr) + \cdots + \bigl( \zeta_{p^{n-1}}^{p^{n-1}} \cdot \alpha_{l}^{1 / p^{n-1}}, 0 \bigr) - p^{n-1} \cdot \infty \Bigr] \in \Jbf^{\mathrm{al}}(\Qbar(U)^{\mathrm{al}}) ,
                \end{equation}
                for $1 \le l \le L$, where the square brackets denote the class of a divisor, form a basis of the $\Fbb_{p}$-vector space $\Jbf [ 1 - \Phi ]$.
            \end{Lem}
            \begin{proof}
                Consider the embedding $\iota \colon \Z [ \zeta_{p^{n}} ] \to \End ( \Jbf^{\mathrm{al}} )$ given by $\zeta_{p^{n}} \mapsto \Phi$ (cf.~the proof of Lemma~\ref{lem: End0Jbf contains Q(zetapn); subsec: Preliminaries; sec: Cyclic groups of order p^n}). In $\Z [ \zeta_{p^{n}} ]$ we have the equality $( 1 - \zeta_{p^{n}} )^{\varphi(p^{n})} = u \cdot p$ for some $u \in \Z [ \zeta_{p^{n}} ]^{*}$. Using $\iota$, we deduce the existence of $\lambda \in \End ( \Jbf^{\mathrm{al}} )^{*}$ such that $[p] = \lambda \cdot ( 1 - \Phi )^{ \varphi(p^{n}) }$. In particular, the kernel of $( 1 - \Phi )$, which is $\Jbf [ 1 - \Phi ]$, is contained in the kernel $\Jbf [ p ]$ of $p$. It follows that $\Jbf [ 1 - \Phi ]$ is an $\Fbb_{p}$-vector subspace of $\Jbf [ p ]$. Moreover, by taking degrees and using that $\lambda$ has degree $1$, we find $p^{ 2 \dim ( \Jbf^{\mathrm{al}} ) } = \deg \bigl( [p] \bigr) = \deg ( 1 - \Phi )^{\varphi(p^{n})}$. It follows that the dimension of $\Jbf [ 1 - \Phi ]$ over $\Fbb_{p}$ is $2 \dim ( \Jbf^{\mathrm{al}} ) / \varphi( p^{n} ) = L$, where the equality uses Remark~\ref{rem: genus of Cbf; subsec: Setup; sec: Cyclic groups of order p^n}.
                    
                By \eqref{eq: phi; subsec: Setup; sec: Cyclic groups of order p^n}, it is easy to check that every $D_{\alpha_{l}}$ belongs to $\Jbf [ 1 - \Phi ]$. Moreover, from a computation in \cite[\S~2]{WojciechSuperJac}, it follows that the $D_{\alpha_{l}}$, as $l$ varies in $\{ 1, \dots, L \}$, are $\Fbb_{p}$-linearly independent. Hence, they form a basis of $\Jbf [ 1 - \Phi ]$. 
            \end{proof}

        \subsection{Computation of specific specializations}
        \label{subsec: specialization methods; sec: Cyclic groups of order p^n}

            An important ingredient in the proof of Theorem~\ref{the: main; sec: Cyclic groups of order p^n} is to find specific specializations of $\Jfrak \to U$. We refer the reader to Section~\ref{subsec: specialization methods; sec: Preliminaries} for the relevant theoretical constructions.
            
            \begin{Pro}
            \label{pro: particular specialization; subsec: specialization methods; sec: Cyclic groups of order p^n}
                Let $L \ge 2$, let $( a_{1}, \dots, a_{L} ) \in U ( \Qbar )$ and let $S$ be a subset of $\{ 1, \dots, L \}$ of cardinality $1 \le M < L$. Consider the curves
                \begin{equation*}
                    \Ccal_{1} \colon y^{p} = x \cdot \prod_{l \notin S} ( x^{p^{n-1}} - a_{l} ) , \qquad\qquad \Ccal_{2} \colon y^{p} = x \cdot \prod_{l \in S} \bigl( x^{p^{n-1}} - a_{l} \bigr) .
                \end{equation*} 
                The following hold:
                \begin{enumerate}
                    \item There exist an admissible $T'$ and a point $t \in \Abb^{1}_{\Qbar} ( \Qbar ) \supset T' ( \Qbar )$ such that $\Jfrak_{T', t}$ is isomorphic to $\Jac ( \Ccal_{1} ) \times \Jac ( \Ccal_{2} )$;

                    \item Let $\eta \colon \Jbf [ 1 - \Phi ] \xrightarrow{\sim} \Jfrak_{T', t} [ 1 - \Phi ]$ be the bijection established by the last statement of Lemma~\ref{lem: End0 embeds and torsion bijects under specialization; subsec: specialization methods; sec: Preliminaries} and let $D_{\alpha_{l}} \in \Jbf [ 1 - \Phi ]$ be the point defined in \eqref{eq: Dalphal; subsec: Preliminaries; sec: Cyclic groups of order p^n}. Under the identification $\Jfrak_{T', t} \cong \Jac ( \Ccal_{1} ) \times \Jac ( \Ccal_{2} )$, for each integer $1 \le l \le L$ we have
                    \begin{equation*}
                        \eta ( D_{\alpha_{l}} ) = 
                        \begin{cases}
                            \Bigl( \bigl[ \sum_{i = 1}^{ p^{n - 1} } P_{l, i} - p^{n - 1}  \cdot \infty_{1} \bigr], 0 \Bigr) & \text{if $l \notin S$,} \\[1.2em]
                            \Bigl( 0, \bigl[ \sum_{i = 1}^{ p^{n - 1} } Q_{l, i} - p^{n - 1} \cdot \infty_{2} \bigr] \Bigr) & \text{if $l \in S$,}
                        \end{cases}
                    \end{equation*}
                    where
                    \begin{equation*}
                        P_{l, i} \coloneqq ( \zeta_{p^{n - 1}}^{i} \cdot a_{l}^{ 1 / p^{n - 1} }, 0 ) \in \Ccal_{1} \qquad \text{if } l \notin S,
                    \end{equation*}
                    \begin{equation*}
                        Q_{l, i} \coloneqq ( \zeta_{p^{n - 1}}^{i} \cdot a_{l}^{1 / p^{n - 1}}, 0 ) \in \Ccal_{2} \qquad \text{if } l \in S,
                    \end{equation*}
                    and $\infty_{1}$, $\infty_{2}$ are the unique points at infinity of $\Ccal_{1}$ and $\Ccal_{2}$.
                \end{enumerate}
            \end{Pro}
            \begin{proof}
                For each $1 \le l \le L$, set $\varepsilon_{l} = 1$ if $l \in S$ and $\varepsilon_{l} = 0$ otherwise, and consider the morphism
                \begin{equation*}
                    \gamma \colon \Abb^{1}_{\Qbar} \longrightarrow \Abb^{L}_{\Qbar} \qquad\qquad\qquad t \mapsto \bigl( t^{p^{n} \cdot \varepsilon_{1}} \cdot a_{1}, \, \dots, \, t^{p^{n} \cdot \varepsilon_{L}} \cdot a_{L} \bigr) .
                \end{equation*}
                Let $T' \subset \Abb^{1}_{\Qbar}$ be the pullback of $U$ under $\gamma$, and let $V$ be the Zariski-open neighborhood of $0$ obtained by adding the point $0$ to $T'$. We extend $\mathfrak{C}_{T'} \to T'$ to $\mathfrak{C}_{V} \to V$ in the obvious way, namely, using the explicit equation
                \begin{equation}
                \label{eq: CfrakT'; pro: particular specialization; subsec: specialization methods; sec: Cyclic groups of order p^n}
                    y^{p} = x \cdot \prod_{l \notin S} ( x^{p^{n-1}} - a_{l} ) \cdot \prod_{l \in S} ( x^{p^{n-1}} - t^{p^{n}} \cdot a_{l} )
                \end{equation}
                over all of $V$. We construct a stable model $\Cfrak_{\mathrm{st}} \to V$ of $\mathfrak{C}_{V} \to V$, which turns out to be the restriction to $V$ of the minimal regular model $\Cfrak_{\mathrm{reg}} \to \Abb_{\Qbar}^{1}$ of $\mathfrak{C}_{T'} \to T'$, and we compute the fiber $\Cfrak_{T',\,0}$ of $\Cfrak_{\mathrm{st}} \to V$ over $0$.
                
                When $t = 0$, \eqref{eq: CfrakT'; pro: particular specialization; subsec: specialization methods; sec: Cyclic groups of order p^n} reduces to the singular curve
                \begin{equation}
                \label{eq: Ccal1; pro: particular specialization; subsec: specialization methods; sec: Cyclic groups of order p^n}
                    y^{p} = x^{M \cdot p^{n-1}} \cdot x \cdot \prod_{l \notin S} ( x^{p^{n-1}} - a_{l} ) ,
                \end{equation}
                whose normalization is $\Ccal_{1}$. Hence, $\Ccal_{1}$ is a component of $\Cfrak_{T', 0}$.

                To compute the possible remaining components, we blow up $\Cfrak_{T'} \to T'$ at the singular point $( x, y, t ) = ( 0, 0, 0)$. In the $t$-chart and using the weight $( p, 1 + M \cdot p^{n-1} )$, that is, applying the transformations $x = t^{p} \cdot u$ and $y = t^{ 1 + M \cdot p^{n-1} } \cdot v$, the strict transform of $\Cfrak_{T'} \to T'$ is 
                \begin{equation*}
                    v^{p} = u \cdot \prod_{l \notin S} ( t^{p^{n}} \cdot u^{p^{n-1}} - a_{l} ) \cdot \prod_{l \in S} ( u^{p^{n-1}} - a_{l} ) .
                \end{equation*}
                When $t = 0$, we obtain the curve
                \begin{equation}
                \label{eq: Ccal2; pro: particular specialization; subsec: specialization methods; sec: Cyclic groups of order p^n}
                    v^{p} = c \cdot u \cdot \prod_{l \in S} ( u^{p^{n-1}} - a_{l} ) ,
                \end{equation}
                where $c$ is the product of the $-a_{l}$ as $l$ ranges over $\{ 1, \, \dots, \, L \} \setminus S$. This is a twist of $\Ccal_{2}$; hence, it is isomorphic to $\Ccal_{2}$ since $\Qbar$ is algebraically closed.

                Since the strict transform is regular at $t = 0$, it remains regular in a neighborhood of $t = 0$. It follows that $\Cfrak_{T', 0}$ has two irreducible components, $\Ccal_{1}$ and $\Ccal_{2}$. As $\Cfrak_{\mathrm{st}} \to V$ is a flat family, \cite[Theorem~9.9, \S~III]{HartshorneAlgGeom} and Remark~\ref{rem: genus of Cbf; subsec: Setup; sec: Cyclic groups of order p^n} imply that the arithmetic genus of $\Cfrak_{T', 0}$ equals the sum of the arithmetic genera of $\Ccal_{1}$ and $\Ccal_{2}$. By \cite[Theorem~2, p.~190]{HironakaAritGenus}, we deduce that $\Ccal_{1}$ and $\Ccal_{2}$ intersect at a single point, namely $(x, y) = (0, 0)$ on $\Ccal_{1}$ and $\infty_{2}$ on $\Ccal_{2}$.

                It follows that the dual graph of $\Cfrak_{T', 0}$ has trivial first homology group, hence \cite[Example~$8$, p.~$246$]{BoschNeronModel} implies that $\Jfrak_{T', 0}$ is isomorphic to $\Jac ( \Ccal_{1} ) \times \Jac ( \Ccal_{2} )$, proving the first part of the proposition.

                Let $\Jbf_{T'}$ be the Jacobian of the generic fiber $\Cbf_{T'}$ of $\Cfrak_{T'} \to T'$. The specialization map $\Jbf \bigl( \Qbar ( U )^{ \mathrm{al} } \bigr) \to \Jbf_{T'} \bigl( \Qbar ( T' )^{ \mathrm{al} } \bigr)$ is easy to evaluate on the points $D_{\alpha_{l}}$; one sees immediately that it sends $D_{\alpha_{l}}$ to the point $D_{\alpha_{l}}'$ given by the class of the divisor
                \begin{equation*}
                    \sum_{i = 1}^{ p^{n - 1} } ( t^{ p \cdot \varepsilon_{l} } \cdot \zeta_{ p^{n - 1} }^{i} \cdot a_{l}^{ 1 / p^{n - 1} }, 0) - p^{n-1} \cdot \infty 
                \end{equation*}
                on $\Cbf_{T'}^{ \mathrm{al} }$, where $\infty$ is the unique point at infinity of $\Cbf_{T'}^{ \mathrm{al} }$. It remains to apply the specialization map $\operatorname{sp} : \Jbf_{T'} \bigl( \Qbar ( T' ) \bigr) \to \Jfrak_{T', 0} ( \Qbar )$, which we shall now describe.
                    
                Since our problem is local at $0$, from now on we let $B$ be the spectrum of the DVR $\Ocal_{V, 0}$ given by the localization of $\Qbar [ V ]$ at $(t)$. We consider the proper flat curve $X \coloneqq ( \Cfrak_{ \mathrm{st} } )_{B} \to B$ and denote by $X_{0}$ its special fiber, which is $\Cfrak_{T', 0}$. The functor $\operatorname{Pic}^{0}_{X / B}$ is representable by a scheme by \cite[Corollary~$3$, p.~260]{BoschNeronModel}, so (by properness) it satisfies $\operatorname{Pic}^{0}_{X / B} \bigl( \Qbar ( T' ) \bigr) = \operatorname{Pic}^{0}_{X / B} ( \mathcal{O}_{V, 0} )$. The scheme representing $\operatorname{Pic}^{0}_{X / B}$ is the Néron model of $\Jbf_{T'}$.

                The specialization map $\operatorname{sp}$ is the composition of the natural maps
                \begin{equation}
                \label{eq: description specialization; pro: particular specialization; subsec: specialization methods; sec: Cyclic groups of order p^n}
                    \begin{aligned}
                        \Jbf_{T'} \bigl( \Qbar ( T' ) \bigr) \cong \operatorname{Pic}^{0}_{X / B} \bigl( \Qbar ( T' ) \bigr) = \operatorname{Pic}^{0}_{X / B} ( \mathcal{O}_{V, 0} ) \to \operatorname{Pic}^{0}_{X / B} \bigl( \mathcal{O}_{V, 0} / (t) \bigr) = \operatorname{Pic}^{0}_{X_{0} / \Qbar} ( \Qbar ) \cong \mathfrak{J}_{T', 0} ( \Qbar )
                    \end{aligned}
                \end{equation}
                and may be described as follows. A point in $\Jbf_{T'} \bigl( \Qbar ( T' ) \bigr)$ is represented by a line bundle $\mathcal{L}$ of degree $0$ on $X_{\Qbar ( T' )}$ (this is the first isomorphism in \eqref{eq: description specialization; pro: particular specialization; subsec: specialization methods; sec: Cyclic groups of order p^n}). In order to specialize this point to $\mathfrak{J}_{T', 0}$, we extend $\mathcal{L}$ to a line bundle $\mathcal{L}_{\mathrm{st}}$ on all of $X$ whose restriction to the generic fiber agrees with $\mathcal{L}$ and whose restriction to the special fiber $X_{0}$ is of degree $0$ on each irreducible component (the existence of such an extension $\mathcal{L}_{\mathrm{st}}$ is the equality $\operatorname{Pic}^{0}_{X / B} ( \Qbar ( T' ) ) = \operatorname{Pic}^{0}_{X / B} ( \mathcal{O}_{V, 0} )$ in \eqref{eq: description specialization; pro: particular specialization; subsec: specialization methods; sec: Cyclic groups of order p^n}). Restricting $\mathcal{L}_{\mathrm{st}}$ to $X_{0}$ gives a line bundle of multidegree $0$ on $\mathfrak{C}_{\mathrm{st}, 0}$, hence a $\Qbar$-point of $\operatorname{Pic}^{0}_{X_{0} / \Qbar} \cong \mathfrak{J}_{T', 0} \cong \Jac ( \Ccal_{1} ) \times \Jac ( \Ccal_{2} )$ (this restriction is the arrow $\operatorname{Pic}^{0}_{X / B} ( \mathcal{O}_{V, 0} ) \to \operatorname{Pic}^{0}_{X / B} \bigl( \mathcal{O}_{V, 0} / (t) \bigr)$ in \eqref{eq: description specialization; pro: particular specialization; subsec: specialization methods; sec: Cyclic groups of order p^n}). The components in $\Jac ( \Ccal_{1} )$ and $\Jac ( \Ccal_{2} )$ under the isomorphism $\operatorname{Pic}^{0}_{X_{0} / \Qbar} \cong  \Jac ( \Ccal_{1} ) \times \Jac ( \Ccal_{2} )$ are the restrictions $\mathcal{L}_{\mathrm{st}}|_{\Ccal_1}$ and $\mathcal{L}_{\mathrm{st}}|_{\Ccal_2}$.

                Let $\widetilde{\infty}$ be the closure of $\infty$ inside $\mathfrak{C}_{\mathrm{st}}$. Since the blow-up we performed is at the finite point $( x, y, t ) = ( 0, 0, 0 )$, the section $\widetilde{\infty}$ is untouched by this operation and specializes to $\infty_{1} \in \Ccal_{1}$. For each $1 \le l \le L$ and each $1 \le i \le p^{n - 1}$, let $\widetilde{P}_{l, i}$ be the closure in $\mathfrak{C}_{\mathrm{st}}$ of the section $( t^{ p \cdot \varepsilon_{l} } \cdot \zeta_{ p^{n - 1} }^{i} \cdot a_{l}^{ 1 / p^{n - 1} }, 0 )$ defined over $\{ t \neq 0 \}$. 

                Assume first that $l \notin S$. Then, the $x$-coordinate of $( t^{ p \cdot \varepsilon_{l} } \cdot \zeta_{ p^{n - 1} }^{i} \cdot a_{l}^{ 1 / p^{n - 1} }, 0 )$ is independent of $t$ and in particular non-zero at $t=0$, so each $( \zeta_{ p^{n - 1} }^{i} \cdot a_{l}^{ 1 / p^{n - 1} }, 0 )$ extends through the smooth locus and specializes to $P_{l, i} = ( \zeta_{ p^{n - 1} }^{i} \cdot a_{l}^{ 1 / p^{n - 1} }, 0 ) \in \Ccal_{1}$. Hence $\eta ( D_{\alpha_{l}} ) = \operatorname{sp} ( D_{\alpha_{l}}' )$ is represented by the line bundle whose restrictions to the two components are $\mathcal{O}_{\Ccal_{1}} \bigl( \sum_{i = 1}^{ p^{n - 1} } P_{l, i} - p^{n-1} \cdot \infty_{1} \bigr)$ and $\mathcal{O}_{\Ccal_{2}}$, as desired.

                Assume instead that $l \in S$. Then, on the generic fiber, $D_{\alpha_{l}}'$ is represented by the divisor
                \begin{equation}
                \label{eq: divisor specializing to exceptional component; pro: particular specialization; subsec: specialization methods; sec: Cyclic groups of order p^n}
                    \sum_{i = 1}^{ p^{n - 1} } ( t^{p} \cdot \zeta_{ p^{n - 1} }^{i} \cdot a_{l}^{ 1 / p^{n - 1} }, 0 ) - p^{n-1} \cdot \infty.
                \end{equation}
                The specialization of \eqref{eq: divisor specializing to exceptional component; pro: particular specialization; subsec: specialization methods; sec: Cyclic groups of order p^n} to the nodal curve $\Ccal_{1} \cup \Ccal_{2}$ has multidegree $( -p^{n-1}, p^{n-1} )$: its restriction to $\Ccal_{1}$ is $\mathcal{O}_{\Ccal_{1}} \bigl( -p^{n-1} \cdot \infty_{1} \bigr)$, and its restriction to $\Ccal_{2}$ is $\mathcal{O}_{\Ccal_{2}} \bigl( \sum_{i = 1}^{ p^{n - 1} } Q_{l, i} \bigr)$. Let $N_{1} = ( 0, 0 ) \in \Ccal_{1}$ and $N_{2} = \infty_{2} \in \Ccal_{2}$ be the points on the two irreducible components corresponding to the node. We consider the curve $\mathcal{C}_{2} \subset \{ t = 0 \}$, seen as a divisor in the total space $\mathfrak{C}_{\mathrm{st}}$, and let $\mathscr{L} \coloneqq \mathcal{O}_{\mathfrak{C}_{\mathrm{st}}} ( \Ccal_{2} )$. A computation shows that $\mathscr{L}$ restricts to $\mathcal{O}_{\Ccal_1} ( N_{1} )$ and $\mathcal{O}_{\Ccal_{2}} ( -N_{2} )$ on the two irreducible components, hence in particular is of multidegree $( 1, -1 )$. We twist the line bundle corresponding to the divisor \eqref{eq: divisor specializing to exceptional component; pro: particular specialization; subsec: specialization methods; sec: Cyclic groups of order p^n} by $\mathscr{L}^{ \otimes p^{n-1} }$. Since $( N_{1} - \infty_{1} )$ is a torsion point of order $p$, the twisted line bundle is linearly equivalent to a line bundle that restricts to $\Ocal_{\Ccal_{1}}$ and $\Ocal_{\Ccal_{2}} \bigl( \sum_{i = 1}^{ p^{n - 1} } Q_{l, i} - p^{n-1} \cdot \infty_{2} \bigr)$ on the two irreducible components, as desired.
            \end{proof}

            We now analyze how $\Phi$ specializes in the specializations provided by Proposition~\ref{pro: particular specialization; subsec: specialization methods; sec: Cyclic groups of order p^n}.

            \begin{Lem}
            \label{lem: how Phi specializes in JTt; subsec: specialization methods; sec: Cyclic groups of order p^n}
                Let $L \ge 2$, let $( a_{1}, \dots, a_{L} ) \in U ( \Qbar )$ and let $S$ be a subset of $\{ 1, \dots, L \}$ of cardinality $1 \le M < L$. Let $\Jcal$ be the abelian variety obtained by applying Proposition~\ref{pro: particular specialization; subsec: specialization methods; sec: Cyclic groups of order p^n} with $( a_{1}, \dots, a_{L})$ and $S$, and let $\iota \colon \End^{0} ( \Jbf^{\mathrm{al}} ) \xhookrightarrow{} \End^{0} ( \Jcal )$ be the canonical embedding provided by the first statement of Lemma~\ref{lem: End0 embeds and torsion bijects under specialization; subsec: specialization methods; sec: Preliminaries}.
                Consider the curves
                \begin{equation*}
                    \Ccal_{1} \colon z^{p} = x \cdot \prod_{l \notin S} ( x^{p^{n-1}} - a_{l} ) , \qquad\qquad \Ccal_{2} \colon v^{p} = u \cdot \prod_{l \in S} \bigl( u^{p^{n-1}} - a_{l} \bigr) ,
                \end{equation*} 
                and recall that $\Jcal$ is isomorphic to $\Jac ( \Ccal_{1} ) \times \Jac ( \Ccal_{2} )$. Let $\varphi_{1}$ (resp.~$\varphi_{2}$) be the automorphism of $\Ccal_{1}$ (resp.~$\Ccal_{2}$) given by $( x, z ) \mapsto ( \zeta_{ p^{n-1} } \cdot x, \zeta_{ p^{n} } \cdot z )$ (resp.~$( u, v ) \mapsto ( \zeta_{ p^{n-1} } \cdot u, \zeta_{ p^{n} } \cdot v )$), and let $\Phi_{1}$ (resp.~$\Phi_{2}$) be the automorphism of $\Jac ( \Ccal_{1} )$ (resp.~$\Jac ( \Ccal_{2} )$) induced by $\varphi_{1}$ (resp.~$\varphi_{2}$). 
                
                The automorphism $\iota ( \Phi )$ stabilizes both $\Jac ( \Ccal_{1} )$ and $\Jac ( \Ccal_{2} )$. Moreover, it acts on $\Jac ( \Ccal_{1} )$ as $\Phi_{1}^{ 1 - M \cdot p^{ n - 1 } }$, and it acts on $\Jac ( \Ccal_{2} )$ as $\Phi_{2}$.
            \end{Lem}
            \begin{proof}
                Recall, from the proof of Proposition~\ref{pro: particular specialization; subsec: specialization methods; sec: Cyclic groups of order p^n}, that $z = y / x^{ M \cdot p^{n-2} }$, $u = x / t^{p}$ and $\sqrt[p]{c} \cdot v = y / t^{ 1 + M \cdot p^{ n - 1 } }$, where $c$ is the product of the $- a_{l}$ as $l$ ranges over $\{ 1, \dots, L \} \setminus S$. From \eqref{eq: phi; subsec: Setup; sec: Cyclic groups of order p^n} we deduce that the induced action of $\varphi$ on these variables is given by 
                \begin{equation*}
                    ( x, z ) \mapsto ( \zeta_{ p^{n - 1} } \cdot x, \zeta_{p^{n}}^{ 1 - M \cdot p^{n - 1} } \cdot z ) = ( \zeta_{ p^{n - 1} }^{ 1 - M \cdot p^{n - 1} } \cdot x, \zeta_{p^{n}}^{ 1 - M \cdot p^{n - 1} } \cdot z ) ,
                \end{equation*}
                and
                \begin{equation*}
                    ( u, v ) \mapsto ( \zeta_{ p^{n - 1} } \cdot u, \zeta_{p^{n}} \cdot v ) ,
                \end{equation*}
                since $\varphi$ does not act on the variable $t$. Hence, $\iota ( \Phi )$ stabilizes both $\Jac ( \Ccal_{1} )$ and $\Jac ( \Ccal_{2} )$, and it acts on $\Jac ( \Ccal_{1} )$ and $\Jac ( \Ccal_{2} )$ as desired.
            \end{proof}

            Proposition~\ref{pro: particular specialization; subsec: specialization methods; sec: Cyclic groups of order p^n} and Lemma~\ref{lem: how Phi specializes in JTt; subsec: specialization methods; sec: Cyclic groups of order p^n} have an interesting consequence.

            \begin{Cor}
            \label{cor: exclusion cases with torsion arguments; subsec: specialization methods; sec: Cyclic groups of order p^n}
                Let $L \ge 2$, let $( a_{1}, \dots, a_{L} ) \in U ( \Qbar )$ and let $S_{1}, \, S_{2}$ be distinct subsets of $\{ 1, \dots, L \}$ of equal cardinality $1 \le M < L$. Consider the curves
                \begin{align*}
                    \Ccal_{1} &\colon y^{p} = x \cdot \prod_{ l \notin S_{1} } ( x^{p^{n-1}} - a_{l} ) , \qquad\qquad \Ccal_{2} \colon y^{p} = x \cdot \prod_{ l \in S_{1} } \bigl( x^{p^{n-1}} - a_{l} \bigr) , \\
                    \Ccal_{3} &\colon y^{p} = x \cdot \prod_{ l \notin S_{2} } ( x^{p^{n-1}} - a_{l} ) , \qquad\qquad \Ccal_{4} \colon y^{p} = x \cdot \prod_{ l \in S_{2} } \bigl( x^{p^{n-1}} - a_{l} \bigr) .
                \end{align*} 
                Let $\Jcal_{1}$ (resp.~$\Jcal_{2}$) be the abelian variety obtained by applying Proposition~\ref{pro: particular specialization; subsec: specialization methods; sec: Cyclic groups of order p^n} with the point $( a_{1}, \dots, a_{L})$ and $S = S_{1}$ (resp.~$S = S_{2}$). The following conditions cannot hold simultaneously:
                \begin{itemize}
                    \item the Jacobian $\Jbf^{\mathrm{al}}$ is isogenous to a product of two abelian subvarieties $Z_{1}, Z_{2}$ of $\Jbf^{\mathrm{al}}$; 
                    \item the abelian variety $Z_{1}$ (resp.~$Z_{2}$) specializes to $\Jac ( \Ccal_{1} )$ (resp.~$\Jac ( \Ccal_{2} )$) in $\Jcal_{1}$; 
                    \item the abelian variety $Z_{1}$ (resp.~$Z_{2}$) specializes to $\Jac ( \Ccal_{3} )$ (resp.~$\Jac ( \Ccal_{4} )$) in $\Jcal_{2}$.
                \end{itemize}
            \end{Cor}
            \begin{proof}
                Assume, by way of contradiction, that the conditions hold. For each $\theta \in \End ( \Jbf^{ \mathrm{al} } )$ with finite kernel and each $i \in \{ 1, 2 \}$, let $\operatorname{red}_{ \theta, i } \colon \Jbf [ \theta ] \xrightarrow{\sim} \Jcal_{i} [ \theta ]$ be the bijection provided by the last statement of Lemma~\ref{lem: End0 embeds and torsion bijects under specialization; subsec: specialization methods; sec: Preliminaries}. We first prove that both $Z_{1}$ and $Z_{2}$ are stable under the action of $1 - \Phi$.

                Let $\iota \colon \End^{0} ( \Jbf^{\mathrm{al}} ) \xhookrightarrow{} \End^{0} ( \Jcal_{1} )$ be the canonical embedding provided by the first statement of Lemma~\ref{lem: End0 embeds and torsion bijects under specialization; subsec: specialization methods; sec: Preliminaries}. Let $\varphi_{1}$ (resp.~$\varphi_{2}$) be the automorphism of $\Ccal_{1}$ (resp.~$\Ccal_{2}$) given by $( x, y ) \mapsto ( \zeta_{ p^{n-1} } \cdot x, \zeta_{ p^{n} } \cdot y )$, and let $\Phi_{1}$ (resp.~$\Phi_{2}$) be the automorphism of $\Jac ( \Ccal_{1} )$ (resp.~$\Jac ( \Ccal_{2} )$) induced by $\varphi_{1}$ (resp.~$\varphi_{2}$). Given a positive integer $m$, denote by $p^{m}$ the multiplication-by-$p^{m}$ (where the abelian variety is clear from the context). Finally, let $Z_{1} [ p^{\infty} ]$ be the $p$-power torsion subgroup of $Z_{1}$, which is the union of $Z_{1} [ p^{a} ]$ as $a$ varies over the positive integers.
                
                Since $Z_{1}$ specializes to $\Jac ( \Ccal_{1} )$, we have that $\dim ( Z_{1} ) = \dim \bigl( \Jac ( \Ccal_{1} ) \bigr)$ (first statement of Lemma~\ref{lem: End0 embeds and torsion bijects under specialization; subsec: specialization methods; sec: Preliminaries}). Hence, the cardinality of $Z_{1} [ p^{m} ]$ equals that of $\Jac ( \Ccal_{1} ) [ p^{m} ]$ and, since $\operatorname{red}_{ p^{m}, 1 }$ is injective, we deduce that $\operatorname{red}_{ p^{m}, 1 } \bigl( Z_{1} [ p^{m} ] \bigr) = \Jac(\Ccal_{1} ) [ p^{m} ]$ for all $m \ge 1$. We know that $\iota ( \Phi )$ stabilizes $\Jac ( \Ccal_{1} )$ and $\iota ( \Phi ) = \Phi_{1}^{ 1 - M \cdot p^{n-1} }$ (Lemma~\ref{lem: how Phi specializes in JTt; subsec: specialization methods; sec: Cyclic groups of order p^n}), hence
                \begin{align*}
                    \operatorname{red}_{ p^{m}, 1 } \Bigl( ( 1 - \Phi ) \bigl( Z_{1} [ p^{m} ] \bigr) \Bigr) &= \bigl( 1 - \iota ( \Phi ) \bigr) \Bigl( \operatorname{red}_{ p^{m}, 1 } \bigl( Z_{1} [ p^{m} ] \bigr) \Bigr) \\
                    &\subseteq ( 1 - \Phi_{1}^{ 1 - M \cdot p^{n-1} } ) \bigl( \Jac ( \Ccal_{1} ) [ p^{m} ] \bigr) \subseteq \Jac ( \Ccal_{1} ) [ p^{m} ] .
                \end{align*}
                We can apply $\operatorname{red}_{ p^{m}, 1 }$ to $( 1 - \Phi ) \bigl( Z_{1} [ p^{m} ] \bigr)$ since $( 1 - \Phi ) \bigl( Z_{1} [ p^{m} ] \bigr) \subseteq \Jbf [p^{m}]$, as $p^{m}$ commutes with $1 - \Phi$. The same argument shows the last containment above. Since $\operatorname{red}_{ p^{m}, 1 }$ is injective and $\operatorname{red}_{ p^{m}, 1 } \bigl( Z_{1} [ p^{m} ] \bigr) = \Jac(\Ccal_{1} ) [ p^{m} ]$, we obtain $( 1 - \Phi ) \bigl( Z_{1} [ p^{m} ] \bigr) \subseteq  Z_{1} [ p^{m} ]$ for all $m \ge 1$.

                It follows that $ 1 - \Phi$ stabilizes $Z_{1} [ p^{\infty} ]$. Since $Z_{1} [ p^{\infty} ]$ is Zariski-dense in $Z_{1}$ (\cite[Theorem~5.30]{EdixhovenAbelianVarieties}) and $1 - \Phi$ is an algebraic endomorphism of $\Jbf^{\mathrm{al}}$, it follows that it stabilizes $Z_{1}$. Similarly, we deduce that $1 - \Phi$ stabilizes $Z_{2}$, as desired.
                
                As $Z_{1}$ specializes to $\Jac ( \Ccal_{1} )$ in $\Jcal_{1}$, we have $\operatorname{red}_{ 1 - \Phi, 1 } \bigl( Z_{1} [ 1 - \Phi ] \bigr) \subseteq \Jac ( \Ccal_{1} ) [ 1 - \Phi_{1}^{ 1 - M \cdot p^{n-1} } ] \times \{ 0 \}$. Note that $1 - \Phi_{1} = u \cdot ( 1 - \Phi_{1}^{ 1 - M \cdot p^{n-1} } )$ for some automorphism $u$ of $\Jac ( \Ccal_{1} )$ (the ideals $( 1 - \zeta_{ p^{n} } )$ and $( 1 - \zeta_{ p^{n} }^{ 1 - M \cdot p^{n-1} } )$ are equal in $\Z [ \zeta_{ p^{n} } ]$). It follows that $1 - \Phi_{1}$ and $1 - \Phi_{1}^{ 1 - M \cdot p^{n-1} }$ have the same kernel; hence, $\operatorname{red}_{ 1 - \Phi, 1 } \bigl( Z_{1} [ 1 - \Phi ] \bigr) \subseteq \Jac ( \Ccal_{1} ) [ 1 - \Phi_{1} ] \times \{ 0 \}$. Similarly, $\operatorname{red}_{ 1 - \Phi, 1 } \bigl( Z_{2} [ 1 - \Phi ] \bigr) \subseteq \{ 0 \} \times \Jac ( \Ccal_{2} ) [ 1 - \Phi_{2} ]$. Since $\operatorname{red}_{ 1 - \Phi, 1 }$ is injective, we deduce that $Z_{1} [ 1 - \Phi ] \cap Z_{2} [ 1 - \Phi ] = \{ 0 \}$.

                In Proposition~\ref{pro: particular specialization; subsec: specialization methods; sec: Cyclic groups of order p^n} we computed the points $\operatorname{red}_{ 1 - \Phi, i } ( D_{\alpha_{l}} )$ for $i = 1, 2$ and $1 \le l \le L$. From the computation of $\operatorname{red}_{ 1 - \Phi, 1 }$, we deduce that $D_{\alpha_{l}} \in Z_{1} [ 1 - \Phi ]$ for every $l \in \{ 1, \dots, L \} \setminus S_{1}$. From the computation of $\operatorname{red}_{ 1 - \Phi, 2 }$, we deduce that $D_{\alpha_{l}} \in Z_{1} [ 1 - \Phi ]$ for every $l \in \{ 1, \dots, L \} \setminus S_{2}$. Hence, $D_{\alpha_{l}} \in Z_{1} [ 1 - \Phi ]$ for every $l \in \{ 1, \dots, L \} \setminus ( S_{1} \cap S_{2} )$. Similarly, we deduce that $D_{\alpha_{l}} \in Z_{2} [ 1 - \Phi ]$ for every $l \in S_{1} \cup S_{2}$.

                As $S_{1}$ and $S_{2}$ are distinct and non-empty, it follows that $\bigl| S_{1} \cup S_{2} \bigr| + \bigl| \{ 1, \dots, L \} \setminus ( S_{1} \cap S_{2} ) \bigr| > L$; hence, the sum of the numbers of elements $D_{\alpha_{l}}$ lying in $Z_{1} [ 1 - \Phi ]$ and in $Z_{2} [ 1 - \Phi ]$ is strictly greater than $L$. Since $L$ is the number of distinct $D_{\alpha_{l}}$, we deduce $Z_{1} [ 1 - \Phi ] \cap Z_{2} [ 1 - \Phi ] \neq \{ 0 \}$, a contradiction. 
            \end{proof}

            The following lemma is useful to apply Corollary~\ref{cor: exclusion cases with torsion arguments; subsec: specialization methods; sec: Cyclic groups of order p^n}. In this lemma only, we add subscripts to the objects involved to highlight their dependence on the parameter $L$ and avoid confusion for the reader.

            \begin{Lem}
            \label{lem: useful to apply cor: exclusion cases with torsion arguments; subsec: specialization methods; sec: Cyclic groups of order p^n}  
                For every $L \ge 1$, there exists $( a_{1}, \dots, a_{L + 1} ) \in U_{L + 1} ( \Qbar ) \subset \Abb_{\Qbar}^{L + 1} ( \Qbar )$ such that $\End^{0} ( \Jfrak_{ L, ( a_{1}, \dots, a_{L} ) } ) \cong \End^{0} ( \Jfrak_{ L, ( a_{2}, \dots, a_{L + 1} ) } ) \cong \End^{0} ( \Jbf_{L}^{\mathrm{al}} )$.
            \end{Lem}
            \begin{proof}
                Note that $\Cfrak \to U$ is the base change to $\Qbar$ of a family $\tilde{\Cfrak} \to \tilde{U}$ defined over $\Q$. We adopt the notation used in Lemma~\ref{lem: how many fibers of Jfrak have bigger End0; subsec: specialization methods; sec: Preliminaries}. For every $B > 0$, let $G_{ L + 1, B }$ be the set of points $( a_{1}, \dots, a_{L+1} ) \in V_{ L + 1 } ( \Q, B )$ such that either $\End^{0} ( \Jfrak_{ L, ( a_{1}, \dots, a_{L} ) } ) \not\cong \End^{0} ( \Jbf_{L}^{\mathrm{al}} )$ or $\End^{0} ( \Jfrak_{ L, ( a_{2}, \dots, a_{L+1} ) } ) \not\cong \End^{0} ( \Jbf_{L}^{\mathrm{al}} )$. Then, at least one of $( a_{1}, \dots, a_{L} )$ and $( a_{2}, \dots, a_{L+1} )$ belongs to $V_{ L, \mathrm{ex} } ( \Q, B )$. Let $S_{1}$ (resp.~$S_{2}$) be the set of points $( a_{1}, \dots, a_{L + 1} ) \in V_{ L + 1 } ( \Q, B )$ such that $( a_{1}, \dots, a_{L} )$ (resp.~$( a_{2}, \dots, a_{ L + 1 } )$) belongs to $V_{ L, \mathrm{ex} } ( \Q, B )$. 
                
                Let $P$ be a polynomial of degree $\omega \bigl( V_{ L, \mathrm{ex} } ( \Q, B ) \bigr)$ in $L$ variables with coefficients in $\Q$ whose zero set in $\Abb_{ \Qbar }^{L} ( \Qbar )$ contains $V_{ L, \mathrm{ex} } ( \Q, B )$. Consider $P$ as a polynomial in $L + 1$ variables. Its zero set in $\Abb_{ \Qbar }^{ L + 1 } ( \Qbar )$ contains $S_{1}$; hence $\omega ( S_{1} ) \le \omega \bigl( V_{ L, \mathrm{ex} } ( \Q, B ) \bigr)$. Combining Lemma~\ref{lem: number of bounded zeros polynomial; subsec: Preliminaries; sec: Cyclic groups of order p^n} and \cite[Theorem on p.~459]{MasserSpecializationEnd} applied to $\tilde{\Cfrak} \to \tilde{U}$, we deduce the existence of a constant $\lambda > 0$ such that $| S_{1} | = O \bigl( \log ( B )^{\lambda} \cdot B^{ L + 1 } \bigr)$. Similarly, we deduce $| S_{2} | = O \bigl( \log ( B )^{\lambda} \cdot B^{ L + 1 } \bigr)$.
                
                Since $| G_{ L + 1, B } | \le | S_{1} | + | S_{2} | = O \bigl( \log ( B )^{\lambda} \cdot B^{ L + 1 } \bigr)$ and, by \eqref{eq: 2; the: reduction first statement main the; subsec: Preliminaries; sec: Cyclic groups of order p^n}, $| V_{ L + 1 } ( \Q, B ) | \asymp B^{ L + 2 }$, for a sufficiently large $B$ we have $G_{ L + 1, B } \subsetneq V_{ L + 1 } ( \Q, B )$, as desired.
            \end{proof}

        \subsection{The case $p>2$}
        \label{subsec: The case p>2; sec: Cyclic groups of order p^n}

            We now prove Theorem~\ref{the: main; sec: Cyclic groups of order p^n} when $p > 2$. In Theorems~\ref{the: reduction first statement main the; subsec: Preliminaries; sec: Cyclic groups of order p^n} and~\ref{the: there are infinitely many classes; subsec: Preliminaries; sec: Cyclic groups of order p^n} we proved that it is sufficient to show that $\End^{0} ( \Jbf^{\mathrm{al}} ) \cong \Q ( \zeta_{p^{n}} )$. We begin by considering the cases $L = 1$ and $L = 2$.

            \begin{Lem}
            \label{lem: L = 1 is true; subsec: The case p>2; sec: Cyclic groups of order p^n}
                When $p > 2$ and $L = 1$, the fibers of $\Cfrak \to U$ over $\Qbar$-points of $U$ belong to a single isomorphism class of curves over $\Qbar$. Moreover, these fibers have simple Jacobians with endomorphism algebras isomorphic to $\Q ( \zeta_{p^{n}} )$.
            \end{Lem}
            \begin{proof}
                Remark~\ref{rem: if L = 1, there are finitely many classes; subsec: Preliminaries; sec: Cyclic groups of order p^n} shows that the curves $y^{p} = x \cdot ( x^{p^{n-1}} - a )$ are isomorphic over $\Qbar$ as $a$ varies in $\Qbar^{*}$. Hence, the fibers of $\Cfrak \to U$ over $\Qbar$-points of $U$ belong to a single isomorphism class of curves over $\Qbar$.

                In Lemma~\ref{lem: if L = 1 then Jbf is of CM type; subsec: Preliminaries; sec: Cyclic groups of order p^n} we proved that $\Jbf^{\mathrm{al}}$ has CM by $\Q ( \zeta_{p^{n}} )$, and in Lemma~\ref{lem: stabilizer of CM type; subsec: Preliminaries; sec: Cyclic groups of order p^n} that its CM type has trivial stabilizer. Hence, \cite[Theorem~3.5]{LangCM} implies that $\Jbf^{\mathrm{al}}$ is a simple abelian variety with endomorphism algebra isomorphic to $\Q ( \zeta_{p^{n}} )$, as desired. 
            \end{proof}

            \begin{Lem}
            \label{lem: L=2 is true; subsec: The case p>2; sec: Cyclic groups of order p^n}
                Theorem~\ref{the: main; sec: Cyclic groups of order p^n} holds for $p > 2$ and $L = 2$.
            \end{Lem}
            \begin{proof}
                Let $\Jbf^{\mathrm{al}} \sim X_{1} \times X_{2} \times \cdots \times X_{A}$ be an isogeny decomposition of $\Jbf^{\mathrm{al}}$, where some $X_{i}$ may be isogenous to each other. We begin by proving that $A \le 2$.

                In Lemma~\ref{lem: L = 1 is true; subsec: The case p>2; sec: Cyclic groups of order p^n} we proved that the curves given by $y^{p} = x \cdot ( x^{p^{n-1}} - a )$ are isomorphic over $\Qbar$ as $a$ varies in $\Qbar^{*}$, and that their Jacobians are isogenous to the same simple abelian variety $Y$, of dimension $\varphi ( p^{n} ) / 2$ (Remark~\ref{rem: genus of Cbf; subsec: Setup; sec: Cyclic groups of order p^n}), with endomorphism algebra isomorphic to $\Q ( \zeta_{p^{n}} )$. Let $\overline{ \Phi }$ be the automorphism of these Jacobians induced by the automorphism of the curves given by $( x, y ) \mapsto ( \zeta_{ p^{n-1} } \cdot x, \zeta_{ p^{n} } \cdot y )$, and recall that $\Q ( \zeta_{p^{n}} ) \cong \Q ( \overline{ \Phi } )$ (cf. Lemma~\ref{lem: End0Jbf contains Q(zetapn); subsec: Preliminaries; sec: Cyclic groups of order p^n}).
                    
                By applying Proposition~\ref{pro: particular specialization; subsec: specialization methods; sec: Cyclic groups of order p^n} with a fixed $( a_{1}, a_{2} ) \in U ( \Qbar )$ and $S = \{ 2 \}$, we deduce the existence of an abelian variety $\Jcal$, of the form $\Jfrak_{T', t}$ for some admissible $T'$ and $t \in \Abb_{\Qbar}^{1} ( \Qbar ) \supset T' ( \Qbar )$, isogenous to $Y^{2}$. By applying the first statement of Lemma~\ref{lem: End0 embeds and torsion bijects under specialization; subsec: specialization methods; sec: Preliminaries} with $\Jcal$, we deduce the existence of embeddings
                \begin{equation}
                \label{eq: End0J embeds in M2Qzetapn}
                    \Q ( \Phi ) \xhookrightarrow{} \End^{0} ( \Jbf^{\mathrm{al}} ) \xhookrightarrow{ \iota } \M \bigl( 2, \Q ( \zeta_{p^{n}} ) \bigr) .
                \end{equation}
                In Lemma~\ref{lem: isogeny decomposition specializes; subsec: specialization methods; sec: Preliminaries} we proved that the number of factors in an isogeny decomposition of $\Jbf^{\mathrm{al}}$ is at most that of an isogeny decomposition of $\Jcal$; hence, $A \le 2$. Moreover, under the identification $\M \bigl( 2, \Q ( \zeta_{p^{n}} ) \bigr) \cong \M \bigl( 2, \Q ( \overline{ \Phi } ) \bigr)$, we know that $\iota ( \Phi^{p} )$ is a multiple of the identity matrix, hence $\iota ( \Phi^{p} )$ lies in the center of $\M \bigl( 2, \Q ( \zeta_{p^{n}} ) \bigr)$ (Lemma~\ref{lem: how Phi specializes in JTt; subsec: specialization methods; sec: Cyclic groups of order p^n}).
                    
                Now, let us prove that $\End^{0} ( \Jbf^{\mathrm{al}} )$ is a simple algebra (equivalently, $\Jbf^{\mathrm{al}}$ is isotypic). If $A = 1$, we are done. Otherwise we have $A = 2$. In Lemma~\ref{lem: End0 embeds and torsion bijects under specialization; subsec: specialization methods; sec: Preliminaries} we proved that an abelian subvariety of $\Jbf^{\mathrm{al}}$ specializes to an abelian subvariety of $\Jcal$ of the same dimension. Since all the proper abelian subvarieties of $\Jcal$ are isogenous to $Y$, it follows that $\dim ( X_{1} ) = \dim ( X_{2} ) = \dim ( Y ) = \varphi ( p^{n} ) / 2$. If $X_{1}$ and $X_{2}$ are not isogenous, both $\End^{0} ( X_{1} )$ and $\End^{0} ( X_{2} )$ contain a field isomorphic to $\Q ( \zeta_{p^{n}} )$ (last statement of Lemma~\ref{lem: End0Jbf contains Q(zetapn); subsec: Preliminaries; sec: Cyclic groups of order p^n}), which has degree $2 \dim ( X_{1} ) = 2 \dim ( X_{2} )$; hence, both $X_{1}$ and $X_{2}$ have CM by $\Q ( \zeta_{p^{n}} )$, which contradicts Lemma~\ref{lem: Cbf is not of CM type if L > 1; subsec: Preliminaries; sec: Cyclic groups of order p^n}. Hence, $\Jbf^{\mathrm{al}}$ is isotypic.
                    
                Finally, we prove that $\End^{0} ( \Jbf^{\mathrm{al}} )$ is isomorphic to $\Q ( \zeta_{p^{n}} )$. Let $d$ be the dimension of $\End^{0} ( \Jbf^{\mathrm{al}} )$ over $\Q ( \Phi ) \cong \Q ( \zeta_{p^{n}} )$. Since the dimension of $\M \bigl( 2, \Q ( \zeta_{p^{n}} ) \bigr)$ over $\Q ( \zeta_{p^{n}} )$ is $4$, \eqref{eq: End0J embeds in M2Qzetapn} implies that $d \le 4$. Let $F$ be the center of the central simple algebra $\End^{0} ( \Jbf^{\mathrm{al}} )$ and let $N^{2}$ be the dimension of $\End^{0} ( \Jbf^{\mathrm{al}} )$ over $F$. Since $\iota ( \Phi^{p} )$ lies in the center of $\M \bigl( 2, \Q ( \zeta_{p^{n}} ) \bigr)$ and $\End^{0} ( \Jbf^{\mathrm{al}} )$ is a subalgebra of $\M \bigl( 2, \Q ( \zeta_{p^{n}} ) \bigr)$, it follows that $\Q ( \Phi^{p} ) \subseteq F$. We have both 
                \begin{equation*}
                    \dim_{\Q ( \Phi^{p} )} \bigl( \End^{0} ( \Jbf^{\mathrm{al}} ) \bigr) = \dim_{\Q ( \Phi^{p} )} (F) \cdot \dim_{F} \bigl( \End^{0} ( \Jbf^{\mathrm{al}} ) \bigr) = \dim_{\Q ( \Phi^{p} )} (F) \cdot N^{2}
                \end{equation*}
                and
                \begin{equation*}
                    \dim_{\Q ( \Phi^{p} )} \bigl( \End^{0} ( \Jbf^{\mathrm{al}} ) \bigr) = \dim_{\Q ( \Phi^{p} )} \bigl( \Q ( \Phi ) \bigr) \cdot \dim_{\Q ( \Phi )} \End^{0}\bigl(  \Jbf^{\mathrm{al}}  \bigr) = p \cdot d ,
                \end{equation*}
                so $N^{2}$ divides $p \cdot d$. When $N = 1$, the algebra $\End^{0} ( \Jbf^{\mathrm{al}} ) = F$ is a field of degree $d \cdot \bigl[ \Q ( \zeta_{p^{n}} ) \colon \Q \bigr]$ over $\Q$.
                \begin{itemize}
                    \item If $d = 1$, then $\End^{0} ( \Jbf^{\mathrm{al}} )$ is isomorphic to $\Q ( \zeta_{p^{n}} )$, as desired.
                    \item If $d = 2$, we have $N^{2} \mid 2p$.
                    Since $p > 2$, it follows that $N = 1$, hence $\End^{0} ( \Jbf^{\mathrm{al}} )$ is a field of degree $2 \varphi ( p^{n} )$ over $\Q$. Since $\dim ( \Jbf^{\mathrm{al}} ) = \varphi ( p^{n} )$ (Remark~\ref{rem: genus of Cbf; subsec: Setup; sec: Cyclic groups of order p^n}), it follows that $\Jbf^{\mathrm{al}}$ has CM, which contradicts Lemma~\ref{lem: Cbf is not of CM type if L > 1; subsec: Preliminaries; sec: Cyclic groups of order p^n}.
                    \item If $d = 3$, then the $\Q$-dimension of $\End^{0} ( \Jbf^{\mathrm{al}} )$ is $3 \varphi ( p^{n} )$. If $\Jbf^{\mathrm{al}}$ is simple, there is a contradiction to Lemma~\ref{lem: divisibility dimension and degree End simple abelian varieties; subsec: Abelian varieties; sec: Preliminaries} which asserts that the $\Q$-dimension of $\End^{0} ( \Jbf^{\mathrm{al}} )$ divides $2 \dim ( \Jbf^{\mathrm{al}} ) = 2 \varphi ( p^{n} )$. Hence, $A = 2$ and $N > 1$, since $\End^{0} ( \Jbf^{\mathrm{al}} )$ cannot be a field if $\Jbf^{\mathrm{al}}$ is not simple. We know that $N^{2}$ divides $d \cdot p = 3 p$. It follows that $p = N = 3$ and 
                    \begin{equation*}
                        9 = N^{2} = \dim_{F} \bigl( \End^{0} ( \Jbf^{\mathrm{al}} ) \bigr) = \dim_{F} \Bigl( \M \bigl( 2, \End^{0} ( X_{1} ) \bigr) \Bigr) = 4 \dim_{F} \bigl( \End^{0} ( X_{1} ) \bigr),
                    \end{equation*}
                    giving a contradiction.
                    \item If $d = 4$, comparing dimensions in \eqref{eq: End0J embeds in M2Qzetapn} we obtain that $\End^{0} ( \Jbf^{\mathrm{al}} )$ is isomorphic to $\M \bigl( 2, \Q ( \zeta_{p^{n}} ) \bigr)$. It follows that $A = 2$ and $\End^{0} ( X_{1} )$ is isomorphic to $\Q ( \zeta_{p^{n}} )$. Since $\dim ( X_{1} ) = \varphi ( p^{n} ) / 2$, it follows that $X_{1}$ has CM, which contradicts Lemma~\ref{lem: Cbf is not of CM type if L > 1; subsec: Preliminaries; sec: Cyclic groups of order p^n}.
                \end{itemize}
                Since $\End^{0} ( \Jbf^{\mathrm{al}} )$ is isomorphic to $\Q ( \zeta_{p^{n}} )$, Theorems~\ref{the: reduction first statement main the; subsec: Preliminaries; sec: Cyclic groups of order p^n} and~\ref{the: there are infinitely many classes; subsec: Preliminaries; sec: Cyclic groups of order p^n} imply that Theorem \ref{the: main; sec: Cyclic groups of order p^n} holds for $L = 2$, as desired.
            \end{proof}

            Finally, we consider the general case.

            \begin{proof}[Proof of Theorem~\ref{the: main; sec: Cyclic groups of order p^n} when $p > 2$]
                We proceed by induction on $L \ge 1$. Lemmas~\ref{lem: L = 1 is true; subsec: The case p>2; sec: Cyclic groups of order p^n} and \ref{lem: L=2 is true; subsec: The case p>2; sec: Cyclic groups of order p^n} prove the claim for $L = 1$ and $L = 2$. So, let us assume that $L \ge 2$ and that the claim holds for every $1 \le L' \le L$, and let us prove it for $L + 1$. 
                
                Let $\Jbf^{\mathrm{al}} \sim X_{1} \times X_{2} \times \cdots \times X_{A}$ be an isogeny decomposition of $\Jbf^{\mathrm{al}}$, where some $X_{i}$ may be isogenous to each other. We regard the $X_{i}$ as abelian subvarieties of $\Jbf^{\mathrm{al}}$.

                In Lemma~\ref{lem: useful to apply cor: exclusion cases with torsion arguments; subsec: specialization methods; sec: Cyclic groups of order p^n} we proved that there exists $( a_{1}, \dots, a_{L+1} ) \in U ( \Qbar ) \subset \Abb_{\Qbar}^{L+1} ( \Qbar )$ such that the endomorphism algebras of the Jacobians of both curves $\Ccal_{1} \colon y^{p} = x \cdot ( x^{p^{n-1}} - a_{1} ) \cdots ( x^{p^{n-1}} - a_{L} )$ and $\Ccal_{3} \colon y^{p} = x \cdot ( x^{p^{n-1}} - a_{2} ) \cdots ( x^{p^{n-1}} - a_{L+1} )$ are isomorphic to that of the Jacobian of the generic fiber of $y^{p} = x \cdot ( x^{p^{n-1}} - \alpha_{1} ) \cdots ( x^{p^{n-1}} - \alpha_{L} )$. The inductive hypothesis implies that the Jacobians of $\Ccal_{1}$ and $\Ccal_{3}$ are simple, of dimension $\varphi(p^{n}) \cdot L / 2$ (Remark~\ref{rem: genus of Cbf; subsec: Setup; sec: Cyclic groups of order p^n}), with endomorphism algebras isomorphic to $\Q ( \zeta_{p^{n}} )$.
                    
                In Lemma~\ref{lem: L = 1 is true; subsec: The case p>2; sec: Cyclic groups of order p^n} we proved that the curves given by $y^{p} = x \cdot ( x^{p^{n-1}} - a )$ are isomorphic over $\Qbar$ as $a$ varies in $\Qbar^{*}$, and that their Jacobians are isogenous to the same simple abelian variety, of dimension $\varphi ( p^{n} ) / 2$ (Remark~\ref{rem: genus of Cbf; subsec: Setup; sec: Cyclic groups of order p^n}), with endomorphism algebra isomorphic to $\Q  ( \zeta_{p^{n}} )$. Consider the curves $\Ccal_{2} \colon y^{p} = x \cdot ( x^{p^{n-1}} - a_{L+1} )$ and $\Ccal_{4} \colon y^{p} = x \cdot ( x^{p^{n-1}} - a_{1} )$.

                By applying Proposition~\ref{pro: particular specialization; subsec: specialization methods; sec: Cyclic groups of order p^n} with $( a_{1}, \dots, a_{L+1} ) \in U ( \Qbar )$ and $S = \{ L+1 \}$, we deduce the existence of an abelian variety $\Jcal_{1}$, of the form $\Jfrak_{T', t}$ for some admissible $T'$ and $t \in \Abb_{\Qbar}^{1} ( \Qbar ) \supset T' ( \Qbar )$, isomorphic to $\Jac ( \Ccal_{1} ) \times \Jac ( \Ccal_{2} )$. Since $\Jac ( \Ccal_{1} )$ and $\Jac ( \Ccal_{2} )$ are simple and have different dimensions, the algebra $\End^{0} ( \Jcal_{1} )$ is isomorphic to $\Q ( \zeta_{p^{n}} ) \times \Q ( \zeta_{p^{n}} )$. 
                    
                Remark~\ref{rem: simple observation of lem: End0 embeds and torsion bijects under specialization; subsec: specialization methods; sec: Preliminaries} shows that the endomorphism algebra of any isotypic component of $\Jbf^{\mathrm{al}}$ embeds into the endomorphism algebra of at least one isotypic component of $\Jcal_{1}$; hence, it embeds into $\Q ( \zeta_{p^{n}} )$. On the other hand, in Lemma~\ref{lem: End0Jbf contains Q(zetapn); subsec: Preliminaries; sec: Cyclic groups of order p^n} we proved that the endomorphism algebra of each isotypic component of $\Jbf^{\mathrm{al}}$ contains a field isomorphic to $\Q ( \zeta_{p^{n}} )$. It follows that the endomorphism algebra of every isotypic component of $\Jbf^{\mathrm{al}}$ is isomorphic to $\Q ( \zeta_{p^{n}} )$. In particular, every isotypic component of $\Jbf^{\mathrm{al}}$ is simple, and $\End^{0} ( \Jbf^{\mathrm{al}} )$ is isomorphic to the product of $A$ copies of $\Q ( \zeta_{p^{n}} )$. The first statement of Lemma~\ref{lem: End0 embeds and torsion bijects under specialization; subsec: specialization methods; sec: Preliminaries} applied to $\Jcal_{1}$ implies that $\End^{0} ( \Jbf^{\mathrm{al}} )$ injects into $\Q ( \zeta_{p^{n}} ) \times \Q ( \zeta_{p^{n}} )$. It follows that either $A = 1$ or $A = 2$.

                Let us exclude the case $A = 2$. The nontrivial abelian subvarieties of $\Jcal_{1}$ are $\Jac ( \Ccal_{1} )$ and $\Jac ( \Ccal_{2}  )$. In Lemma~\ref{lem: End0 embeds and torsion bijects under specialization; subsec: specialization methods; sec: Preliminaries} we proved that an abelian subvariety of $\Jbf^{\mathrm{al}}$ specializes to an abelian subvariety of $\Jcal_{1}$ of the same dimension. Hence, up to renumbering the $X_{i}$, we have that $X_{1}$ (resp.~$X_{2}$) specializes to $\Jac ( \Ccal_{1} )$ (resp.~$\Jac ( \Ccal_{2} )$) in $\Jcal_{1}$. By applying Proposition~\ref{pro: particular specialization; subsec: specialization methods; sec: Cyclic groups of order p^n} with $( a_{1}, \dots, a_{L+1} ) \in U ( \Qbar )$ and $S = \{ 1 \}$, we deduce the existence of an abelian variety $\Jcal_{2}$, of the form $\Jfrak_{T', t}$ for some admissible $T'$ and $t \in \Abb_{\Qbar}^{1} ( \Qbar ) \supset T' ( \Qbar )$, isomorphic to $\Jac ( \Ccal_{3} ) \times \Jac ( \Ccal_{4} )$. Similarly, it follows that $X_{1}$ (resp.~$X_{2}$) specializes to $\Jac ( \Ccal_{3} )$ (resp.~$\Jac ( \Ccal_{4} )$) in $\Jcal_{2}$. This contradicts Corollary~\ref{cor: exclusion cases with torsion arguments; subsec: specialization methods; sec: Cyclic groups of order p^n} applied to $( a_{1}, \dots, a_{L+1} ) \in U ( \Qbar )$, $S_{1} = \{ L+1 \}$ and $S_{2} = \{ 1 \}$.

                Hence, $A = 1$ and $\End^{0} ( \Jbf^{\mathrm{al}} )$ is isomorphic to $\Q ( \zeta_{p^{n}} )$. Theorems~\ref{the: reduction first statement main the; subsec: Preliminaries; sec: Cyclic groups of order p^n} and~\ref{the: there are infinitely many classes; subsec: Preliminaries; sec: Cyclic groups of order p^n} imply that Theorem \ref{the: main; sec: Cyclic groups of order p^n} holds for $L+1$, as desired.
            \end{proof}

        \subsection{The case $p=2$}
        \label{subsec: The case p=2; sec: Cyclic groups of order p^n}

            We now prove Theorem~\ref{the: main; sec: Cyclic groups of order p^n} when $p = 2$. In Theorems~\ref{the: reduction first statement main the; subsec: Preliminaries; sec: Cyclic groups of order p^n} and~\ref{the: there are infinitely many classes; subsec: Preliminaries; sec: Cyclic groups of order p^n} we proved that it is sufficient to show that $\End^{0} ( \Jbf^{\mathrm{al}} ) \cong \Q ( \zeta_{2^{n}} )$. 
              
            The main difference from the case of odd $p$ arises in the application of Lemma~\ref{lem: stabilizer of CM type; subsec: Preliminaries; sec: Cyclic groups of order p^n}. When $p = 2$ and $L = 1$, the abelian variety $\Jbf^{\mathrm{al}}$ again has CM by $\Q ( \zeta_{2^{n}} )$. However, when $n > 2$, its CM type no longer has trivial stabilizer in $\Gal \bigl( \Q ( \zeta_{2^{n}} ) / \Q \bigr)$; hence, $\Jbf^{\mathrm{al}}$ is no longer simple.

            \begin{Lem}
            \label{lem: L = 1; subsec: The case p=2; sec: Cyclic groups of order p^n}
                When $p = 2$ and $L = 1$, the fibers of $\Cfrak \to U$ over $\Qbar$-points of $U$ belong to a single isomorphism class of curves over $\Qbar$. Moreover:
                \begin{enumerate}
                    \item if $n = 2$ the Jacobians of these fibers are elliptic curves with endomorphism algebras isomorphic to $\Q ( i )$; 
                    \item if $n > 2$ the Jacobians of these fibers are isogenous to the square of a simple abelian variety with endomorphism algebra isomorphic to $K_{n} \coloneqq \Q  \bigl( \zeta_{2^{n}} - \zeta_{2^{n}}^{-1} \bigr)$.
                \end{enumerate}
            \end{Lem}
            \begin{proof}
                Remark~\ref{rem: if L = 1, there are finitely many classes; subsec: Preliminaries; sec: Cyclic groups of order p^n} shows that the curves $y^{2} = x \cdot ( x^{2^{n-1}} - a )$ are isomorphic over $\Qbar$ as $a$ varies in $\Qbar^{*}$. Hence, the fibers of $\Cfrak \to U$ over $\Qbar$-points of $U$ belong to a single isomorphism class of curves over $\Qbar$.

                In Lemma~\ref{lem: if L = 1 then Jbf is of CM type; subsec: Preliminaries; sec: Cyclic groups of order p^n} we proved that $\Jbf^{\mathrm{al}}$ has CM by $\Q ( \zeta_{2^{n}} )$, and in Lemma~\ref{lem: stabilizer of CM type; subsec: Preliminaries; sec: Cyclic groups of order p^n} we computed the stabilizer of its CM type. Now it suffices to apply \cite[Theorem~3.5]{LangCM}.
            \end{proof}

            \begin{Rem}
            \label{rem: Kn not in Kn+1; subsec: The case p=2; sec: Cyclic groups of order p^n}
                If $n > 2$, then $K_{n+1}$ does not contain any subfield isomorphic to $K_{n}$. Indeed, let $S$ be a field isomorphic to $K_{n}$. There exists $s \in S$ whose minimal polynomial over $\Q$ coincides with that of $\zeta_{2^{n}} - \zeta_{2^{n}}^{-1}$; hence, $s = \zeta_{2^{n}}^{a} - \zeta_{2^{n}}^{-a}$ for some odd integer $a$. Let $\sigma \in \Gal \bigl( \Q ( \zeta_{2^{n+1}} ) / \Q \bigr)$ be the automorphism associated with the integer $2^{n} - 1$, and note that $K_{n+1}$ is the subfield of $\Q ( \zeta_{2^{n+1}} )$ fixed by $\sigma$. Since $\sigma ( s ) = -s$, it follows that $S$ cannot be a subfield of $K_{n+1}$.
            \end{Rem}

            \begin{Rem}
            \label{rem: extra automorphism L=1; subsec: The case p=2; sec: Cyclic groups of order p^n}
                In Proposition~\ref{pro: unique element of order 2 in AutC if C is simple; sec: Preliminaries} we proved that the automorphism group of a curve with genus $\ge 2$ over an algebraically closed field of characteristic zero with simple Jacobian has at most one element of order $2$. If $p = 2$ and $L = 1$, then $\Cbf^{\mathrm{al}}$ has the automorphism
                \begin{equation*}
                    \psi \colon \Cbf^{\mathrm{al}} \to \Cbf^{\mathrm{al}} \qquad\qquad\qquad ( x, y ) \mapsto \biggl( \frac{\gamma}{x}, \frac{ \beta \cdot y }{ x^{ 2^{n - 2} + 1 } } \biggr) ,
                \end{equation*}
                where $\gamma^{ 2^{n - 2} } = -\alpha_{1}$ and $\beta^{2} = -\gamma \cdot \alpha_{1}$. Moreover, if $n > 2$ the genus of $\Cbf^{\mathrm{al}}$ is $\ge 2$. Since $\psi$ has order $2$ and it is not the hyperelliptic involution of $\Cbf^{\mathrm{al}}$, Proposition~\ref{pro: unique element of order 2 in AutC if C is simple; sec: Preliminaries} gives another proof of the non-simplicity of $\Jbf^{\mathrm{al}}$.
            \end{Rem}

            We find that $\Jbf^{\mathrm{al}}$ is also not simple when $p = 2$ and $L = 2$.

            \begin{Lem}
            \label{lem: L = 2; subsec: The case p=2; sec: Cyclic groups of order p^n}
                When $p = 2$ and $L = 2$, the Jacobian $\Jbf^{\mathrm{al}}$ is isogenous to the square of a simple abelian variety with endomorphism algebra isomorphic to a subfield $F_{n}$ of $K_{n+1} = \Q  \bigl( \zeta_{2^{n+1}} - \zeta_{2^{n+1}}^{-1} \bigr)$ of index $2$.
            \end{Lem}
            \begin{proof}
                Let $\Jbf^{\mathrm{al}} \sim X_{1} \times X_{2} \times \cdots \times X_{A}$ be an isogeny decomposition of $\Jbf^{\mathrm{al}}$, where some $X_{j}$ may be isogenous to each other. We regard the $X_{j}$ as abelian subvarieties of $\Jbf^{\mathrm{al}}$. Let us prove that $A = 2$.

                Consider the specialization $\Jfrak_{ ( 1, -1 ) }$ of $\Jfrak \to U$, namely, the Jacobian of the curve $y^{2} = x \cdot ( x^{ 2^{n - 1} } - 1 ) \cdot ( x^{ 2^{n - 1} } + 1 ) = x \cdot ( x^{2^{n}} - 1 )$. Lemma~\ref{lem: L = 1; subsec: The case p=2; sec: Cyclic groups of order p^n} implies that $\Jfrak_{ ( 1, -1 ) }$ is isogenous to the square of a simple abelian variety $Y$, of dimension $\varphi ( 2^{n} ) / 2$ (Remark~\ref{rem: genus of Cbf; subsec: Setup; sec: Cyclic groups of order p^n}), with endomorphism algebra isomorphic to $K_{n+1}$. In Lemma~\ref{lem: isogeny decomposition specializes; subsec: specialization methods; sec: Preliminaries} we proved that the number of factors in an isogeny decomposition of $\Jbf^{\mathrm{al}}$ is at most that of an isogeny decomposition of $\Jfrak_{ ( 1, -1 ) }$; hence, $A \le 2$.

                Consider the automorphism of order $2$ of $\Cbf^{\mathrm{al}}$
                \begin{equation*}
                    \psi \colon \Cbf^{\mathrm{al}} \to \Cbf^{\mathrm{al}} \qquad\qquad\qquad ( x, y ) \mapsto \biggl( \frac{\gamma}{x}, \frac{\beta \cdot y}{ x^{ 2^{n - 1} + 1 } } \biggr) ,
                \end{equation*}
                where $\gamma^{ 2^{n - 1} } = \alpha_{1} \cdot \alpha_{2}$ and $\beta^{2} = \gamma \cdot \alpha_{1} \cdot \alpha_{2}$. In Proposition~\ref{pro: unique element of order 2 in AutC if C is simple; sec: Preliminaries} we proved that the automorphism group of a curve over an algebraically closed field of characteristic zero, with genus $\ge 2$ and simple Jacobian, has at most one element of order $2$. Since $\psi$ is not the hyperelliptic involution of $\Cbf^{\mathrm{al}}$, Proposition~\ref{pro: unique element of order 2 in AutC if C is simple; sec: Preliminaries} implies that $\Jbf^{\mathrm{al}}$ is not simple; hence, $A = 2$.
                    
                In Lemma~\ref{lem: End0 embeds and torsion bijects under specialization; subsec: specialization methods; sec: Preliminaries} we proved that an abelian subvariety of $\Jbf^{\mathrm{al}}$ specializes to an abelian subvariety of $\Jfrak_{ ( 1, -1 ) }$ of the same dimension. All the proper abelian subvarieties of $\Jfrak_{ (1, -1 ) }$ are isogenous to $Y$; hence, $\dim ( X_{1} ) = \dim ( X_{2} ) = \dim ( Y ) = \varphi ( 2^{n} ) / 2$. If $X_{1}$ and $X_{2}$ are not isogenous, both $\End^{0} ( X_{1} )$ and $\End^{0} ( X_{2} )$ contain a field isomorphic to $\Q ( \zeta_{2^{n}} )$ (last statement of Lemma~\ref{lem: End0Jbf contains Q(zetapn); subsec: Preliminaries; sec: Cyclic groups of order p^n}), which has degree $2 \dim ( X_{1} ) = 2 \dim ( X_{2} )$; hence, both $X_{1}$ and $X_{2}$ have CM by $\Q ( \zeta_{2^{n}} )$, which contradicts Lemma~\ref{lem: Cbf is not of CM type if L > 1; subsec: Preliminaries; sec: Cyclic groups of order p^n}. Hence, $X_{1}$ and $X_{2}$ are isogenous.     
                    
                Finally, we show that $\End^{0} ( X_{1} )$ is isomorphic to a subfield of $K_{n+1}$ of index $2$. By applying the first statement of Lemma~\ref{lem: End0 embeds and torsion bijects under specialization; subsec: specialization methods; sec: Preliminaries} with $\Jfrak_{ ( 1, -1 ) }$, we deduce that $\End^{0} ( \Jbf^{\mathrm{al}} )$ injects into $\M ( 2, K_{n+1} )$. Furthermore, since $X_{1}$ specializes to an abelian subvariety of $\Jfrak_{ ( 1, -1 ) }$ isogenous to $Y$, we have that $\End^{0} ( X_{1} )$ injects into $K_{n+1}$; hence, $\End^{0} ( X_{1} )$ is a subfield of $K_{n+1}$.

                Since $\bigl[ \M ( 2, K_{n+1} ) \colon \Q \bigr] = 2^{n+1}$, we deduce that $\bigl[ \End^{0} ( \Jbf^{\mathrm{al}} ) \colon \Q \bigr] \le 2^{n+1}$. Given that $\End^{0} ( \Jbf^{\mathrm{al}} )$ is a $\Q ( \Phi )$-module and $\Q ( \Phi ) \cong \Q ( \zeta_{2^{n}} )$ (first statement of Lemma~\ref{lem: End0Jbf contains Q(zetapn); subsec: Preliminaries; sec: Cyclic groups of order p^n}), it follows that $2^{n-1}$ divides $\bigl[ \End^{0} ( \Jbf^{\mathrm{al}} ) \colon \Q \bigr]$. In particular, $\bigl[ \End^{0} ( \Jbf^{\mathrm{al}} ) \colon \Q \bigr]$ can only be $2^{n-1}$, $2^{n}$, $3 \cdot 2^{n-1}$ or $2^{n+1}$. Moreover, if $n = 2$ then $\bigl[ \End^{0} ( \Jbf^{\mathrm{al}} ) \colon \Q \bigr]$ can only be $2^{n}$ or $2^{n+1}$. Indeed, since $\Jbf^{\mathrm{al}}$ is isogenous to the square of a simple abelian variety, $\End^{0} ( \Jbf^{\mathrm{al}} )$ is isomorphic to a $2 \times 2$ matrix algebra, which implies that $4$ divides $\bigl[ \End^{0} ( \Jbf^{\mathrm{al}} ) \colon \Q \bigr]$.
                \begin{itemize}
                    \item If $\bigl[ \End^{0} ( \Jbf^{\mathrm{al}} ) \colon \Q \bigr] = 2^{n-1}$, then $\End^{0} ( \Jbf^{\mathrm{al}} ) \cong \Q ( \zeta_{2^{n}} )$, which contradicts $A = 2$.
                    \item If $\bigl[ \End^{0} ( \Jbf^{\mathrm{al}} ) \colon \Q \bigr] = 2^{n}$, then $\bigl[ \End^{0} ( X_{1} ) \colon \Q \bigr] = 2^{n-2}$, hence $\End^{0} ( X_{1} )$ is isomorphic to a subfield of $K_{n+1}$ of index $2$, as desired.
                    \item If $\bigl[ \End^{0} ( \Jbf^{\mathrm{al}} ) \colon \Q \bigr] = 3 \cdot 2^{n-1}$, then $n > 2$ and $\bigl[ \End^{0} ( X_{1} ) \colon \Q \bigr] = 3 \cdot 2^{n-3}$. Since $X_{1}$ is simple and $\bigl[ \End^{0} ( X_{1} ) \colon \Q \bigr]$ does not divide $2 \dim ( X_{1} ) = 2^{n-1}$, this contradicts Lemma~\ref{lem: divisibility dimension and degree End simple abelian varieties; subsec: Abelian varieties; sec: Preliminaries}.
                    \item If $\bigl[ \End^{0} ( \Jbf^{\mathrm{al}} ) \colon \Q \bigr] = 2^{n+1}$, then $\End^{0} ( \Jbf^{\mathrm{al}} ) \cong \M ( 2, \, K_{n+1} )$ and $\End^{0} ( X_{1} ) \cong K_{n+1}$. Since $\dim ( X_{1} ) = 2^{n-2}$, it follows that $X_{1}$ has CM by $K_{n+1}$, which contradicts Lemma~\ref{lem: Cbf is not of CM type if L > 1; subsec: Preliminaries; sec: Cyclic groups of order p^n}. 
                \end{itemize}
            \end{proof}

            \begin{Rem}
            \label{rem: F2 = Q; subsec: The case p=2; sec: Cyclic groups of order p^n}
                Note that $F_{2} = \Q$. Indeed, $F_{2}$ is a subfield of index $2$ of $K_{3} = \Q ( \zeta_{8} - \zeta_{8}^{-1} ) = \Q ( \sqrt{-2} )$.
            \end{Rem}

            \begin{Rem}
            \label{rem: Emiliano; subsec: The case p=2; sec: Cyclic groups of order p^n}
                Explicit abelian varieties with endomorphism algebra $\Q(\zeta_n)$ have been used in \cite{EmilianoJac4} to construct counterexamples to the local-global principle for quadratic twists. The dimension of the smallest example constructed in \cite[Theorem 1.2]{EmilianoJac4} is 12. Using the techniques of \cite{EmilianoJac4}, a counterexample of minimal dimension (namely, 4) requires an abelian fourfold over $\Q$ (or more generally a number field containing only $\pm 1$ as roots of unity) with geometric endomorphism algebra $\Q(\zeta_{16})$. It is not known whether such a variety can be realized as a Jacobian. However, the Prym variety $A/\Q$ associated with the map of curves $(x,y) \mapsto (x,y^2)$ from $y^{16} = x \cdot ( x - 1 ) ^{2}$ to $y^{8} = x \cdot ( x - 1 ) ^{2}$ has these properties. The field of definition of the geometric endomorphisms of $A$ is easily seen to be $\Q(\zeta_{16})$. Applying \cite[Proposition 4.3]{EmilianoJac4} to $A$ (with $m=8$, $\alpha=16$ and $S = \{ 2 \}$), we obtain that there are two $4$-dimensional, geometrically simple abelian varieties over $\Q$ that are strongly locally quadratic twists, but not quadratic twists of each other. In turn, this shows that the condition $\dim A \leq 3$ in \cite{zbMATH07796282} is sharp, even for geometrically simple abelian varieties.

                The theory we have developed so far imposes several constraints on the existence of a $4$-dimensional Jacobian with geometric endomorphism algebra $\Q(\zeta_{16})$. Indeed, let $\Ccal$ be a genus $4$ curve over $\Qbar$ whose Jacobian has endomorphism algebra isomorphic to $\Q ( \zeta_{16} )$. Since the group of roots of unity of $\Q ( \zeta_{16} )$ is isomorphic to $C_{16}$, \cite[Theorem~$12.1$, \S~III]{milneAV} implies that $\Aut ( \Ccal )$ injects into $C_{16}$. In Section~\ref{subsec: Origin of the family; sec: Cyclic groups of order p^n}, we found a versal model for a genus $4$ curve over $\Qbar$ with simple Jacobian and automorphism group isomorphic to $C_{2^{n}}$, for $n > 1$: 
                \begin{itemize}
                    \item if $\Aut ( \Ccal )$ is isomorphic to $C_{8}$, then $\M ( 2, F_{3} )$ injects into the endomorphism algebra of $\Jac ( \Ccal )$ (Lemmas~\ref{lem: L = 2; subsec: The case p=2; sec: Cyclic groups of order p^n} and~\ref{lem: End0 embeds and torsion bijects under specialization; subsec: specialization methods; sec: Preliminaries}). Since the $\Q$-dimension of $\M ( 2, F_{3} )$ is equal to that of $\Q ( \zeta_{16} )$, we get a contradiction;
                    \item if $\Aut ( \Ccal )$ is isomorphic to $C_{16}$, then $\M ( 2, K_{3} )$ injects into the endomorphism algebra of $\Jac ( \Ccal )$ (Lemmas~\ref{lem: L = 1; subsec: The case p=2; sec: Cyclic groups of order p^n} and~\ref{lem: End0 embeds and torsion bijects under specialization; subsec: specialization methods; sec: Preliminaries}). Since the $\Q$-dimension of $\M ( 2, K_{3} )$ is equal to that of $\Q ( \zeta_{16} )$, we get a contradiction.
                \end{itemize}
            \end{Rem}
                
            The previous lemmas show that for $L \le 2$ the generic fiber $\Jbf^{\mathrm{al}}$ is isogenous to the square of an abelian variety. We now prove that the situation changes for $L = 3$: in this case, $\Jbf^{\mathrm{al}}$ is simple.
                
            \begin{Lem}
            \label{lem: L = 3; subsec: The case p=2; sec: Cyclic groups of order p^n}
                Theorem~\ref{the: main; sec: Cyclic groups of order p^n} holds for $p = 2$ and $L = 3$.
            \end{Lem}
            \begin{proof}
                Let $\Jbf^{\mathrm{al}} \sim Z_{1} \times Z_{2} \times \cdots \times Z_{A}$ be the decomposition of $\Jbf^{\mathrm{al}}$ into isotypic components (up to isogeny). We regard the $Z_{j}$ as abelian subvarieties of $\Jbf^{\mathrm{al}}$.

                In Lemma~\ref{lem: useful to apply cor: exclusion cases with torsion arguments; subsec: specialization methods; sec: Cyclic groups of order p^n} we proved that there exists  $( a_{1}, a_{2}, a_{3} ) \in U ( \Qbar )$ such that the endomorphism algebras of the Jacobians of both curves $\Ccal_{1} \colon y^{2} = x \cdot ( x^{2^{n-1}} - a_{1} ) \cdot ( x^{2^{n-1}} - a_{2} )$ and $\Ccal_{3} \colon y^{2} = x \cdot ( x^{2^{n-1}} - a_{2} ) \cdot ( x^{2^{n-1}} - a_{3} )$ are isomorphic to that of the Jacobian of the generic fiber of $y^{2} = x \cdot ( x^{2^{n-1}} - \alpha_{1} ) \cdot ( x^{2^{n-1}} - \alpha_{2} )$. Lemma~\ref{lem: L = 2; subsec: The case p=2; sec: Cyclic groups of order p^n} implies that the Jacobians of $\Ccal_{1}$ and $\Ccal_{3}$ are isogenous to the square of a simple abelian variety $X$, of dimension $\varphi ( 2^{n} ) / 2$ (Remark~\ref{rem: genus of Cbf; subsec: Setup; sec: Cyclic groups of order p^n}), with endomorphism algebra isomorphic to $F_{n}$, where $F_{n}$ is a subfield of $K_{n+1} = \Q \bigl( \zeta_{2^{n+1}} - \zeta_{2^{n+1}}^{-1} \bigr)$ of index $2$.
                    
                In Lemma~\ref{lem: L = 1; subsec: The case p=2; sec: Cyclic groups of order p^n} we proved that both curves $\Ccal_{2} \colon y^{2} = x \cdot ( x^{2^{n-1}} - a_{3} )$ and $\Ccal_{4} \colon y^{2} = x \cdot ( x^{2^{n-1}} - a_{1} )$ have Jacobians isogenous
                \begin{itemize}
                    \item to the same elliptic curve $E$, of dimension $\varphi ( 2^{n} ) / 2$ (Remark~\ref{rem: genus of Cbf; subsec: Setup; sec: Cyclic groups of order p^n}), with endomorphism algebra isomorphic to $\Q ( i )$ if $n = 2$;
                    \item to the square of the same simple abelian variety $Y$, of dimension $\varphi ( 2^{n} ) / 4$ (Remark~\ref{rem: genus of Cbf; subsec: Setup; sec: Cyclic groups of order p^n}), with endomorphism algebra isomorphic to $K_{n}$ if $n > 2$.
                \end{itemize}

                By applying Proposition~\ref{pro: particular specialization; subsec: specialization methods; sec: Cyclic groups of order p^n} with $( a_{1}, a_{2}, a_{3} ) \in U ( \Qbar )$ and $S = \{ 3 \}$, we deduce the existence of an abelian variety $\Jcal_{1}$, of the form $\Jfrak_{T', t}$ for some admissible $T'$ and $t \in \Abb_{\Qbar}^{1} ( \Qbar ) \supset T' ( \Qbar )$, isomorphic to $\Jac ( \Ccal_{1} ) \times \Jac ( \Ccal_{2} )$. Hence
                \begin{equation*}
                    \Jcal_{1} \sim 
                    \begin{cases}
                        X^{2} \times E & \text{if $n = 2$,} \\
                        X^{2} \times Y^{2} & \text{if $n > 2$.}
                    \end{cases}
                \end{equation*}
                    
                Note that $\Jac ( \Ccal_{1} )$ and $\Jac ( \Ccal_{2} )$ are two distinct isotypic components of $\Jcal_{1}$. Indeed, if $n = 2$, then $\Jac ( \Ccal_{2} )$ is simple and not isogenous to $X$, since $\End^{0} ( X ) \cong F_{2} = \Q$ (Remark~\ref{rem: F2 = Q; subsec: The case p=2; sec: Cyclic groups of order p^n}), while $\End^{0} \bigl( \Jac( \Ccal_{2} ) \bigr) \cong \Q ( i )$. If $n > 2$, then $X$ is not isogenous to $Y$ because they are simple and have different dimensions. Hence
                \begin{equation}
                \label{eq: 1; lem: L = 3; subsec: The case p=2; sec: Cyclic groups of order p^n}
                    \End^{0} ( \Jcal_{1} ) \cong 
                    \begin{cases}
                        \M ( 2, \Q ) \times \Q(i) & \text{if $n = 2$,} \\
                        \M ( 2, F_{n} ) \times \M ( 2, K_{n} ) & \text{if $n > 2$.}
                    \end{cases}
                \end{equation}
                We analyze the cases $n = 2$ and $n > 2$ separately.
                    
                Suppose first that $n = 2$. Remark~\ref{rem: simple observation of lem: End0 embeds and torsion bijects under specialization; subsec: specialization methods; sec: Preliminaries} shows that the endomorphism algebra of any isotypic component of $\Jbf^{\mathrm{al}}$ embeds into the endomorphism algebra of at least one isotypic component of $\Jcal_{1}$; hence, \eqref{eq: 1; lem: L = 3; subsec: The case p=2; sec: Cyclic groups of order p^n} implies that it embeds into either $\Q ( i )$ or $\M ( 2, \Q )$. On the other hand, in Lemma~\ref{lem: End0Jbf contains Q(zetapn); subsec: Preliminaries; sec: Cyclic groups of order p^n} we proved that the endomorphism algebra of each isotypic component of $\Jbf^{\mathrm{al}}$ contains a field isomorphic to $\Q ( i )$. Since $\Q ( i )$ has $\Q$-dimension $2$ and $\M ( 2, \Q )$ has $\Q$-dimension $4$, it follows that $\End^{0} ( Z_{j} )$ is isomorphic to either $\Q ( i )$ or $\M ( 2, \Q )$ for every $1 \le j \le A$. The first statement of Lemma~\ref{lem: End0 embeds and torsion bijects under specialization; subsec: specialization methods; sec: Preliminaries} applied to $\Jcal_{1}$ implies that $\End^{0} ( \Jbf^{\mathrm{al}} )$ injects into $\Q ( i ) \times \M ( 2, \Q )$, which has $\Q$-dimension $6$; hence, $A \le 6 / 2 = 3$.

                If $A = 3$, then, by a dimension argument, $\End^{0} ( Z_{j} )$ is isomorphic to $\Q ( i )$ for every $1 \le j \le 3$. Since $\Q ( i ) \times \Q ( i ) \times \Q ( i )$ does not embed into $\End^{0} ( \Jcal_{1} )$, we reach a contradiction.

                Let us exclude the case $A = 2$. Up to reordering $Z_{1}$ and $Z_{2}$, we may assume that $\dim ( Z_{1} ) = 1$ and $\dim ( Z_{2} ) = 2$, since $\dim ( \Jbf^{\mathrm{al}} ) = 3$ (Remark~\ref{rem: genus of Cbf; subsec: Setup; sec: Cyclic groups of order p^n}). Moreover, $\End^{0} ( Z_{1} ) \cong \Q ( i )$ because $Z_{1}$ is simple. The only abelian subvarieties of $\Jcal_{1}$ of dimension $1$ are isogenous to either $E$ or $X$. Note that $Z_{1}$ cannot specialize to an abelian subvariety isogenous to $X$, as this would imply the existence of an embedding of $\End^{0} ( Z_{1} ) \cong \Q ( i )$ into $\End^{0} ( X ) \cong \Q$, a contradiction. In Lemma~\ref{lem: isogeny decomposition specializes; subsec: specialization methods; sec: Preliminaries} we proved that a decomposition (up to isogeny) of $\Jbf^{\mathrm{al}}$ specializes to a decomposition (up to isogeny) of $\Jcal_{1}$. Since $Z_{1}$ specializes to $E \sim \Jac ( \Ccal_{2} )$ in $\Jcal_{1}$, $Z_{2}$ specializes to $X^{2} \sim \Jac ( \Ccal_{1} )$ in $\Jcal_{1}$.

                By applying Proposition~\ref{pro: particular specialization; subsec: specialization methods; sec: Cyclic groups of order p^n} with $( a_{1}, a_{2}, a_{3} ) \in U ( \Qbar )$ and $S = \{ 1 \}$, we deduce the existence of an abelian variety $\Jcal_{2}$, of the form $\Jfrak_{T', t}$ for some admissible $T'$ and $t \in \Abb_{\Qbar}^{1} ( \Qbar ) \supset T' ( \Qbar )$, isomorphic to $\Jac ( \Ccal_{3} ) \times \Jac ( \Ccal_{4} )$. As before, we show that $Z_{1}$ specializes to $\Jac ( \Ccal_{4} )$ and $Z_{2}$ specializes to $\Jac ( \Ccal_{3} )$ in $\Jcal_{2}$. This contradicts Corollary~\ref{cor: exclusion cases with torsion arguments; subsec: specialization methods; sec: Cyclic groups of order p^n}; hence, we must have $A = 1$, given that $A \le 2$.

                Since $A = 1$, $\End^{0} ( \Jbf^{\mathrm{al}} )$ is isomorphic to either $\Q ( i )$ or $\M ( 2, \Q )$. The latter case, however, cannot occur because $\dim ( \Jbf^{\mathrm{al}} )$ is odd. Hence, $\Jbf^{\mathrm{al}}$ is simple and $\End^{0} ( \Jbf^{\mathrm{al}} )$ is isomorphic to $\Q ( i )$. Theorems~\ref{the: reduction first statement main the; subsec: Preliminaries; sec: Cyclic groups of order p^n} and~\ref{the: there are infinitely many classes; subsec: Preliminaries; sec: Cyclic groups of order p^n} imply that Theorem \ref{the: main; sec: Cyclic groups of order p^n} holds for this subcase. 

                Suppose now that $n > 2$. As above, we deduce that $\End^{0} ( Z_{j} )$ is isomorphic to either $\Q ( \zeta_{2^{n}} )$, $\M ( 2, F_{n} )$ or $\M ( 2, K_{n} )$ for every $1 \le j \le A$. The first statement of Lemma~\ref{lem: End0 embeds and torsion bijects under specialization; subsec: specialization methods; sec: Preliminaries} applied to $\Jcal_{1}$ implies that $\End^{0} ( \Jbf^{\mathrm{al}} )$ injects into $\M ( 2, F_{n} ) \times \M( 2, K_{n} )$, which has $\Q$-dimension $2^{n+1}$. Since $\Q ( \zeta_{2^{n}} )$ has $\Q$-dimension $2^{n-1}$, and both $\M ( 2, F_{n} ),  \M( 2, K_{n} )$ have $\Q$-dimension $2^{n}$, it follows that $A \le 2^{n+1} / 2^{n-1} = 4$.

                As before, we exclude the case $A = 4$ and, in the case $A = 3$, we deduce that each $\End^{0} ( Z_{j} )$ is isomorphic to $\Q ( \zeta_{2^{n}} )$. So, if $A = 3$ each $Z_{j}$ is simple of dimension a multiple of $\bigl[ \Q( \zeta_{2^{n}} ) \colon \Q \bigr] / 2 = 2^{n-2}$ (Lemma~\ref{lem: divisibility dimension and degree End simple abelian varieties; subsec: Abelian varieties; sec: Preliminaries}). Since $\dim ( \Jbf^{\mathrm{al}} ) = 3 \cdot 2^{n-2}$ (Remark~\ref{rem: genus of Cbf; subsec: Setup; sec: Cyclic groups of order p^n}), it follows that each $Z_{j}$ has dimension $2^{n-2}$, half the $\Q$-dimension of $\End^{0} ( Z_{j} )$. Hence, every $Z_{j}$ has CM by $\Q ( \zeta_{2^{n}} )$, which contradicts Lemma~\ref{lem: Cbf is not of CM type if L > 1; subsec: Preliminaries; sec: Cyclic groups of order p^n}.

                Let us exclude the case $A = 2$. Note that either $\End^{0} ( Z_{1} ) \cong \End^{0} ( Z_{2} ) \cong \Q ( \zeta_{2^{n}} )$ or, up to reordering $Z_{1}$ and $Z_{2}$, we may assume $\End^{0} ( Z_{1} ) \cong \M ( 2, K_{n} )$ (resp.~$\End^{0} ( Z_{1} ) \cong \M ( 2, F_{n} )$). In the latter case, $Z_{1}$ specializes to $\Jac ( \Ccal_{2} )$ (resp.~$\Jac ( \Ccal_{1} )$) in $\Jcal_{1}$. In Lemma~\ref{lem: isogeny decomposition specializes; subsec: specialization methods; sec: Preliminaries} we proved that a decomposition (up to isogeny) of $\Jbf^{\mathrm{al}}$ specializes to a decomposition (up to isogeny) of $\Jcal_{1}$. Consequently, $Z_{2}$ specializes to $\Jac ( \Ccal_{1} )$ (resp.~$\Jac ( \Ccal_{2} )$). We derive a contradiction as before by using Corollary~\ref{cor: exclusion cases with torsion arguments; subsec: specialization methods; sec: Cyclic groups of order p^n}, together with the specialization obtained by applying Proposition~\ref{pro: particular specialization; subsec: specialization methods; sec: Cyclic groups of order p^n} with $( a_{1}, a_{2}, a_{3} ) \in U ( \Qbar )$ and $S = \{ 1 \}$.

                If $\End^{0} ( Z_{1} ) \cong \End^{0} ( Z_{2} ) \cong \Q ( \zeta_{2^{n}} )$, we can derive a similar contradiction. First, $Z_{1}$ and $Z_{2}$ are simple; hence, their dimensions are multiples of $\bigl[ \Q ( \zeta_{2^{n}} ) \colon \Q \bigr] / 2 = 2^{n-2}$ (Lemma~\ref{lem: divisibility dimension and degree End simple abelian varieties; subsec: Abelian varieties; sec: Preliminaries}). Up to reordering $Z_{1}$ and $Z_{2}$, we may assume that $\dim ( Z_{1} ) = 2 \cdot 2^{n-2} = 2^{n-1}$ and $\dim ( Z_{2} ) = 2^{n-2}$, since $\dim ( \Jbf^{\mathrm{al}} ) = 3 \cdot 2^{n-2}$. The nontrivial abelian subvarieties of $\Jcal_{1}$ of dimension $2^{n-1}$ are, up to isogeny, $X^{2}$ and $X \times Y^{2}$, while those of dimension $2^{n-2}$ are, up to isogeny, $X$ and $Y^{2}$. Note that $Z_{1}$ cannot specialize to an abelian subvariety isogenous to $X \times Y^{2}$ in $\Jcal_{1} \sim X^{2} \times Y^{2}$, as this would imply, by Lemma~\ref{lem: isogeny decomposition specializes; subsec: specialization methods; sec: Preliminaries}, that $Z_{2}$ specializes to an abelian subvariety isogenous to $X$, yielding an embedding of $\End^{0} ( Z_{2} ) \cong \Q( \zeta_{2^{n}} )$ into $\End^{0} ( X ) \cong F_{n}$, which has $\Q$-dimension $2^{n-2}$. Hence, $Z_{1}$ specializes to $X^{2} \sim \Jac ( \Ccal_{1} )$ and $Z_{2}$ specializes to $Y^{2} \sim \Jac ( \Ccal_{2} )$ in $\Jcal_{1}$. We derive a contradiction as before by using Corollary~\ref{cor: exclusion cases with torsion arguments; subsec: specialization methods; sec: Cyclic groups of order p^n}, together with the specialization obtained by applying Proposition~\ref{pro: particular specialization; subsec: specialization methods; sec: Cyclic groups of order p^n} with $( a_{1}, a_{2}, a_{3} ) \in U ( \Qbar )$ and $S = \{ 1 \}$.

                Hence, $A = 1$ and $\Jbf^{\mathrm{al}}$ is isotypic. Remark~\ref{rem: simple observation of lem: End0 embeds and torsion bijects under specialization; subsec: specialization methods; sec: Preliminaries} shows that if $\Jbf^{\mathrm{al}}$ is isotypic its endomorphism algebra embeds into the endomorphism algebra of every isotypic component of $\Jcal_{1}$. By \eqref{eq: 1; lem: L = 3; subsec: The case p=2; sec: Cyclic groups of order p^n} we deduce that $\End^{0} ( \Jbf^{\mathrm{al}} )$ embeds into both $\M ( 2, F_{n} )$ and $\M ( 2, K_{n} )$, which both have $\Q$-dimension $2^{n}$. Since $K_{n}$ is not isomorphic to $F_{n}$ (Remark~\ref{rem: Kn not in Kn+1; subsec: The case p=2; sec: Cyclic groups of order p^n}), we deduce that $\End^{0} ( \Jbf^{\mathrm{al}} )$ has $\Q$-dimension at most $2^{n-1}$. Since we know that $\End^{0} ( \Jbf^{\mathrm{al}} )$ is isomorphic to either $\Q ( \zeta_{2^{n}} )$, $\M ( 2, F_{n} )$ or $\M ( 2, K_{n} )$, it follows that it is isomorphic to $\Q ( \zeta_{2^{n}} )$. Theorems~\ref{the: reduction first statement main the; subsec: Preliminaries; sec: Cyclic groups of order p^n} and~\ref{the: there are infinitely many classes; subsec: Preliminaries; sec: Cyclic groups of order p^n} imply that Theorem \ref{the: main; sec: Cyclic groups of order p^n} holds for this subcase. 
            \end{proof}

            Finally, we complete the proof of Theorem~\ref{the: main; sec: Cyclic groups of order p^n}.

            \begin{proof}[Proof of Theorem~\ref{the: main; sec: Cyclic groups of order p^n} when $p = 2$]
                We proceed by induction on $L \ge 3$. Lemma~\ref{lem: L = 3; subsec: The case p=2; sec: Cyclic groups of order p^n} proves the claim for $L = 3$. So, let us assume that $L \ge 3$ and that the claim holds for every $3 \le L' \le L$, and let us prove it for $L + 1$. 
                
                Let $\Jbf^{\mathrm{al}} \sim Z_{1} \times Z_{2} \times \cdots \times Z_{A}$ be the decomposition of $\Jbf^{\mathrm{al}}$ into isotypic components (up to isogeny). We regard the $Z_{j}$ as abelian subvarieties of $\Jbf^{\mathrm{al}}$.

                In Lemma~\ref{lem: useful to apply cor: exclusion cases with torsion arguments; subsec: specialization methods; sec: Cyclic groups of order p^n} we proved that there exists $( a_{1}, \dots, a_{L+1} ) \in U ( \Qbar )$ such that the endomorphism algebras of the Jacobians of both curves $\Ccal_{1} \colon y^{2} = x \cdot ( x^{2^{n-1}} - a_{1} ) \cdots ( x^{2^{n-1}} - a_{L} )$ and $\Ccal_{3} \colon y^{2} = x \cdot ( x^{2^{n-1}} - a_{2} ) \cdots ( x^{2^{n-1}} - a_{L+1} )$ are isomorphic to that of the Jacobian of the generic fiber of $y^{2} = x \cdot ( x^{2^{n-1}} - \alpha_{1} ) \cdots ( x^{2^{n-1}} - \alpha_{L} )$. The inductive hypothesis implies that the Jacobians of $\Ccal_{1}$ and $\Ccal_{3}$ are simple, of dimension $\varphi ( 2^{n} ) \cdot L / 2$ (Remark~\ref{rem: genus of Cbf; subsec: Setup; sec: Cyclic groups of order p^n}), with endomorphism algebras isomorphic to $\Q ( \zeta_{2^{n}} )$.

                In Lemma~\ref{lem: L = 1; subsec: The case p=2; sec: Cyclic groups of order p^n} we proved that both curves $\Ccal_{2} \colon y^{2} = x \cdot ( x^{2^{n-1}} - a_{L+1} )$ and $\Ccal_{4} \colon y^{2} = x \cdot ( x^{2^{n-1}} - a_{1} )$ have Jacobians isogenous
                \begin{itemize}
                    \item to the same elliptic curve $E$, of dimension $\varphi ( 2^{n} ) / 2$ (Remark~\ref{rem: genus of Cbf; subsec: Setup; sec: Cyclic groups of order p^n}), with endomorphism algebra isomorphic to $\Q ( i )$ if $n = 2$;
                    \item to the square of the same simple abelian variety $Y$, of dimension $\varphi ( 2^{n} ) / 4$ (Remark~\ref{rem: genus of Cbf; subsec: Setup; sec: Cyclic groups of order p^n}), with endomorphism algebra isomorphic to $K_{n}$ if $n > 2$.
                \end{itemize}

                By applying Proposition~\ref{pro: particular specialization; subsec: specialization methods; sec: Cyclic groups of order p^n} with $( a_{1}, \dots, a_{L+1} ) \in U ( \Qbar )$ and $S = \{ L+1 \}$, we deduce the existence of an abelian variety $\Jcal_{1}$, of the form $\Jfrak_{T', t}$ for some admissible $T'$ and $t \in \Abb_{\Qbar}^{1} ( \Qbar ) \supset T' ( \Qbar )$, isomorphic to $\Jac ( \Ccal_{1} ) \times \Jac ( \Ccal_{2} )$. Moreover, since $\Jac ( \Ccal_{1} )$ is simple and has dimension strictly greater than that of $\Jac ( \Ccal_{2} )$, it follows that $\Jac ( \Ccal_{1} )$ and $\Jac ( \Ccal_{2} )$ are two distinct isotypic components of $\Jcal_{1}$. Hence
                \begin{equation}
                \label{eq: 1; proof of main the; subsec: The case p=2; sec: Cyclic groups of order p^n}
                    \End^{0} ( \Jcal_{1} ) \cong
                    \begin{cases}
                        \Q ( i ) \times \Q ( i ) & \text{if $n = 2$,} \\
                        \Q ( \zeta_{2^{n}} ) \times \M ( 2, K_{n} ) & \text{if $n > 2$.}
                    \end{cases}
                \end{equation}
                We analyze the cases $n = 2$ and $n > 2$ separately. Since the arguments are very similar to those used in the proof of Lemma~\ref{lem: L = 3; subsec: The case p=2; sec: Cyclic groups of order p^n}, we shall omit some of the details.

                Suppose first that $n = 2$. Remark~\ref{rem: simple observation of lem: End0 embeds and torsion bijects under specialization; subsec: specialization methods; sec: Preliminaries} shows that the endomorphism algebra of any isotypic component of $\Jbf^{\mathrm{al}}$ embeds into the endomorphism algebra of at least one isotypic component of $\Jcal_{1}$; hence, \eqref{eq: 1; proof of main the; subsec: The case p=2; sec: Cyclic groups of order p^n} implies that it embeds into $\Q ( i )$. On the other hand, in Lemma~\ref{lem: End0Jbf contains Q(zetapn); subsec: Preliminaries; sec: Cyclic groups of order p^n} we proved that the endomorphism algebra of each isotypic component of $\Jbf^{\mathrm{al}}$ contains a field isomorphic to $\Q ( i )$. Hence, $\End^{0} ( Z_{j} )$ is isomorphic to $\Q ( i )$ for every $1 \le j \le A$. The first statement of Lemma~\ref{lem: End0 embeds and torsion bijects under specialization; subsec: specialization methods; sec: Preliminaries} applied to $\Jcal_{1}$ implies that $\End^{0} ( \Jbf^{\mathrm{al}} )$ injects into $\Q ( i ) \times \Q ( i )$; hence, $A \le 2$.

                Let us exclude the case $A = 2$. Note that $Z_{1}$ and $Z_{2}$ are simple. In Lemma~\ref{lem: isogeny decomposition specializes; subsec: specialization methods; sec: Preliminaries} we proved that there exists a simple factor of $\Jbf^{\mathrm{al}}$ of dimension at least the maximal dimension of a simple factor of $\Jcal_{1}$; hence, up to reordering $Z_{1}$ and $Z_{2}$, we may assume that $\dim ( Z_{1} ) \ge \dim \bigl( \Jac ( \Ccal_{1} ) \bigr) = L$. Since $\dim ( \Jbf^{\mathrm{al}} ) = L + 1$ (Remark~\ref{rem: genus of Cbf; subsec: Setup; sec: Cyclic groups of order p^n}), we deduce $\dim ( Z_{1} ) = L$ and $\dim ( Z_{2} ) = 1$. Note that $\Jcal_{1}$ contains a unique abelian subvariety of dimension $1$, namely $\Jac ( \Ccal_{2} )$, and a unique abelian subvariety of dimension $L$, namely $\Jac ( \Ccal_{1} )$. Hence, $Z_{1}$ specializes to $\Jac ( \Ccal_{1} )$ and $Z_{2}$ specializes to $\Jac ( \Ccal_{2} )$ in $\Jcal_{1}$. We derive a contradiction by using Corollary~\ref{cor: exclusion cases with torsion arguments; subsec: specialization methods; sec: Cyclic groups of order p^n}, together with the specialization obtained by applying Proposition~\ref{pro: particular specialization; subsec: specialization methods; sec: Cyclic groups of order p^n} with $( a_{1}, a_{2}, \dots, a_{L+1} ) \in U ( \Qbar )$ and $S = \{ 1 \}$. 
                    
                Hence, $A = 1$ and $\End^{0} ( \Jbf^{\mathrm{al}} )$ is isomorphic to $\Q ( i )$. Theorems~\ref{the: reduction first statement main the; subsec: Preliminaries; sec: Cyclic groups of order p^n} and~\ref{the: there are infinitely many classes; subsec: Preliminaries; sec: Cyclic groups of order p^n} imply that Theorem \ref{the: main; sec: Cyclic groups of order p^n} holds for this subcase.

                Suppose now that $n > 2$. As above, we deduce that $\End^{0} ( Z_{j} )$ is isomorphic to either $\Q ( \zeta_{2^{n}} )$ or $\M ( 2, K_{n} )$ for every $1 \le j \le A$. Since $\M ( 2, K_{n} ) \times \Q ( \zeta_{2^{n}} )$ has $\Q$-dimension $3 \cdot 2^{n-1}$, we deduce $A \le 3 \cdot 2^{n-1} / 2^{n-1} = 3$. Since $\Q ( \zeta_{2^{n}} ) \times \Q ( \zeta_{2^{n}} ) \times \Q ( \zeta_{2^{n}} )$ does not embed into $\End^{0} ( \Jcal_{1} ) \cong \Q ( \zeta_{2^{n}} ) \times \M ( 2, K_{n} )$, it follows that $A \le 2$.

                Let us exclude the case $A = 2$. Up to reordering $Z_{1}$ and $Z_{2}$, we may assume that $\End^{0} ( Z_{1} ) \cong \Q ( \zeta_{2^{n}} )$, and hence that $Z_{1}$ is simple. If $\End^{0} ( Z_{2} )$ is also isomorphic to $\Q ( \zeta_{2^{n}} )$, then $Z_{2}$ is simple, and it follows that $2^{n-2} = \bigl[ \Q ( \zeta_{2^{n}} ) \colon \Q \bigr] / 2$ divides the dimensions of $Z_{1}$ and $Z_{2}$ (Lemma~\ref{lem: divisibility dimension and degree End simple abelian varieties; subsec: Abelian varieties; sec: Preliminaries}). The second statement of Lemma~\ref{lem: isogeny decomposition specializes; subsec: specialization methods; sec: Preliminaries} applied to $\Jcal_{1}$ tells us that either $Z_{1}$ or $Z_{2}$ has a dimension greater than or equal to that of $\Jac ( \Ccal_{1} )$, which is $L \cdot 2^{n-2}$. Hence, we can assume that $\dim ( Z_{1} ) = L \cdot 2^{n-2}$ and $\dim ( Z_{2} ) = 2^{n-2}$. We derive a contradiction as before by using Corollary~\ref{cor: exclusion cases with torsion arguments; subsec: specialization methods; sec: Cyclic groups of order p^n}, together with the specialization obtained by applying Proposition~\ref{pro: particular specialization; subsec: specialization methods; sec: Cyclic groups of order p^n} with $( a_{1}, a_{2}, \dots, a_{L+1} ) \in U ( \Qbar )$ and $S = \{ 1 \}$.
                    
                If instead $\End^{0} ( Z_{2} ) \cong \M ( 2, K_{n} )$, then $Z_{2}$ is isogenous to the square of a simple abelian variety $X$. The second statement of Lemma~\ref{lem: isogeny decomposition specializes; subsec: specialization methods; sec: Preliminaries} applied to $\Jcal_{1}$ tells us that either $Z_{1}$ or $X$ has a dimension greater than or equal to that of $\Jac ( \Ccal_{1} )$, which is $L \cdot 2^{n-2}$. Since the dimension of $\Jbf^{\mathrm{al}}$ is $( L + 1 ) \cdot 2^{n-2}$ (Remark~\ref{rem: genus of Cbf; subsec: Setup; sec: Cyclic groups of order p^n}) and $2 L > L + 1$, we deduce that $\dim ( Z_{1} ) = L \cdot 2^{n-2}$ and $\dim ( Z_{2} ) = 2^{n-2}$, and we proceed as before.

                Hence, $A = 1$ and $\End^{0} ( \Jbf^{\mathrm{al}} )$ is isomorphic to either $\Q ( \zeta_{2^{n}} )$ or $\M ( 2, K_{n} )$. The latter case, however, contradicts the second statement of Lemma~\ref{lem: isogeny decomposition specializes; subsec: specialization methods; sec: Preliminaries} applied to $\Jcal_{1}$, which says that there exists a simple abelian subvariety of $\Jbf^{\mathrm{al}}$ of dimension at least $\dim \bigl( \Jac ( \Ccal_{1} ) \bigr) = L \cdot 2^{n-2}$. Hence, $\End^{0} ( \Jbf^{\mathrm{al}} )$ is isomorphic to $\Q ( \zeta_{2^{n}} )$. Theorems~\ref{the: reduction first statement main the; subsec: Preliminaries; sec: Cyclic groups of order p^n} and~\ref{the: there are infinitely many classes; subsec: Preliminaries; sec: Cyclic groups of order p^n} imply that Theorem \ref{the: main; sec: Cyclic groups of order p^n} holds for this subcase.
            \end{proof}

    \section{Cyclic groups of order \texorpdfstring{$pq$}{pq}}
    \label{sec: Cyclic groups of order pq}

        In this section, we analyze the cyclic groups $C_{pq}$, where $p, q$ are distinct primes. The Catalan curve over $\Qbar$ defined by $y^{p} = x^{q} - 1$ is an example of a curve of genus $\ge 2$ with automorphism group isomorphic to $C_{pq}$ and simple Jacobian (Theorem~\ref{the: properties of Catalan curve; sec: Preliminaries}), unless $pq = 6$, in which case the Catalan curve has genus $1$. We prove that, in fact, this is the unique curve, up to isomorphism, with these properties.

        \begin{The}
        \label{the: main; sec: Cyclic groups of order pq} 
            Let $p, q$ be distinct primes, and let $\Ccal / \Qbar$ be a curve of genus $\ge 2$ with automorphism group isomorphic to $C_{pq}$ and simple Jacobian. If $pq \neq 6$, the curve $\Ccal$ is isomorphic to the Catalan curve $y^{p} = x^{q} - 1$. When $pq = 6$, no such curve exists.
        \end{The}
        \begin{proof}
            Let $g$ be a generator of $\Aut ( \Ccal )$. Since $g^{q}$ is an element of order $p$, Proposition~\ref{pro: a curve with simple Jacobian is superelliptic; sec: Preliminaries} implies that $\Ccal$ is isomorphic to a superelliptic curve $y^{p} = f(x)$, where the projection to the $x$-coordinate corresponds to the cover $\Ccal \twoheadrightarrow \Ccal / \langle g^{q} \rangle \cong \mathbb{P}_{1, \Qbar}$. Every element of $\Aut ( \Ccal ) / \langle g^{q} \rangle \cong C_{q}$ acts on $\Ccal / \langle g^{q} \rangle$ and stabilizes the branch locus of $\Ccal \to \mathbb{P}_{1, \Qbar}$. Since every cyclic subgroup of order $n$ of $\Aut ( \mathbb{P}_{1, \Qbar} )$ is conjugate to the subgroup generated by $z \mapsto \zeta_{n} \cdot z$, up to a change of coordinates on $\mathbb{P}_{1, \Qbar}$ we may assume that
            \begin{equation*}
                f(x) = x^{n_{0}} \cdot \prod_{i = 1}^{d} \prod_{j = 1}^{q} ( x - \zeta_{q}^{j} \cdot a_{i} )^{n_{i, j}}
            \end{equation*}
            for some integer $d \ge 1$, some elements $a_{i} \in \Qbar^{*}$ such that the $a_{i}^{q}$ are pairwise distinct, some integer $0 \le n_{0} < p$ and some integers $1 \le n_{i, j} < p$ for $1 \le i \le d$ and $1 \le j \le q$. Let $\varepsilon = 1$ if $\infty$ is a branch point of $\Ccal \to \mathbb{P}_{1, \Qbar}$ and $\varepsilon = 0$ otherwise.

            By construction, $g$ acts on the $x$-coordinate as $x \mapsto \zeta_{q} \cdot x$. For such an automorphism to exist, there must exist an integer $1 \le r \le p - 1$ whose multiplicative order in $( \Z / p \Z )^{*}$ divides $q$, and a polynomial $u \in \Qbar[x]$ such that $f( \zeta_{q} \cdot x ) = u(x)^{p} \cdot f(x)^{r}$ (by Kummer theory). Moreover, we must have $n_{i, j} \equiv r \cdot n_{i, j - 1} \pmod{p}$. When these conditions are satisfied, the automorphism $g$ is given by $( x, y ) \mapsto \bigl( \zeta_{q} \cdot x, u(x) \cdot y^{r} \bigr)$. On the other hand, by construction, $g^{q}$ is given by $( x, y ) \mapsto ( x, \zeta_{p} \cdot y )$. Since $g^{q}$ and $g$ commute, it follows that $r = 1$. In other words, for every fixed $i$, the integers $n_{i, j}$ are all congruent modulo $p$; hence, up to absorbing powers of $p$ into the $y$-variable, we can assume that for every fixed $i$ the integers $n_{i, j}$ are all equal. Then
            \begin{equation*}
                f(x) = x^{n_{0}} \cdot \prod_{i = 1}^{d} ( x^{q} - a_{i}^{q} )^{n_{i}}
            \end{equation*}
            for some integers $1 \le n_{i} < p$ where $1 \le i \le d$.
            
            From the Riemann--Hurwitz formula we deduce that the genus of $\Ccal$ is $( p - 1 ) ( d q + \varepsilon - 1 ) / 2$ if $n_{0} > 0$, and $( p - 1 ) ( d q + \varepsilon - 2 ) / 2$ if $n_{0} = 0$. On the other hand, since $\Aut ( \Ccal )$ contains an element of order $pq$, Propositions~\ref{pro: a curve with simple Jacobian is superelliptic; sec: Preliminaries} and~\ref{pro: genus superelliptic curve pq; sec: Preliminaries} imply that the genus of $\Ccal$ is either $( p - 1 ) ( q - 1 )$ or $( p - 1 ) ( q - 1 ) / 2$. Hence, there are four possibilities:
            \begin{enumerate}
                \item we have $\varepsilon = 1$, $d = 1$, $q = 2$ and $n_{0} > 0$, corresponding to $f(x) = x^{n_{0}} \cdot ( x^{2} - a_{1}^{2} )^{n_{1}}$ and $n_{0} + 2n_{1} \not\equiv 0 \pmod{p}$;
                \item we have $\varepsilon = 0$, $d = 1$ and $n_{0} > 0$, corresponding to $f(x) = x^{n_{0}} \cdot ( x^{q} - a_{1}^{q} )^{n_{1}}$ and $n_{0} + q \cdot n_{1} \equiv 0 \pmod{p}$;
                \item we have $\varepsilon = 0$, $d = 2$ and $n_{0} = 0$, corresponding to $f(x) = ( x^{q} - a_{1}^{q} )^{n_{1}} \cdot ( x^{q} - a_{2}^{q} )^{n_{2}}$ and $n_{1} + n_{2} \equiv 0 \pmod{p}$;
                \item we have $\varepsilon = 1$, $d = 1$ and $n_{0} = 0$, corresponding to $f(x) = ( x^{q} - a_{1}^{q} )^{n_{1}}$.
            \end{enumerate}
            We analyze each case separately.

            Suppose that we are in the first case. Up to isomorphism, we can assume $n_{0} = 1$. Since $q = 2$, the prime $p$ must be odd. Since $\varepsilon = 1$, there is a unique point $\infty$ of $\Ccal$ in the fiber of the natural cover $\Ccal \to \mathbb{P}_{1, \Qbar}$ above the point at infinity. Consider the automorphism of $\Ccal$
            \begin{equation*}
                \varphi \colon \Ccal \to \Ccal \qquad\qquad\qquad ( x, y ) \mapsto ( -x, -y ) .
            \end{equation*}
            The ramification points of the cover $\Ccal \to \Ccal / \langle \varphi \rangle$ are $( 0, 0 )$ and $\infty$. From the Riemann--Hurwitz formula, we deduce that the genus of $\Ccal / \langle \varphi \rangle$ is half the genus of $\Ccal$; hence, it is non-zero. However, we proved in Theorem~\ref{the: necessary condition simplicity; sec: Preliminaries} that the genus of $\Ccal / \langle \varphi \rangle$ is zero; hence, we reach a contradiction.

            Suppose now that we are in the second case. The point at infinity of $\mathbb{P}_{1, \Qbar}$ is not a branch point of the natural cover $\Ccal \to \mathbb{P}_{1, \Qbar}$, whereas $0$ is. By applying the automorphism $x \mapsto 1 / x$ of $\mathbb{P}_{1, \Qbar}$, which does not change the form of the action of $\Aut ( \Ccal ) / \langle g^{q} \rangle$ on $\Ccal / \langle g^{q} \rangle$, we can assume that the point at infinity of $\mathbb{P}_{1, \Qbar}$ is a branch point of the cover while $0$ is not. This allows us to reduce to the fourth case, which will be analyzed below.
            
            Suppose now that we are in the third case. Up to isomorphism, we can assume $n_{1} = 1$ and $n_{2} = p - 1$. We can rewrite the model of $\Ccal$ as $z^{p} = ( x^{q} - a_{1}^{q} ) / ( x^{q} - a_{2}^{q} )$, where $z = y / ( x^{q} - a_{2}^{q} )$. Consider, in this new model, the automorphism of $\Ccal$
            \begin{equation*}
                \varphi \colon \Ccal \to \Ccal \qquad\qquad\qquad (x, z) \mapsto \biggl( \frac{ a_{1}a_{2} }{x}, \frac{\gamma}{z} \biggr) ,
            \end{equation*}
            where $\gamma^{p} = a_{1}^{q} / a_{2}^{q}$. Since $\varphi$ has order $2$ and $\Aut ( \Ccal )$ is isomorphic to $C_{pq}$, it follows that either $p = 2$ or $q = 2$. If $p = 2$ (resp.~$q = 2$), there exists a different automorphism of $\Ccal$ of order $2$, namely $( x, y ) \mapsto ( x, -y )$ (resp.~$( x, y ) \mapsto ( -x, y )$). In Proposition~\ref{pro: unique element of order 2 in AutC if C is simple; sec: Preliminaries} we proved that there exists at most one element of order $2$ in $\Aut ( \Ccal )$; hence, we reach a contradiction.

            Finally, suppose that we are in the fourth case. Up to isomorphism, we can assume $n_{1} = 1$ and $a_{1} = 1$; hence, $\Ccal$ is isomorphic to the Catalan curve $y^{p} = x^{q} - 1$, as desired.
        \end{proof}

    \section{Generalized quaternion groups}
    \label{sec: Generalized quaternions groups}

        In this section, we analyze the generalized quaternion groups
        \begin{equation*}
            G_{2^{n}} = \langle x, y \,\lvert\, x^{2} \cdot y^{2^{n}}, y^{2^{n+1}}, x^{-1} \cdot y \cdot x \cdot y \rangle ,
        \end{equation*}
        where $n$ is a positive integer. We show that these groups can be realized as the automorphism groups of curves over $\Qbar$ of genus $\ge 2$ with simple Jacobians. More precisely:
            
        \begin{The}
        \label{the: main; sec: Generalized quaternions groups} 
            Let $n \ge 1$ and $L \ge 1$ be integers. For $\abf = ( a_{1}, \dots, a_{L} ) \in \Gbb_{m, \Qbar}^{L} ( \Qbar )$, set $f_{\abf} ( x ) \coloneqq x \cdot \bigl( x^{2^{n+1}} - 1 \bigr) \cdot \prod_{l = 1}^{L} \bigl[ \bigl( x^{2^{n}} - a_{l} \bigr) \cdot \bigl( x^{2^{n}} - a^{-1}_{l} \bigr) \bigr]$, and let $U \subset \Gbb_{m, \Qbar}^{L}$ be the non-empty open subscheme defined by $\operatorname{disc}_{x} ( f_{\abf} ) \neq 0$. For every $\abf \in U ( \Qbar )$, let $\Cfrak_{\abf} / \Qbar$ denote the smooth projective model of the affine curve 
            \begin{equation}
            \label{eq: Cfraka1...aL; sec: Generalized quaternions groups}
                y^{2} = f_{\abf} ( x ) = x \cdot \bigl( x^{2^{n+1}} - 1 \bigr) \cdot \prod_{l = 1}^{L} \bigl[ \bigl( x^{2^{n}} - a_{l} \bigr) \cdot \bigl( x^{2^{n}} - a^{-1}_{l} \bigr) \bigr] .
            \end{equation}
            For every number field $K$ and for $100 \%$ of the points $\abf \in U ( \Qbar ) \cap K^{L}$, one has $\Aut_{\Qbar} ( \Cfrak_{\abf} ) \cong G_{ 2^{n} }$, and the Jacobian $\Jfrak_{\abf} \coloneqq \Jac ( \Cfrak_{\abf} )$ is (geometrically) simple, with endomorphism algebra $\End_{\Qbar}^{0} (\Jfrak_{\abf} ) \cong \left( \frac{ ( \zeta_{2^{n+1}} - \zeta_{2^{n+1}}^{-1} )^{2}, \; -1 }{ \Q( \zeta_{2^{n+1}} + \zeta_{2^{n+1}}^{-1} ) } \right)$. Moreover, if $L\ge 2$, the curves $\Cfrak_{\abf}$ obtained in this way represent infinitely many distinct isomorphism classes over $\Qbar$, and their Jacobians represent infinitely many distinct isogeny classes over $\Qbar$.
        \end{The}
            
        \begin{Rem}
        \label{rem: case n = 1; sec: Generalized quaternions groups}
            The case $n = 1$ of Theorem~\ref{the: main; sec: Generalized quaternions groups} was established by Cantoral-Farfán, Lombardo and Voight in \cite{cantoral2023monodromy}. In the rest of this section, we therefore assume $n \ge 2$.
        \end{Rem}

        This section is organized as follows. In Section~\ref{subsec: Setup; sec: Generalized quaternions groups}, we fix the notation and provide some background for the curves \eqref{eq: Cfraka1...aL; sec: Generalized quaternions groups}. Section~\ref{subsec: Preliminaries; sec: Generalized quaternions groups} is devoted to analyzing several properties of these curves, while in Section~\ref{subsec: Specialization methods; sec: Generalized quaternions groups} we extend the analysis using specialization techniques and degenerations of curves. Since the techniques and proofs in Sections~\ref{subsec: Preliminaries; sec: Generalized quaternions groups} and~\ref{subsec: Specialization methods; sec: Generalized quaternions groups} are very similar to those in Sections~\ref{subsec: Preliminaries; sec: Cyclic groups of order p^n} and~\ref{subsec: specialization methods; sec: Cyclic groups of order p^n}, we omit many of the details. Finally, in Section~\ref{subsec: Proof of main result; sec: Generalized quaternions groups} we prove Theorem~\ref{the: main; sec: Generalized quaternions groups}.

        \subsection{Setup}
        \label{subsec: Setup; sec: Generalized quaternions groups}

            We begin with a brief setup. Throughout Section~\ref{sec: Generalized quaternions groups}, we fix an integer $n \ge 2$. Moreover, we denote by $L$ an integer $\ge 1$ and by $\alpha_{1}, \dots, \alpha_{L}$ the variables of $\Gbb_{m, \Qbar}^{L}$.

            \begin{Def}
            \label{def: Cfrak, Jfrak, Cbf, Jbf; subsec: Setup; sec: Generalized quaternions groups}
                Define the family $\Cfrak \to U$ of curves over $\Qbar$ given by the relative smooth projective model of
                \begin{equation}
                \label{eq: Cfrak; subsec: Setup; sec: Generalized quaternions groups}
                    y^{2} = x \cdot \bigl( x^{2^{n+1}} - 1 \bigr) \cdot \prod_{l = 1}^{L} \bigl[ \bigl( x^{2^{n}} - \alpha_{l} \bigr) \cdot \bigl( x^{2^{n}} - \alpha^{-1}_{l} \bigr) \bigr] ,
                \end{equation}
                where $U$ is the non-empty open subscheme of $\Gbb_{m, \Qbar}^{L}$ on which the discriminant (with respect to $x$) of the polynomial on the right-hand side does not vanish. Let $\Jfrak \to U$ be the Jacobian scheme of $\Cfrak \to U$. 

                Moreover, let $\Cbf$ (resp.~$\Jbf$) be the generic fiber over $\Qbar ( U )$ of $\Cfrak \to U$ (resp.~$\Jfrak \to U$), and let $\Cbf^{\mathrm{al}}$ (resp.~$\Jbf^{\mathrm{al}}$) be the scalar extension to an algebraic closure $\Qbar ( U )^{\mathrm{al}}$ of $\Qbar ( U )$. 
            \end{Def}

            \begin{Def}
            \label{def: varphi and Phi; subsec: Setup; sec: Generalized quaternions groups}
                Define the automorphisms of $\Cbf$ given by
                \begin{gather*}
                    \varphi \colon \Cbf \to \Cbf \qquad\qquad\qquad ( x, y ) \mapsto ( \zeta_{2^{n}} \cdot x, \zeta_{2^{n+1}} \cdot y ) , \notag \\ 
                    \psi \colon \Cbf \to \Cbf \qquad\qquad\qquad ( x, y ) \mapsto \biggl( \frac{1}{x}, \frac{ \zeta_{4} \cdot y }{ x^{ (L+1) \cdot 2^{n} + 1 } } \biggr) ,
                \end{gather*}
                which generate a subgroup of $\Aut ( \Cbf )$ isomorphic to $G_{2^{n}}$ via the map sending $\psi$ and $\varphi$ to the standard generators of $G_{2^n}$. Let $\Phi$ (resp.~$\Psi$) be the endomorphism of $\Jbf$ induced by $\varphi$ (resp.~$\psi$).
            \end{Def}

            \begin{Def}
            \label{def: D and R; subsec: Setup; sec: Generalized quaternions groups}
                Let $D$ be the algebra over $\Q( \zeta_{2^{n+1}} + \zeta_{2^{n+1}}^{-1} )$ generated by $\alpha, \beta$ satisfying the relations $\alpha^{2} = ( \zeta_{2^{n+1}} - \zeta_{2^{n+1}}^{-1} )^{2}$, $\beta^{2} = -1$ and $\alpha \cdot \beta = - \beta \cdot \alpha$, that is, the quaternion algebra $\left( \frac{ ( \zeta_{2^{n+1}} - \zeta_{2^{n+1}}^{-1} )^{2}, \; -1 }{ \Q( \zeta_{2^{n+1}} + \zeta_{2^{n+1}}^{-1} ) } \right)$. Moreover, let $\gamma \coloneqq \bigl( ( \zeta_{2^{n+1}} + \zeta_{2^{n+1}}^{-1} ) + \alpha \bigr) / 2$ and let $R$ be the order of $D$ generated by $1, \gamma, \beta, \gamma \cdot \beta$.
            \end{Def}
                
            \begin{Rem}
            \label{rem: D is a division algebra; subsec: Setup; sec: Generalized quaternions groups}
                The quaternion algebra $D$ is a division algebra. Indeed, if $D$ were isomorphic to the $2 \times 2$ matrix algebra over $\Q ( \zeta_{2^{n+1}} + \zeta_{2^{n+1}}^{-1} )$, then the completion of $D$ with respect to a real place of $\Q( \zeta_{2^{n+1}} + \zeta_{2^{n+1}}^{-1} )$ would be isomorphic to the $2 \times 2$ matrix algebra over $\R$. This contradicts the fact that both $\alpha^{2}$ and $\beta^{2}$ are negative real numbers.
            \end{Rem}

            The origin of the family $\Cfrak \to U$ is explained in \cite[Sections~$3.1$, $3.2$ and~$3.3$]{cantoral2023monodromy}.

            \begin{Rem}
            \label{rem: genus of Cbf; subsec: Setup; sec: Generalized quaternions groups}
                By Lemma~\ref{lem: genus superelliptic curves; sec: Preliminaries}, $\Cfrak \to U$ is a family of curves of genus $( L + 1 ) \cdot 2^{n}$ and $\Jfrak \to U$ is an abelian scheme of relative dimension $( L + 1 ) \cdot 2^{n}$.
            \end{Rem}

        \subsection{Preliminaries}
        \label{subsec: Preliminaries; sec: Generalized quaternions groups}
            
            We continue by analyzing the families $\Cfrak \to U$ and $\Jfrak \to U$. First, we recall a description of a basis of $H^{0} ( \Cbf^{\mathrm{al}}, \Omega^{1}_{ \Cbf^{\mathrm{al}} } )$. 

            \begin{Lem}
            \label{lem: basis of regular differential forms on Cbf; subsec: Preliminaries; sec: Generalized quaternions groups}
                The differential forms
                \begin{equation*}
                    \omega_{i} \coloneqq \frac{x^{i} \cdot dx}{y} , \qquad\qquad\qquad \text{for } 0 \le i \le ( L + 1 ) \cdot 2^{n} - 1 ,
                \end{equation*}
                form a basis of $H^{0} ( \Cbf^{\mathrm{al}}, \Omega^{1}_{ \Cbf^{\mathrm{al}} } )$.
            \end{Lem}

            Next, we recall a basic fact concerning $\Jbf[ 2 ]$.

            \begin{Lem}
            \label{lem: properties of 2-torsion part; subsec: Preliminaries; sec: Generalized quaternions groups}
                The points
                \begin{align}
                \label{eq: Dalphal; subsec: Preliminaries; sec: Generalized quaternions groups}
                    D_{ \alpha_{l}, i }^{+} &\coloneqq \bigl[ ( \zeta_{2^{n}}^{i} \cdot \alpha_{l}^{1 / 2^{n}}, 0 ) - \infty \bigr] & &\text{for $0 \le i \le 2^{n} - 1$ and $1 \le l \le L$,} \notag \\ 
                    D_{ \alpha_{l}, i }^{-} &\coloneqq \bigl[ ( \zeta_{2^{n}}^{i} \cdot \alpha_{l}^{-1 / 2^{n}}, 0 ) - \infty \bigr] & &\text{for $0 \le i \le 2^{n} - 1$ and $1 \le l \le L$,} \notag \\
                    P_{i} &\coloneqq \bigl[ ( \zeta_{2^{n+1}}^{i}, 0 ) - \infty \bigr] & &\text{for $0 \le i \le 2^{n+1} - 1$,}
                \end{align}
                where the square brackets denote the class of a divisor, form a basis of the $\Fbb_{2}$-vector space $\Jbf [ 2 ]$.
            \end{Lem}
            \begin{proof}
                It follows from a computation in \cite[\S~2]{WojciechSuperJac} that these points are $\Fbb_{2}$-linearly independent. Since their number is $L \cdot 2^{n} + L \cdot 2^{n} +2^{n+1}=\dim_{\Fbb_2} \Jbf[2]$ (see Remark~\ref{rem: genus of Cbf; subsec: Setup; sec: Generalized quaternions groups}), they form a basis.
            \end{proof}

            The actions of $\varphi$ and $\psi$ (Definition~\ref{def: varphi and Phi; subsec: Setup; sec: Generalized quaternions groups}) on $H^{0} ( \Cbf^{\mathrm{al}}, \Omega^{1}_{ \Cbf^{\mathrm{al}} } )$ via pullback provide information about $\End^{0} ( \Jbf^{\mathrm{al}} )$.

            \begin{Lem}
            \label{lem: End0Jbf contains D; subsec: Preliminaries; sec: Generalized quaternions groups}
                The $\Q$-subalgebra of $\End^{0} ( \Jbf^{\mathrm{al}} )$ generated by $\Phi$ and $\Psi$ is isomorphic to $D$. Moreover, this subalgebra injects into the endomorphism algebra of every isotypic component of $\Jbf^{\mathrm{al}}$. 
            \end{Lem}
            \begin{proof}
                The action of $\varphi$ on $H^{0} ( \Cbf^{\mathrm{al}}, \Omega^{1}_{ \Cbf^{\mathrm{al}} } )$ via pullback is $\omega_{i} \mapsto \zeta_{2^{n+1}}^{ 2 ( i + 1 ) - 1 } \cdot \omega_{i}$, hence $\Phi$ generates a $\Q$-subalgebra of $\End^{0} ( \Jbf^{\mathrm{al}} )$ isomorphic to $\Q ( \zeta_{2^{n+1}} )$ (cf.~proof of Lemma~\ref{lem: End0Jbf contains Q(zetapn); subsec: Preliminaries; sec: Cyclic groups of order p^n}). The action of $\psi$ on $H^{0} ( \Cbf^{\mathrm{al}}, \Omega^{1}_{ \Cbf^{\mathrm{al}} } )$ via pullback is $\omega_{i} \mapsto \zeta_{4} \cdot \omega_{ ( L + 1 ) \cdot 2^{n} - i - 1 }$, hence $\Psi$ generates a $\Q$-subalgebra of $\End^{0} ( \Jbf^{\mathrm{al}} )$ isomorphic to $\Q ( \zeta_{4} )$. It is straightforward to verify that $\Psi \circ ( \Phi - \Phi^{-1} ) = -( \Phi - \Phi^{-1} ) \circ \Psi$ and $\Psi \circ ( \Phi + \Phi^{-1} ) = ( \Phi + \Phi^{-1} ) \circ \Psi$; hence, $\Phi - \Phi^{-1}$ and $\Psi$ satisfy the same relations as the generators $\alpha, \beta$ of $D$ (see Definition~\ref{def: D and R; subsec: Setup; sec: Generalized quaternions groups}).
                
                Finally, since $\End^{0} ( \Jbf^{\mathrm{al}} )$ projects onto the endomorphism algebra of every isotypic component of $\Jbf^{\mathrm{al}}$, there exists a morphism from $D$ to the endomorphism algebra of every isotypic component of $\Jbf^{\mathrm{al}}$. Since $D$ is a division algebra (Remark~\ref{rem: D is a division algebra; subsec: Setup; sec: Generalized quaternions groups}), this morphism is either trivial or injective. As $\Phi$ acts nontrivially on every isotypic component of $\Jbf^{\mathrm{al}}$, it is injective.
            \end{proof}

            We are interested in $\End^{0} ( \Jbf^{\mathrm{al}} )$ because of the following theorem:
                
            \begin{The}
            \label{the: reduction first statement main the; subsec: Preliminaries; sec: Generalized quaternions groups}
                If $\End^{0} ( \Jbf^{\mathrm{al}} )$ is isomorphic to $\left( \frac{ ( \zeta_{2^{n+1}} - \zeta_{2^{n+1}}^{-1} )^{2}, \; -1 }{ \Q( \zeta_{2^{n+1}} + \zeta_{2^{n+1}}^{-1} ) } \right)$, then for every number field $K$ and for $100\%$ of the points $( a_{1},  \ldots, a_{L} ) \in U ( \Qbar ) \cap K^{L}$, the smooth projective model over $\Qbar$ of the affine curve
                \begin{equation*}
                    y^{2} = x \cdot \bigl( x^{2^{n+1}} - 1 \bigr) \cdot \prod_{l = 1}^{L} \bigl[ \bigl( x^{2^{n}} - a_{l} \bigr) \cdot \bigl( x^{2^{n}} - a^{-1}_{l} \bigr) \bigr]
                \end{equation*}
                has automorphism group isomorphic to $G_{2^{n}}$, and its Jacobian is (geometrically) simple with endomorphism algebra isomorphic to $\left( \frac{ ( \zeta_{2^{n+1}} - \zeta_{2^{n+1}}^{-1} )^{2}, \; -1 }{ \Q( \zeta_{2^{n+1}} + \zeta_{2^{n+1}}^{-1} ) } \right)$.
            \end{The}
            \begin{proof}
                Note that $\Cfrak \to U$ is the base change to $\Qbar$ of a family defined over $\Q$. Lemma~\ref{lem: how many fibers of Jfrak have bigger End0; subsec: specialization methods; sec: Preliminaries} implies that for every number field $K$ and for $100\%$ of the points $( a_{1},  \ldots, a_{L} )$ in $U ( \Qbar ) \cap K^{L}$, we have that $\End^{0} ( \Jfrak_{ ( a_{1}, \dots, a_{L} ) } )$ is isomorphic to $\End^{0} ( \Jbf^{\mathrm{al}} ) \cong D$.
                    
                Since $D$ is a division algebra (Remark~\ref{rem: D is a division algebra; subsec: Setup; sec: Generalized quaternions groups}), it follows that $\Jfrak_{ ( a_{1}, \dots, a_{L} ) }$ is simple. Since $\Aut ( \Cfrak_{ ( a_{1}, \dots, a_{L} ) } )$ contains the non-abelian group $G_{2^{n}}$, Theorem~\ref{the: possible Aut nice curve with simple Jacobian; sec: Possible automorphism groups} implies that $\Aut ( \Cfrak_{ ( a_{1}, \dots, a_{L} ) } )$ is isomorphic to $G_{2^{m}}$ for some integer $m \ge n$.

                From \cite[Theorem~$12.1$, \S~III]{milneAV} it follows that $\Aut ( \Cfrak_{ ( a_{1}, \dots, a_{L} ) } )$ injects into the group of units of finite order of $\End^{0} ( \Jfrak_{ ( a_{1}, \dots, a_{L} ) } )$. If we had $m > n$, then $\Aut ( \Cfrak_{ ( a_{1}, \dots, a_{L} ) } )$ would contain an element of order $2^{n+2}$ that generates a subfield of $\End^{0} ( \Jfrak_{ ( a_{1}, \dots, a_{L} ) } )$ of degree $2^{n+1}$ over $\Q$. This contradicts the fact that $[D \colon \Q] = 2^{n+1}$ and $D$ is not a field; hence, $\Aut ( \Cfrak_{ ( a_{1}, \dots, a_{L} ) } )$ is isomorphic to $G_{2^{n}}$.  

                Finally, we show that $\End ( \Jfrak_{ ( a_{1}, \dots, a_{L} ) } )$ is isomorphic to $R$. Let $R_{1}$ be the order of $D$ generated by $1, \alpha, \beta, \alpha \cdot \beta$ (see Definition~\ref{def: D and R; subsec: Setup; sec: Generalized quaternions groups}), and note that $R_{1}$ is a suborder of $R$ of index $2$. By definition, we know that $\End ( \Jfrak_{ ( a_{1}, \dots, a_{L} ) } )$ is an order of $\End^{0} ( \Jfrak_{ ( a_{1}, \dots, a_{L} ) } )$. Moreover, $R$ is contained in $\End ( \Jfrak_{ ( a_{1}, \dots, a_{L} ) } )$ via the injective map defined by $\gamma \mapsto \Phi$, $\beta \mapsto \Psi$ and $( \zeta_{2^{n+1}} + \zeta_{2^{n+1}}^{-1} ) \mapsto ( \Phi + \Phi^{-1} )$ (see Definition~\ref{def: D and R; subsec: Setup; sec: Generalized quaternions groups}). We claim that the index of $R$ in $\End ( \Jfrak_{ ( a_{1}, \dots, a_{L} ) } )$ is a power of $2$.

                By \cite[Corollary~13.4.1]{VoightQuaternionAlgebras}, the discriminant of the quaternion algebra $D$ is a product of prime ideals of $\Q( \zeta_{2^{n+1}} + \zeta_{2^{n+1}}^{-1} )$ above $2$; hence, by \cite[Theorem~15.5.5]{VoightQuaternionAlgebras}, the discriminant of a maximal order of $D$ is a product of prime ideals of $\Q( \zeta_{2^{n+1}} + \zeta_{2^{n+1}}^{-1} )$ above $2$. By \cite[Example~15.2.10]{VoightQuaternionAlgebras}, the discriminant of the order $R_{1}$ is a product of prime ideals of $\Q( \zeta_{2^{n+1}} + \zeta_{2^{n+1}}^{-1} )$ above $2$. Hence, by \cite[Lemma~15.2.15]{VoightQuaternionAlgebras}, we deduce that the index of $R_{1}$ inside a maximal order of $D$ containing $R_{1}$ is a power of $2$. Since the index is multiplicative in a tower of extensions of orders, the index of $R$ in $\End ( \Jfrak_{ ( a_{1}, \dots, a_{L} ) } )$ is a power of $2$.

                Assume, for contradiction, that $\End ( \Jfrak_{ ( a_{1}, \dots, a_{L} ) } )$ strictly contains $R$. Then, by what we have just shown, the index of $R$ inside $\End ( \Jfrak_{ ( a_{1}, \dots, a_{L} ) } )$ is a nontrivial power of $2$; hence, there exists an element of $R$
                \begin{equation*}
                    a = \sum_{i = 0}^{2^{n} - 1} c_{i} \cdot \Phi^{i} + \sum_{i = 0}^{2^{n} - 1} d_{i} \cdot ( \Phi^{i} \circ \Psi ) , \qquad\qquad\qquad c_{i}, d_{i} \in \Z ,
                \end{equation*}
                with $a/2 \notin R$ but $a/2 \in \End ( \Jfrak_{ ( a_{1}, \dots, a_{L} ) } )$. In particular, $a$ annihilates the $2$-torsion of $\Jfrak_{ ( a_{1}, \dots, a_{L} ) }$.
                    
                Denote by $\overline{ D_{\alpha_{l}, i}^{\pm} }$ and $\overline{ P_{i} }$ the points defined in \eqref{eq: Dalphal; subsec: Preliminaries; sec: Generalized quaternions groups}, where $\alpha_{l}$ is replaced with $a_{l}$. The points $\overline{ D_{\alpha_{l}, i}^{\pm} }, \overline{ P_{i} }$ form a basis for the $2$-torsion of $\Jfrak_{ ( a_{1}, \dots, a_{L} ) }$. From the explicit expressions of $\varphi$ and $\psi$ (see Definition~\ref{def: varphi and Phi; subsec: Setup; sec: Generalized quaternions groups}), it is straightforward to determine the action of $\Phi$ and $\Psi$ on the chosen basis for the $2$-torsion of $\Jfrak_{ ( a_{1}, \dots, a_{L} ) }$ (see Lemma~\ref{lem: properties of 2-torsion part; subsec: Preliminaries; sec: Generalized quaternions groups}). More precisely:
                \begin{align*}
                    \Phi ( \overline{ D_{\alpha_{l}, i}^{\pm} } ) &= \overline{ D_{\alpha_{l}, ( i + 1 \bmod{ 2^{n} } ) }^{\pm} } , & \Phi ( \overline{ P_{i} } ) &= \overline{ P_{ ( i + 2 \bmod{ 2^{n+1} } ) } } , \\
                    \Psi ( \overline{ D_{\alpha_{l}, i}^{\pm} } ) &= \overline{ D_{\alpha_{l}, ( - i \bmod{ 2^{n} } ) }^{\mp} } - \overline{ Q_{0} } , & \Psi ( \overline{ P_{i} } ) &= \overline{ P_{ ( - i \bmod{ 2^{n+1} } ) } } - \overline{ Q_{0} } ,
                \end{align*}
                where $\overline{ Q_{0} }$ is the class of the divisor $(0, 0) - \infty$, which is the sum of all the elements of the chosen basis. It follows that
                \begin{align*}
                    a ( \overline{ D_{\alpha_{1}, 1}^{+} } ) &= \sum_{i = 0}^{2^{n} - 1} c_{i} \cdot \overline{ D_{\alpha_{1}, ( i + 1 \bmod{2^{n}} ) }^{+} } + \sum_{i = 0}^{2^{n} - 1} d_{i} \cdot ( \overline{ D_{\alpha_{1}, ( i - 1 \bmod{2^{n}} ) }^{-} } - \overline{ Q_{0} } ) \\
                    &= \sum_{i = 0}^{2^{n} - 1} ( c_{i} - N ) \cdot \overline{ D_{\alpha_{1}, ( i + 1 \bmod{2^{n}} ) }^{+} } + \sum_{i = 0}^{2^{n} - 1} ( d_{i} - N ) \cdot \overline{ D_{\alpha_{1}, ( i - 1 \bmod{2^{n}} ) }^{-} } - \sum_{i = 0}^{2^{n+1}-1} N \cdot \overline{P_{i}} ,
                \end{align*}
                where $N = d_{0} + d_{1} + \cdots + d_{2^{n} - 1}$. Since $a$ annihilates the $2$-torsion of $\Jfrak_{ ( a_{1}, \dots, a_{L} ) }$, we have $a ( \overline{ D_{\alpha_{1}, 1}^{+} } ) = 0$; hence, all the $c_{i}$ and $d_{i}$ are even and $a / 2$ belongs to $R$, a contradiction.
            \end{proof}

            We conclude this preliminary section by showing that the map to moduli space induced by the family $\Cfrak \to U$ is nonconstant.
                
            \begin{The}
            \label{the: there are infinitely many classes; subsec: Preliminaries; sec: Generalized quaternions groups}
                The image of the map to the moduli space induced by the family $\Cfrak \to U$ has dimension $L$. Moreover, the second statement of Theorem~\ref{the: main; sec: Generalized quaternions groups} holds.
            \end{The}
            \begin{proof}
                Conceptually, the proof is identical to that of Theorem~\ref{the: there are infinitely many classes; subsec: Preliminaries; sec: Cyclic groups of order p^n}.
            \end{proof}

        \subsection{Computation of specific specializations}
        \label{subsec: Specialization methods; sec: Generalized quaternions groups}

            An important ingredient in the proof of Theorem~\ref{the: main; sec: Generalized quaternions groups} is to find specific specializations of $\Jfrak \to U$. We refer the reader to Section~\ref{subsec: specialization methods; sec: Preliminaries} for the relevant theoretical constructions.

            \begin{Pro}
            \label{pro: particular specialization; subsec: Specialization methods; sec: Generalized quaternions groups}
                Let $( a_{1}, \dots, a_{L} ) \in U ( \Qbar )$ and let $S$ be a subset of $\{ 1, \dots, L \}$ of cardinality $1 \le M \le L$. Consider the curves
                \begin{equation*}
                    \Ccal_{1} \colon y^{2} = x \cdot \bigl( x^{2^{n+1}} - 1 \bigr) \cdot \prod_{l \notin S} \bigl[ \bigl( x^{2^{n}} - a_{l} \bigr) \cdot \bigl( x^{2^{n}} - a^{-1}_{l} \bigr) \bigr] , \qquad\qquad \Ccal_{2} \colon y^{2} = x \cdot \prod_{l \in S} \bigl( x^{2^{n}} - a_{l} \bigr) .
                \end{equation*} 
                The following hold:
                \begin{enumerate}
                    \item There exist an admissible $T'$ and a point $t \in \Abb^{1}_{\Qbar} ( \Qbar ) \supset T' ( \Qbar )$ such that $\Jfrak_{T', t}$ is isomorphic to $\Jac ( \Ccal_{1} ) \times \Jac ( \Ccal_{2} )^{2}$;

                    \item Let $\eta \colon \Jbf [ 2 ] \xrightarrow{\sim} \Jfrak_{T', t} [ 2 ]$ be the bijection given by the last statement of Lemma~\ref{lem: End0 embeds and torsion bijects under specialization; subsec: specialization methods; sec: Preliminaries}, and let $D_{ \alpha_{l}, i }^{+}, D_{ \alpha_{l}, i }^{-}, P_{i} \in \Jbf [ 2 ]$  be the points defined in \eqref{eq: Dalphal; subsec: Preliminaries; sec: Generalized quaternions groups}. Under the identification $\Jfrak_{T', t} \cong \Jac ( \Ccal_{1} ) \times \Jac ( \Ccal_{2} ) \times \Jac ( \Ccal_{2} )$, we have
                    \begin{align*}
                        \eta ( P_{i} ) &= \Bigl( \bigl[ ( \zeta_{2^{n+1}}^{i}, 0 ) - \infty_{1} \bigr], 0, 0 \Bigr) & &\text{for $1 \le i \le 2^{n+1}$,} \\
                        \eta ( D_{ \alpha_{l}, i }^{\pm} ) &= \Bigl( \bigl[ ( \zeta_{2^{n}}^{i} \cdot a_{l}^{ \pm 1 / 2^{n} }, 0 ) - \infty_{1} \bigr], 0, 0 \Bigr) & &\text{for $1 \le i \le 2^{n}$ and $l \notin S$,} \\
                        \eta ( D_{ \alpha_{l}, i }^{+} ) &= \Bigl( 0, \bigl[ ( \zeta_{2^{n}}^{i} \cdot a_{l}^{ 1 / 2^{n} }, 0 ) - \infty_{2} \bigr], 0 \Bigr) & &\text{for $1 \le i \le 2^{n}$ and $l \in S$,} \\
                        \eta ( D_{ \alpha_{l}, i }^{-} ) &= \Bigl( 0, 0, \bigl[ ( \zeta_{2^{n}}^{i} \cdot a_{l}^{ -1 / 2^{n} }, 0 ) - ( 0, 0 ) \bigr] \Bigr) & &\text{for $1 \le i \le 2^{n}$ and $l \in S$,}
                    \end{align*}
                    where $\infty_{1}$, $\infty_{2}$ are the unique points at infinity of respectively $\Ccal_{1}$ and $\Ccal_{2}$. 
                \end{enumerate}
            \end{Pro}
            \begin{proof}
                Conceptually, the proof of the first part is similar to that of \cite[Lemma~6.3.2]{cantoral2023monodromy}, while the proof of the second part is analogous to that of Proposition~\ref{pro: particular specialization; subsec: specialization methods; sec: Cyclic groups of order p^n}.
            \end{proof}

            Proposition~\ref{pro: particular specialization; subsec: Specialization methods; sec: Generalized quaternions groups} has an interesting consequence.

            \begin{Cor}
            \label{cor: exclusion cases with torsion arguments; subsec: Specialization methods; sec: Generalized quaternions groups}
                Let $L \ge 2$, let $( a_{1}, \dots, a_{L} ) \in U ( \Qbar )$ and let $S_{1}, S_{2}$ be distinct subsets of $\{ 1, \dots, L \}$ of equal cardinality $1 \le M < L$. Consider the curves
                \begin{align*}
                    \Ccal_{1} &\colon y^{2} = x \cdot \bigl( x^{2^{n+1}} - 1 \bigr) \cdot \prod_{l \notin S_{1}} \bigl[ \bigl( x^{2^{n}} - a_{l} \bigr) \cdot \bigl( x^{2^{n}} - a^{-1}_{l} \bigr) \bigr] , \qquad\qquad \Ccal_{2} \colon y^{2} = x \cdot \prod_{l \in S_{1}} \bigl( x^{2^{n}} - a_{l} \bigr) , \\
                    \Ccal_{3} &\colon y^{2} = x \cdot \bigl( x^{2^{n+1}} - 1 \bigr) \cdot \prod_{l \notin S_{2}} \bigl[ \bigl( x^{2^{n}} - a_{l} \bigr) \cdot \bigl( x^{2^{n}} - a^{-1}_{l} \bigr) \bigr] , \qquad\qquad \Ccal_{4} \colon y^{2} = x \cdot \prod_{l \in S_{2}} \bigl( x^{2^{n}} - a_{l} \bigr) .
                \end{align*} 
                Let $\Jcal_{1}$ (resp.~$\Jcal_{2}$) be the abelian variety obtained by applying Proposition~\ref{pro: particular specialization; subsec: Specialization methods; sec: Generalized quaternions groups} with the point $( a_{1}, \dots, a_{L} )$ and $S = S_{1}$ (resp.~$S = S_{2}$). The following conditions cannot hold simultaneously:
                \begin{itemize}
                    \item the Jacobian $\Jbf^{\mathrm{al}}$ is isogenous to a product of two abelian subvarieties $Z_{1}, Z_{2}$ of $\Jbf^{\mathrm{al}}$;
                    \item the abelian variety $Z_{1}$ (resp.~$Z_{2}$) specializes to $\Jac ( \Ccal_{1} )$ (resp.~$\Jac ( \Ccal_{2} )^{2}$) in $\Jcal_{1}$;
                    \item the abelian variety $Z_{1}$ (resp.~$Z_{2}$) specializes to $\Jac ( \Ccal_{3} )$ (resp.~$\Jac ( \Ccal_{4} )^{2}$) in $\Jcal_{2}$.
                \end{itemize}
            \end{Cor}
            \begin{proof}
                The proof is identical to that of Corollary~\ref{cor: exclusion cases with torsion arguments; subsec: specialization methods; sec: Cyclic groups of order p^n}.
            \end{proof}

            \begin{Rem}
            \label{rem: lemma remains valid; subsec: Specialization methods; sec: Generalized quaternions groups}
                Lemma~\ref{lem: useful to apply cor: exclusion cases with torsion arguments; subsec: specialization methods; sec: Cyclic groups of order p^n} remains valid in this context. 
            \end{Rem}

        \subsection{Proof of main result}
        \label{subsec: Proof of main result; sec: Generalized quaternions groups}

            We now prove Theorem~\ref{the: main; sec: Generalized quaternions groups}. In Theorems~\ref{the: reduction first statement main the; subsec: Preliminaries; sec: Generalized quaternions groups} and~\ref{the: there are infinitely many classes; subsec: Preliminaries; sec: Generalized quaternions groups} we proved that it is sufficient to show that $\End^{0} ( \Jbf^{\mathrm{al}} ) \cong D$. We begin by considering the case $L = 1$.

            \begin{Lem}
            \label{lem: L = 1 is true; subsec: Proof of main result; sec: Generalized quaternions groups}
                Theorem~\ref{the: main; sec: Generalized quaternions groups} holds for $L = 1$.
            \end{Lem}
            \begin{proof}
                Let $\Jbf^{\mathrm{al}} \sim X_{1} \times X_{2} \times \cdots \times X_{A}$ be an isogeny decomposition of $\Jbf^{\mathrm{al}}$, where some $X_{j}$ may be isogenous to each other. We regard the $X_{j}$ as abelian subvarieties of $\Jbf^{\mathrm{al}}$. 

                Consider the specialization $\Jfrak_{i}$ of $\Jfrak \to U$, namely, the Jacobian of the curve $y^{2} = x \cdot ( x^{2^{n+1}} - 1 ) \cdot ( x^{2^{n}} - i ) \cdot ( x^{2^{n}} + i ) = x \cdot ( x^{2^{n+2}} - 1 )$. In Lemma~\ref{lem: L = 1; subsec: The case p=2; sec: Cyclic groups of order p^n} we proved that $\Jfrak_{i}$ is isogenous to the square of a simple abelian variety $X$ with endomorphism algebra isomorphic to $K_{n+3} = \Q ( \zeta_{2^{n+3}} - \zeta_{2^{n+3}}^{-1} )$. In Lemma~\ref{lem: isogeny decomposition specializes; subsec: specialization methods; sec: Preliminaries} we proved that the number of factors in an isogeny decomposition of $\Jbf^{\mathrm{al}}$ is at most that of an isogeny decomposition of $\Jfrak_{i}$; hence, $A \le 2$. 

                Let us prove that $\End^{0} ( \Jbf^{\mathrm{al}} )$ is a simple algebra (equivalently, $\Jbf^{\mathrm{al}}$ is isotypic). Suppose, for the sake of contradiction, that $\Jbf^{\mathrm{al}}$ is not isotypic. Since $A \le 2$, it follows that $A = 2$ and $X_{1}, X_{2}$ are two distinct isotypic components of $\Jbf^{\mathrm{al}}$. As all nontrivial abelian subvarieties of $\Jfrak_{i}$ are isogenous to $X$, the subvariety $X_{1}$ specializes to an abelian variety isogenous to $X$ under the specialization $\Jfrak_{i}$. Consequently, $\End^{0} ( X_{1} )$ injects into $\End^{0} ( X ) \cong K_{n+3}$, which implies that $\End^{0}( X _{1} )$ is a field. By Lemma~\ref{lem: End0Jbf contains D; subsec: Preliminaries; sec: Generalized quaternions groups}, $D$ injects into $\End^{0} ( X_{1} )$. Since $D$ is not commutative, this yields a contradiction.

                In Lemma~\ref{lem: L = 1; subsec: The case p=2; sec: Cyclic groups of order p^n} we proved that the curve $\Ccal_{1} \colon y^{2} = x \cdot ( x^{ 2^{n + 1} } - 1 )$ (resp.~$\Ccal_{2} \colon y^{2} = x \cdot ( x^{2^{n}} - 2 )$) has Jacobian isogenous to the square of a simple abelian variety with endomorphism algebra isomorphic to $K_{n+2}$ (resp.~$K_{n+1}$).
                    
                By applying Proposition~\ref{pro: particular specialization; subsec: Specialization methods; sec: Generalized quaternions groups} with $2 \in U ( \Qbar )$ and $S = \{ 1 \}$, we deduce the existence of an abelian variety $\Jcal$, of the form $\Jfrak_{T', t}$ for some admissible $T'$ and $t \in \Abb_{\Qbar}^{1} ( \Qbar ) \supset T' ( \Qbar )$, isomorphic to $\Jac ( \Ccal_{1} ) \times \Jac ( \Ccal_{2} )^{2}$. Note that $\Jac ( \Ccal_{1} )$ and $\Jac ( \Ccal_{2} )^{2}$ are two distinct isotypic components of $\Jcal$. In Lemma~\ref{lem: End0 embeds and torsion bijects under specialization; subsec: specialization methods; sec: Preliminaries} applied to $\Jcal$ we proved that there exists an embedding
                \begin{equation*}
                    \End^{0} ( \Jbf^{\mathrm{al}} ) \xhookrightarrow{} \M ( 2, K_{n+2} ) \times \M ( 4, K_{n+1} ) .
                \end{equation*}

                Since $\Jbf^{\mathrm{al}}$ is isotypic, Remark~\ref{rem: simple observation of lem: End0 embeds and torsion bijects under specialization; subsec: specialization methods; sec: Preliminaries} shows that $\End^{0} ( \Jbf^{\mathrm{al}} )$ embeds into $\M ( 2, K_{n+2} )$. Hence, the $\Q$-dimension of $\End^{0} ( \Jbf^{\mathrm{al}} )$ is at most $2^{n+2}$. Since $\End^{0} ( \Jbf^{\mathrm{al}} )$ is a $D$-module and $D$ is a division algebra of dimension $2^{n+1}$ over $\Q$ (Remark~\ref{rem: D is a division algebra; subsec: Setup; sec: Generalized quaternions groups}), it follows that the $\Q$-dimension of $\End^{0} ( \Jbf^{\mathrm{al}} )$ is divisible by $2^{n+1}$. Recalling that $\End^{0} ( \Jbf^{\mathrm{al}} )$ embeds into $\M ( 2, K_{n+2} )$, whose $\Q$-dimension is $2^{n+2}$, it follows that $\End^{0} ( \Jbf^{\mathrm{al}} )$ is isomorphic to either $D$ or $\M ( 2, K_{n+2} )$.

                The second possibility, however, cannot occur. Indeed, if $\End^{0} ( \Jbf^{\mathrm{al}} )$ were isomorphic to $\M ( 2, K_{n+2} )$, then we would have $A = 2$ and $\End^{0} ( X_{1} ) \cong K_{n+2}$. As before, $X_{1}$ specializes to an abelian variety isogenous to $X$ under the specialization $\Jfrak_{i}$. Consequently, $\End^{0} ( X_{1} ) \cong K_{n+2}$ injects into $\End^{0} ( X ) \cong K_{n+3}$, which contradicts Remark~\ref{rem: Kn not in Kn+1; subsec: The case p=2; sec: Cyclic groups of order p^n}.

                Hence, $\End^{0} ( \Jbf^{\mathrm{al}} )$ is isomorphic to $D$. Theorems~\ref{the: reduction first statement main the; subsec: Preliminaries; sec: Generalized quaternions groups} and~\ref{the: there are infinitely many classes; subsec: Preliminaries; sec: Generalized quaternions groups} imply that Theorem~\ref{the: main; sec: Generalized quaternions groups} holds for $L = 1$.
            \end{proof}

            Finally, we consider the general case.

            \begin{proof}[Proof of Theorem~\ref{the: main; sec: Generalized quaternions groups}]
                We proceed by induction on $L \ge 1$. Lemma~\ref{lem: L = 1 is true; subsec: Proof of main result; sec: Generalized quaternions groups} proves the claim for $L = 1$. So, let us assume that $L \ge 1$ and that the claim holds for every $1 \le L' \le L$, and let us prove it for $L + 1$. 
                
                Let $\Jbf^{\mathrm{al}} \sim X_{1} \times X_{2} \times \cdots \times X_{A}$ be an isogeny decomposition of $\Jbf^{\mathrm{al}}$, where some $X_{j}$ may be isogenous to each other. We regard the $X_{j}$ as abelian subvarieties of $\Jbf^{\mathrm{al}}$.

                In Lemma~\ref{lem: useful to apply cor: exclusion cases with torsion arguments; subsec: specialization methods; sec: Cyclic groups of order p^n} and Remark~\ref{rem: lemma remains valid; subsec: Specialization methods; sec: Generalized quaternions groups} we proved that there exists $( a_{1}, \dots, a_{L+1} ) \in U ( \Qbar )$ such that the endomorphism algebras of the Jacobians of both curves $\Ccal_{1} \colon y^{2} = x \cdot ( x^{2^{n+1}} - 1 ) \cdot \bigl[ ( x^{2^{n}} - a_{1} ) \cdot ( x^{2^{n}} - a_{1}^{-1} ) \bigr] \cdots \bigl[ ( x^{2^{n}} - a_{L} ) \cdot ( x^{2^{n}} - a_{L}^{-1} ) \bigr]$ and $\Ccal_{3} \colon y^{2} = x \cdot ( x^{2^{n+1}} - 1 ) \cdot \bigl[ ( x^{2^{n}} - a_{2} ) \cdot ( x^{2^{n}} - a_{2}^{-1} ) \bigr] \cdots \bigl[ ( x^{2^{n}} - a_{L+1} ) \cdot ( x^{2^{n}} - a_{L+1}^{-1} ) \bigr]$ are isomorphic to that of the Jacobian of the generic fiber of $y^{2} = x \cdot ( x^{2^{n+1}} - 1 ) \cdot \bigl[ ( x^{2^{n}} - \alpha_{1} ) \cdot ( x^{2^{n}} - \alpha_{1}^{-1} ) \bigr] \cdots \bigl[ ( x^{2^{n}} - \alpha_{L} ) \cdot ( x^{2^{n}} - \alpha_{L}^{-1} ) \bigr]$. The inductive hypothesis implies that the Jacobians of $\Ccal_{1}$ and $\Ccal_{3}$ are simple, of dimension $( L + 1 ) \cdot 2^{n}$ (Remark~\ref{rem: genus of Cbf; subsec: Setup; sec: Generalized quaternions groups}), with endomorphism algebras isomorphic to $D$.

                In Lemma~\ref{lem: L = 1; subsec: The case p=2; sec: Cyclic groups of order p^n} we proved that both curves $\Ccal_{2} \colon y^{2} = x \cdot ( x^{2^{n}} - a_{L+1} )$ and $\Ccal_{4} \colon y^{2} = x \cdot ( x^{2^{n}} - a_{1} )$ have Jacobians isogenous to the square of the same simple abelian variety $X$, of dimension $2^{n-2}$ (Remark~\ref{rem: genus of Cbf; subsec: Setup; sec: Cyclic groups of order p^n}), with endomorphism algebra isomorphic to $K_{n+1} = \Q \bigl( \zeta_{2^{n+1}} - \zeta_{2^{n+1}}^{-1} \bigr)$.

                By applying Proposition~\ref{pro: particular specialization; subsec: Specialization methods; sec: Generalized quaternions groups} with $( a_{1}, \dots, a_{L+1} ) \in U ( \Qbar )$ and $S = \{ L+1 \}$, we deduce the existence of an abelian variety $\Jcal_{1}$, of the form $\Jfrak_{T', t}$ for some admissible $T'$ and $t \in \Abb_{\Qbar}^{1} ( \Qbar ) \supset T' ( \Qbar )$, isomorphic to $\Jac ( \Ccal_{1} ) \times \Jac ( \Ccal_{2} )^{2}$. 
                    
                Let us prove that $\Jbf^{\mathrm{al}}$ is simple. Let us assume, for the sake of contradiction, that $X_{1}$ is a proper abelian subvariety of $\Jbf^{\mathrm{al}}$. In Lemma~\ref{lem: isogeny decomposition specializes; subsec: specialization methods; sec: Preliminaries} we proved that there exists a simple factor of $\Jbf^{\mathrm{al}}$ of dimension at least the maximal dimension of a simple factor of $\Jcal_{1}$; hence, up to reordering the $X_{j}$, we can assume that $\dim ( X_{1} ) \ge \dim \bigl( \Jac ( \Ccal_{1} ) \bigr)$.
                    
                We show that $X_{1}$ specializes to $\Jac ( \Ccal_{1} )$ and $X_{2} \times \dots \times X_{A}$ specializes to $\Jac ( \Ccal_{2} )^{2}$ in $\Jcal_{1}$. In Lemma~\ref{lem: End0 embeds and torsion bijects under specialization; subsec: specialization methods; sec: Preliminaries} we proved that an abelian subvariety of $\Jbf^{\mathrm{al}}$ specializes to an abelian subvariety of $\Jcal_{1}$ of the same dimension. Since $\Jcal_{1} \cong \Jac ( \Ccal_{1} ) \times \Jac ( \Ccal_{2} )^{2}$, $\Jac ( \Ccal_{2} ) \sim X^{2}$ and both $\Jac ( \Ccal_{1} )$ and $X$ are simple, the proper abelian subvarieties of $\Jcal_{1}$ are isogenous either to $X^{m}$, for some integer $1 \le m \le 4$, or $\Jac ( \Ccal_{1} ) \times X^{m}$, for some integer $0 \le m \le 3$. Since $X$ has dimension $2^{n-2}$ and $\Jac ( \Ccal_{1} )$ has dimension $( L + 1 ) \cdot 2^{n}$, the proper abelian subvarieties of $\Jcal_{1}$ of dimension at least $\dim \bigl( \Jac ( \Ccal_{1} ) \bigr)$ are isogenous to $\Jac ( \Ccal_{1} ) \times X^{m}$ for some integer $0 \le m \le 3$. It follows that $X_{1}$ specializes to an abelian variety isogenous to $\Jac ( \Ccal_{1} ) \times X^{m}$ in $\Jcal_{1}$ for some integer $0 \le m \le 3$. Note that $X_{1}$ is an isotypic component of $\Jbf^{\mathrm{al}}$, since $2 \dim ( X_{1} ) > \dim ( \Jbf^{\mathrm{al}} )$.

                If $m = 3$, then $A = 2$ and $X_{2}$ specializes to an abelian variety isogenous to $X$ in $\Jcal_{1}$, since in Lemma~\ref{lem: isogeny decomposition specializes; subsec: specialization methods; sec: Preliminaries} we proved that a decomposition (up to isogeny) of $\Jbf^{\mathrm{al}}$ specializes to a decomposition (up to isogeny) of $\Jcal_{1}$. Hence, $\End^{0} ( X_{2} )$ embeds into $\End^{0} ( X ) \cong K_{n+1}$ (Lemma~\ref{lem: End0 embeds and torsion bijects under specialization; subsec: specialization methods; sec: Preliminaries}). In Lemma~\ref{lem: End0Jbf contains D; subsec: Preliminaries; sec: Generalized quaternions groups} we proved that $D$ embeds into $\End^{0} ( X_{2} )$, since $X_{2}$ is an isotypic component of $\Jbf^{\mathrm{al}}$, which yields a contradiction since $D$ is not commutative.

                If $m = 2$, then $\End^{0} ( X_{1} )$ embeds into $\End^{0} \bigl( \Jac ( \Ccal_{1} ) \times X^{2} \bigr) \cong D \times \M ( 2, K_{n+1} )$ (Lemma~\ref{lem: End0 embeds and torsion bijects under specialization; subsec: specialization methods; sec: Preliminaries}). Since $X_{1}$ is an isotypic component of $\Jbf^{\mathrm{al}}$, the algebra $\End^{0} ( X_{1} )$ contains a subalgebra isomorphic to $D$ (Lemma~\ref{lem: End0Jbf contains D; subsec: Preliminaries; sec: Generalized quaternions groups}). The image of the embedding $\End^{0} ( X_{1} ) \xhookrightarrow{} D \times \M ( 2, K_{n+1} )$ is contained in $D \times \{ 0 \}$. Indeed, the induced morphism $\End^{0} ( X_{1} ) \hookrightarrow D \times \M ( 2, K_{n+1} ) \twoheadrightarrow \M ( 2, K_{n+1} )$ cannot be injective, since $\dim_{\Q} ( D ) = \dim_{\Q} \bigl( \M ( 2, K_{n+1} ) \bigr)$ and $D$ is not isomorphic to $\M ( 2, K_{n+1} )$ (Remark~\ref{rem: D is a division algebra; subsec: Setup; sec: Generalized quaternions groups}). Hence, the induced morphism is trivial, since $\End^{0} ( X_{1} )$ is a simple algebra. This yields a contradiction, as the dimension of the identity component of the kernel of an idempotent in $\End^{0} ( \Jbf^{\mathrm{al}} )$ is preserved under specialization (Lemma~\ref{lem: End0 embeds and torsion bijects under specialization; subsec: specialization methods; sec: Preliminaries}). If $m = 1$, we argue in the same way.

                It follows that $m = 0$, hence $X_{1}$ specializes to $\Jac ( \Ccal_{1} )$ and $X_{2} \times \cdots \times X_{A}$ specializes to $\Jac ( \Ccal_{2} )^{2}$ in $\Jcal_{1}$. By applying Proposition~\ref{pro: particular specialization; subsec: Specialization methods; sec: Generalized quaternions groups} with $( a_{1}, \dots, a_{L+1} ) \in U ( \Qbar )$ and $S = \{ 1 \}$, we deduce the existence of an abelian variety $\Jcal_{2}$, of the form $\Jfrak_{T', t}$ for some admissible $T'$ and $t \in \Abb_{\Qbar}^{1} ( \Qbar ) \supset T' ( \Qbar )$, isomorphic to $\Jac ( \Ccal_{3} ) \times \Jac ( \Ccal_{4} )^{2}$. As before, we show that $X_{1}$ specializes to $\Jac ( \Ccal_{3} )$ and $X_{2} \times \cdots \times X_{A}$ specializes to $\Jac ( \Ccal_{4} )^{2}$ in $\Jcal_{2}$. This contradicts Corollary~\ref{cor: exclusion cases with torsion arguments; subsec: Specialization methods; sec: Generalized quaternions groups}.

                It follows that $X_{1} = \Jbf^{\mathrm{al}}$, hence $\Jbf^{\mathrm{al}}$ is simple. Note that $\Jac ( \Ccal_{1} )$ and $\Jac ( \Ccal_{2} )^{2}$ are two distinct isotypic components of $\Jcal_{1}$. Remark~\ref{rem: simple observation of lem: End0 embeds and torsion bijects under specialization; subsec: specialization methods; sec: Preliminaries} shows that $\End^{0} ( \Jbf^{\mathrm{al}} )$ embeds into $\End^{0} \bigl( \Jac ( \Ccal_{1} ) \bigr) \cong D$. On the other hand, we proved in Lemma~\ref{lem: End0Jbf contains D; subsec: Preliminaries; sec: Generalized quaternions groups} that $\End^{0} ( \Jbf^{\mathrm{al}} )$ contains a subalgebra isomorphic to $D$. We deduce that $\End^{0} ( \Jbf^{\mathrm{al}} )$ is isomorphic to $D$. Theorems~\ref{the: reduction first statement main the; subsec: Preliminaries; sec: Generalized quaternions groups} and~\ref{the: there are infinitely many classes; subsec: Preliminaries; sec: Generalized quaternions groups} imply that Theorem~\ref{the: main; sec: Generalized quaternions groups} holds for $L + 1$, as desired.
            \end{proof}



\printbibliography


\end{document}